\pdfoutput=1
\documentclass[11pt,a4paper,reqno]{amsart}

\usepackage[british]{babel}

\usepackage[DIV=9,oneside,BCOR=0mm]{typearea}
\usepackage{microtype}
\usepackage[T1]{fontenc}
\usepackage{setspace}
\usepackage{amssymb}
\usepackage{amsmath}
\usepackage{amsthm}
\usepackage{enumitem}
\usepackage{mathrsfs}
\usepackage[colorlinks,citecolor=blue,urlcolor=black,linkcolor=blue,linktocpage]{hyperref}
\usepackage{mathrsfs}
\usepackage{mathtools}
\usepackage{xcolor}
\usepackage{xfrac}
\usepackage{multicol}
\usepackage{ifsym}
\usepackage{comment}
\usepackage{lmodern}

\newtheorem{innercustomgeneric}{\customgenericname}
\providecommand{\customgenericname}{}
\newcommand{\newcustomtheorem}[2]{%
	\newenvironment{#1}[1]
	{%
		\renewcommand\customgenericname{#2}%
		\renewcommand\theinnercustomgeneric{##1}%
		\innercustomgeneric
	}
	{\endinnercustomgeneric}
}

\newcustomtheorem{customthm}{Theorem}

\newtheorem{thm}{Theorem}[section]

\newtheorem{cor}[thm]{Corollary}
\newtheorem{lem}[thm]{Lemma}
\newtheorem{prop}[thm]{Proposition}

\theoremstyle{definition}

\newtheorem{definition}[thm]{Definition}
\newtheorem{remark}[thm]{Remark}

\renewcommand{\epsilon}{\varepsilon}
\renewcommand{\phi}{\varphi}
\newcommand{\defeq}{\mathrel{\mathop:}=}
\newcommand{\eqdef}{\mathrel{\mathopen={\mathclose:}}}

\renewcommand{\Re}{\operatorname{Re}}
\renewcommand{\Im}{\operatorname{Im}}

\DeclareMathOperator{\Cstar}{C^{\ast}}

\DeclareMathOperator{\id}{id}

\DeclareMathOperator{\C}{\mathbb{C}}
\DeclareMathOperator{\R}{\mathbb{R}}

\DeclareMathOperator{\N}{\mathbb{N}}

\DeclareMathOperator{\Pfin}{\mathcal{P}_{fin}}
\DeclareMathOperator{\Pow}{\mathcal{P}}
\DeclareMathOperator{\dist}{dist}
\DeclareMathOperator{\diam}{diam}
\DeclareMathOperator{\rk}{rk}

\DeclareMathOperator{\Lip}{Lip}
\DeclareMathOperator{\UCB}{UCB}
\DeclareMathOperator{\U}{U}
\DeclareMathOperator{\B}{B}
\DeclareMathOperator{\HS}{HS}
\DeclareMathOperator{\tr}{tr}
\DeclareMathOperator{\Pro}{P}
\DeclareMathOperator{\Sym}{Sym}
\DeclareMathOperator{\Iso}{Iso}
\DeclareMathOperator{\Sph}{\mathbb{S}}
\DeclareMathOperator{\Borel}{\mathcal{B}}
\DeclareMathOperator{\Mat}{M}

\makeatletter
\def\moverlay{\mathpalette\mov@rlay}
\def\mov@rlay#1#2{\leavevmode\vtop{%
		\baselineskip\z@skip \lineskiplimit-\maxdimen
		\ialign{\hfil$\m@th#1##$\hfiincr#2\crcr}}}
\newcommand{\charfusion}[3][\mathord]{
	#1{\ifx#1\mathop\vphantom{#2}\fi
		\mathpalette\mov@rlay{#2\cr#3}
	}
	\ifx#1\mathop\expandafter\displaylimits\fi}

\def\smallunderbrace#1{\mathop{\vtop{\m@th\ialign{##\crcr
				$\hfil\displaystyle{#1}\hfil$\crcr
				\noalign{\kern3\p@\nointerlineskip}%
				\tiny\upbracefill\crcr\noalign{\kern3\p@}}}}\limits}
				
\makeatother

\begin{document}

\author{Paula Kahl}
\address{P.K., Institute of Discrete Mathematics and Algebra, TU Bergakademie Freiberg, 09596 Freiberg, Germany}
\email{paula.kahl@math.tu-freiberg.de}
\author{Friedrich Martin Schneider}
\address{F.M.S., Institute of Discrete Mathematics and Algebra, TU Bergakademie Freiberg, 09596 Freiberg, Germany}
\email{martin.schneider@math.tu-freiberg.de}
	
\title[Hyperlinearity via amenable near representations]{Hyperlinearity via amenable near representations}
\date{\today}
	
\begin{abstract} We note a characterization of the amenability of unitary representations (in the sense of Bekka) via the existence of an orthonormal basis supporting an invariant probability charge. Based on this, we explore several natural notions of amenable near representations on Hilbert spaces. Establishing an analogue of a theorem about sofic groups by Elek and Szabó, we prove that a group is hyperlinear if and only if it admits an essentially free amenable near representation. This answers a question raised by Pestov and Kwiatkowska. For comparison, we also provide characterizations of Kirchberg's factorization property, as well as Haagerup's property, along similar lines. \end{abstract}
	
\subjclass[2020]{22D10, 20P05, 47A20, 22D55}
	
\keywords{Hyperlinear group, amenability, unitary representation, concentration of measure, hypertrace, Haagerup property, Kirchberg's factorization property}
	
\maketitle
	

\tableofcontents

\newpage

\section{Introduction}

The study of sofic and hyperlinear groups constitutes a major theme in infinite group theory (see~\cite{Pestov08,PestovKwiatkowska,CapraroLupini} for excellent introductions). While the former originates in the works of Gromov~\cite{Gromov} and Weiss~\cite{Weiss} on Gottschalk's surjunctivity conjecture, the latter term was coined by R\u{a}dulescu~\cite{Radulescu} in the context of Connes' embedding problem~\cite[p.~105]{Connes}. The present manuscript establishes new connections between hyperlinearity of groups and amenability of unitary representations.

The starting point of our work is a result of Elek and Szabó~\cite{ElekSzabo} characterizing sofic groups by means of \emph{amenable near actions}. Recall that an action of a group $G$ on a set $X$ is \emph{amenable} if there exists a probability charge\footnote{A \emph{probability charge}~\cite[Definition~2.1.1(7), p.~35]{RaoRao} on a Boolean algebra $\mathcal{B}$ is a map $\mu \colon \mathcal{B} \to [0,1]$ such that $\mu(1) = 1$ and $\mu(A \vee B) = \mu(A)+\mu(B)$ for all $A,B \in \mathcal{B}$ with $A \wedge B = 0$.} $\mu$ on the power set $\Pow(X)$ such that $\mu(gB) = \mu(B)$ for all $g \in G$ and $B \subseteq X$. Now, given a set $X$ and a probability charge $\mu$ on $\Pow(X)$, a \emph{$\mu$-near action}~\cite{ElekSzabo,PestovKwiatkowska} of a group $G$ on $X$ is a map $\pi \colon G \to \Sym(X)$ such that $\mu$ is $\pi(G)$-invariant and, for all $g,h \in G$, \begin{displaymath}
	\mu(\{ x \in X \mid \pi(gh)x = \pi(g)\pi(h)x \}) \, = \, 1 .
\end{displaymath} By work of Elek and Szabó~\cite{ElekSzabo}, a group $G$ is sofic if and only if there exist a set $X$, a probability charge $\mu$ on $\Pow(X)$ and $\mu$-near action $\pi \colon G \to \Sym(X)$ such that, for every $g \in G \setminus \{ e \}$, \begin{displaymath}
	\mu(\{ x \in X \mid x \ne \pi(g)x \}) \, = \, 1.
\end{displaymath} This characterization elegantly parallels the fact that a group is amenable if and only if it admits a free amenable action on some set. In~\cite{PestovKwiatkowska}, Pestov and Kwiatkowska raised the question about an analogue of the Elek--Szabó criterion for hyperlinearity. The present paper resolves this problem.

Seeking a hyperlinear sibling to the Elek--Szabó theorem, we are naturally lead to study amenable unitary representations on Hilbert spaces. This notion was introduced by Bekka~\cite{bekka}: a unitary representation $\pi \colon G \to \U(H)$ of a group $G$ on a Hilbert space $H$ is called \emph{amenable} if the algebra $\B(H)$ of bounded linear operators on $H$ admits a state $\phi$ which is $\pi$-invariant, in the sense that $\phi(\pi(g)^{\ast}a\pi(g)) = \phi(a)$ for all $a \in \B(H)$ and $g \in G$. Linking Bekka's characterization of amenability in terms of the existence of approximately invariant finite-rank projections (Lemma~\ref{lemma:folner}) with concentration of measure in high-dimensional Euclidean spheres, Pestov proved that a unitary representation on a Hilbert space $H$ is amenable if and only if the induced action on the unit sphere $\Sph(H)$ admits an invariant state on the space $\UCB(\Sph(H))$ of uniformly continuous bounded complex-valued functions~\cite[Theorem~7.6]{Pestov00}. We revisit Pestov's work and extract from it the tools to prove a new characterization of amenability of unitary representations in terms of invariant probability charges on orthonormal bases. To be more precise, let us recall that, if $X$ is an orthonormal basis of a Hilbert space $H$ and $\mu$ is a probability charge on $\Pow(X)$, then \begin{displaymath}
	\mu_{\bullet} \colon \, \B(H) \, \longrightarrow \, \C, \quad a \, \longmapsto \, \int \langle x,ax \rangle \, \mathrm{d}\mu(x)
\end{displaymath} constitutes a state on $\B(H)$ (see Remark~\ref{remark:charges.induce.states}). The following is a consequence of the results of~\cite{bekka} and~\cite{Pestov00}.

\begin{customthm}{A}[Theorem~\ref{theorem:amenable.basis}]\label{theorem:a} A continuous\footnote{with respect to the strong operator topology on $\U(H)$} unitary representation $\pi$ of a $\sigma$-compact locally compact group $G$ on a Hilbert space $H$ is amenable if and only if there exist an orthonormal basis $X$ for $H$ and a probability charge $\mu$ on $\Pow(X)$ such that $\mu_{\bullet}$ is $\pi$-invariant. \end{customthm}

Based on Theorem~\ref{theorem:a}, we propose an analogue of amenable near actions, namely \emph{amenable near representations}. Given an orthonormal basis $X$ of a Hilbert space $H$ and a probability charge $\mu$ on $\Pow(X)$, we define a \emph{$\mu$-near representation} of a group $G$ on $H$ to be a map $\pi \colon G \to \U(H)$ such that $\mu_{\bullet}$ is $\pi(G)$-invariant and, for all $g,h \in G$ and $\epsilon \in \R_{>0}$, \begin{displaymath}
	\mu (\{ x \in X \mid \Vert \pi(gh)x - \pi(g)\pi(h)x \Vert \leq \epsilon \}) \, = \, 1
\end{displaymath} (see Definition~\ref{definition:near.representation}). Using this terminology, we establish the following result.

\begin{customthm}{B}[cf.~Theorem~\ref{theorem:elek.szabo}]\label{theorem:b} A group $G$ is hyperlinear if and only if there exist a set~$X$, a probability charge $\mu$ on $\Pow(X)$ and a $\mu$-near representation $\pi \colon G \to \U(\ell^{2}(X))$ such that, for all $g \in G\setminus \{ e \}$ and $\epsilon \in \R_{>0}$, \begin{displaymath}
	\mu (\{ x \in X \mid \lvert \langle x,\pi(g)x \rangle \rvert \leq \epsilon \}) \, = \, 1.
\end{displaymath} \end{customthm}

The proof of Theorem~\ref{theorem:b} is rather different from the strategy pursued by Elek and Szabó~\cite{ElekSzabo} characterizing sofic groups. While the latter involves paradoxical decompositions, the former makes use of concentration of measure in finite-dimensional spheres, subtle approximation techniques due to Kirchberg~\cite{Kirchberg} and Ozawa~\cite{Ozawa}, and Stinespring's dilation theorem~\cite{Stinespring}. Along the way, we develop additional notions of amenable near representations on a Hilbert space $H$ (Definition~\ref{definition:near.representation.variations}), namely relative to probability charges on the Borel $\sigma$-algebra $\Borel(\Sph(H))$, states on $\UCB(\Sph(H))$, and states on $\B(H)$. This study results in the following theorem, which serves as a major tool in the proof of Theorem~\ref{theorem:b}, but also appears to be of interest in its own right.

\begin{customthm}{C}[cf.~Theorem~\ref{theorem:amenable.trace}]\label{theorem:c} Let $G$ be a group. The following are equivalent. \begin{enumerate}
	\item\label{theorem:c.1}  $G$ is hyperlinear.
	\item\label{theorem:c.2}  There exist a Hilbert space $H$, a probability charge $\nu$ on $\Borel(\Sph(H))$ and a $\nu$-near representation $\pi \colon G \to \U(H)$ such that, for all $g \in G\setminus \{ e \}$ and $\epsilon \in \R_{>0}$, \begin{displaymath}
			\qquad \nu(\{ x \in \Sph(H) \mid \lvert \langle x,\pi(g)x \rangle \rvert \leq \epsilon \}) \, = \, 1 .
		\end{displaymath}
	\item\label{theorem:c.3}  There exist a Hilbert space $H$, a state $\mu$ on $\UCB(\Sph(H))$ and a $\mu$-near representation $\pi \colon G \to \U(H)$ such that \begin{displaymath}
			\qquad \forall g \in G\setminus \{ e\} \colon \quad  \mu(x \mapsto \vert\langle x,\pi(g)x \rangle\vert) = 0 .
		\end{displaymath}
	\item\label{theorem:c.4}  There exist a Hilbert space $H$, a state $\phi$ on $\B(H)$ and a $\phi$-near representation $\pi \colon G \to \U(H)$ such that \begin{displaymath}
			\qquad \forall g \in G\setminus \{ e\} \colon \quad \phi(\pi(g)) = 0 .
		\end{displaymath}
\end{enumerate} \end{customthm}

If $H$ is a Hilbert space, $\phi$ is a state on $\B(H)$ and $\pi$ is a $\phi$-near representation of a group $G$ on $H$, then $\phi\circ \pi$ is a character on $G$ (Proposition~\ref{proposition:near.rep}\ref{proposition:near.rep.character})---and the freeness condition in~\ref{theorem:c.4} of Theorem~\ref{theorem:c} means precisely that this character is the regular character on $G$. One may wonder which groups even admit a genuine unitary representation with such a free invariant state. We prove that this happens precisely for groups possessing \emph{Kirchberg's factorization property}~\cite{KirchbergInventiones}, thus drawing a connection between results by Elek and Szabó~\cite{ElekSzabo} and earlier work of Kirchberg~\cite{Kirchberg}, in particular concerning Kazhdan groups (Theorem~\ref{theorem:kazhdan}).

\begin{customthm}{D}[cf.~Theorem~\ref{theorem:kirchberg}]\label{theorem:d} Let $G$ be a group. The following are equivalent. \begin{enumerate}
	\item\label{theorem:d.1} $G$ has Kirchberg's factorization property.
	\item\label{theorem:d.2} There exist a unitary representation $\pi$ of $G$ on a Hilbert space $H$ and a $\pi$-invariant state $\mu$ on $\UCB(\Sph(H))$ such that \begin{displaymath}
			\qquad \forall g \in G\setminus \{ e\} \colon \quad  \mu(x \mapsto \vert\langle x,\pi(g)x \rangle\vert) = 0 .
		\end{displaymath}	
	\item\label{theorem:d.3} There exist a unitary representation $\pi$ of $G$ on a Hilbert space $H$ and a $\pi$-invariant state $\phi$ on $\B(H)$ such that \begin{displaymath}
			\qquad \forall g \in G \setminus \{ e \} \colon \quad \phi(\pi(g)) = 0 .
		\end{displaymath}
	\item\label{theorem:d.4} There exist a set $X$, a unitary representation $\pi \colon G \to \U(\ell^{2}(X))$ and a probability charge $\mu$ on $\Pow(X)$ such that $\mu_{\bullet}$ is $\pi$-invariant and \begin{displaymath}
			\qquad \forall g \in G\setminus \{ e\} \ \forall \epsilon \in \R_{>0} \colon \ \ \mu(\{ x \in X \mid \lvert \langle x,\pi(g)x \rangle \rvert \leq \epsilon \}) = 1.
		\end{displaymath}
\end{enumerate}\end{customthm}

By force of the concentration of measure in finite-dimensional Euclidean spheres, in condition~\ref{theorem:c.3} of Theorem~\ref{theorem:c} one may additionally arrange for $\mu$ to be a ring homomorphism from $\UCB(\Sph(H))$ to $\C$ (see Remark~\ref{remark:ring.homomorphism}). The same applies to condition~\ref{theorem:d.2} of Theorem~\ref{theorem:d} (see Remark~\ref{remark:ring.homomorphism.kirchberg}).

Our results provide a new perspective on the comparison between hyperlinearity and \emph{Haagerup's property} (which is also known as \emph{a-T-menability}). It is an open problem~\cite[Open question~9.5]{Pestov08} whether the Haagerup property would imply hyperlinearity. Based on Theorem~\ref{theorem:a}, we establish a characterization of countable groups with the Haagerup property (Corollary~\ref{corollary:haagerup}) in the spirit of Theorem~\ref{theorem:b}.

This article is organized as follows. Our Section~\ref{section:amenable.bases} is dedicated to the proof of Theorem~\ref{theorem:a}. Subsequently, we turn to spherical measures and their concentration properties in Section~\ref{section:spherical.measures} and use these phenomena to establish some useful characterizations of hyperlinearity in Section~\ref{section:hyperlinear.groups}. We proceed to developing several natural notions of amenable near representations in the subsequent Section~\ref{section:near.representations}, which also contains the proof of Theorem~\ref{theorem:c}. These findings are then used to establish Theorem~\ref{theorem:b}, the direct analogue of the Elek--Szabó theorem, in Section~\ref{section:elek.szabo}. Our final Section~\ref{section:kirchberg} discusses characterizations of Kirchberg's factorization property in terms of essentially free amenable unitary representations, establishing Theorem~\ref{theorem:d} along with some additional ramifications. 

\smallskip

\textbf{Disclaimer.} In this manuscript, to avoid repetition, the term \emph{Hilbert space} is reserved for complete inner product spaces over $\C$, and inner products are supposed to be linear in the \emph{second} argument.

\section{Amenable bases}\label{section:amenable.bases}

Recall that a \emph{state} on a unital $\Cstar$-algebra $A$ is a positive unital linear form on $A$. Let $H$ be a Hilbert space. The unitary group $\U(H)$ admits a natural action on the set of states on $\B(H)$, which is given by \begin{displaymath}
	u.\phi \colon \, \B(H) \, \longrightarrow \, \C, \quad a \, \longmapsto \, \phi(u^{\ast}au)
\end{displaymath} for every $u \in \U(H)$ and every state $\phi$ on $\B(H)$. A representation $\pi \colon G \to \U(H)$ of a group $G$ on $H$ is said to be \emph{amenable}~\cite{bekka} if there exists a state $\phi \colon \B(H) \to \C$ such that $\pi(g).\phi = \phi$ for all $g \in G$.
As is customary, \begin{displaymath}
	\Pro(H) \, \defeq \, \!\left\{ p \in \B(H) \left\vert \, p^{2} = p = p^{\ast} \right\} .\right.
\end{displaymath}

The following two results are the central ingredients in the proof of Theorem~\ref{theorem:amenable.basis}.

\begin{lem}[{\cite[Theorem~6.2]{bekka}}] \label{lemma:folner} Let $\pi$ be a continuous unitary representation of a locally compact group $G$ on a Hilbert space $H$. Then $\pi$ is amenable if and only if, for every $\epsilon \in \R_{>0}$ and every compact subset $E \subseteq G$, there exists $p \in \Pro(H)$ such that $0 < \rk(p) < \infty$ and \begin{displaymath}
	\forall g \in E \colon \qquad \Vert \pi(g)p\pi(g)^{\ast}-p\Vert_{1} \, \leq \, \epsilon\Vert p\Vert_{1} .
\end{displaymath} \end{lem}

\begin{lem}[{\cite[Lemma~6.3]{Pestov00}}]\label{lemma:pestov} Let $\pi$ be a continuous unitary representation of a locally compact group $G$ on a Hilbert space $H$. If $\pi$ is not of the form $\pi_{1}\oplus\pi_{2}$, where $\pi_{1}$ is finite-dimensional and $\pi_{2}$ is non-amenable, then for every compact subset $E \subseteq G$ and all $\epsilon,\ell \in \R_{>0}$, there exists an orthogonal projection $p \in \Pro(H)$ such that $\ell \leq \rk(p) < \infty$ and \begin{displaymath}
	\forall g \in E \colon \qquad \Vert \pi(g)p\pi(g)^{\ast}-p\Vert_{1} \, \leq \, \epsilon\Vert p\Vert_{1} .
\end{displaymath} \end{lem}

For the reader's convenience, we provide a detailed proof of Lemma~\ref{lemma:pestov} in the Appendix~\ref{section:proof.pestov.lemma}.

\begin{remark}\label{remark:charges.induce.states} If $X$ is an orthonormal basis of a Hilbert space $H$ and $\mu$ is a probability charge on $\Pow(X)$, then \begin{displaymath}
	\mu_{\bullet} \colon \, \B(H) \, \longrightarrow \, \C, \quad a \, \longmapsto \, \int \langle x,ax \rangle \, \mathrm{d}\mu(x)
\end{displaymath} constitutes a state on $\B(H)$. We refer to \cite[363L]{Fremlin3} for details about integration with respect to a finitely additive functional. \end{remark}

\begin{thm}\label{theorem:amenable.basis} A strongly continuous unitary representation $\pi$ of a $\sigma$-compact locally compact group $G$ on a Hilbert space $H$ is amenable if and only if there exist an orthonormal basis $X$ for $H$ and a probability charge $\mu$ on $\Pow(X)$ such that $\mu_{\bullet}$ is $\pi $-invariant.
\end{thm}

\begin{proof} ($\Longleftarrow$) This is due to Remark~\ref{remark:charges.induce.states}.

($\Longrightarrow$) Below we will use the fact that, for every $a\in\B(H)$ and every trace-class operator $b \in \B_{1}(H)$, \begin{equation} \label{eq:1norm}
	\Vert ab \Vert_{1},\, \Vert ba \Vert_{1},\, \vert \tr(ab) \vert \,\leq\, \Vert a \Vert \Vert b \Vert_{1} 
\end{equation} (cf.~\cite[Theorem~18.11, p.~89]{ConwayBook}).

In order to prove the desired implication, suppose first that $\pi$ has a non-trivial finite-dimensional subrepresentation, i.e., $H$ admits a $\pi$-invariant, non-zero, finite-dimensional subspace $H_{0}$. In this case, we fix some orthonormal basis $X_{0}$ for $H_{0}$, choose any orthonormal basis $X$ for $H$ with $X_{0} \subseteq X$, and consider the probability charge \begin{displaymath}
	\mu \colon \Pow(X) \, \longrightarrow \, [0,1], \quad A \, \longmapsto \, \tfrac{\vert A \cap {X_{0}} \vert}{\vert X_{0}\vert} .
\end{displaymath} Letting $p \in \Pro(H)$ denote the orthogonal projection onto $H_{0}$, we observe that \begin{align*}
	\mu_{\bullet}(a) \, &= \, \int \langle x,ax \rangle \, \mathrm{d}\mu(x) \, = \, \tfrac{1}{\vert X_{0} \vert} \sum\nolimits_{x \in X_{0}} \langle x,ax \rangle \\
	& = \, \tfrac{1}{\vert X_{0} \vert} \sum\nolimits_{x \in X} \langle x,apx \rangle \, = \, \tfrac{1}{\dim(H_{0})} \tr(ap) 
\end{align*} for all $a \in \B(H)$. Now, if $g \in G$, then $\pi$-invariance of $H_{0}$ implies that $\pi(g)p\pi(g)^{\ast} = p$, whence \begin{align*}
	(\pi(g).(\mu_{\bullet}))(a) \, &= \, \mu_{\bullet}(\pi(g)^{\ast}a\pi(g)) \, = \, \tfrac{1}{\dim(H_{0})} \tr(\pi(g)^{\ast}a\pi(g)p) \\
	& = \, \tfrac{1}{\dim(H_{0})} \tr(a\pi(g)p\pi(g)^{\ast}) \, = \, \tfrac{1}{\dim(H_{0})} \tr(ap) \, = \, \mu_{\bullet}(a)
\end{align*} for every $a \in \B(H)$, that is, $\pi(g).(\mu_{\bullet}) = \mu_{\bullet}$. This completes the argument in the considered case.

For the remainder of the proof, let us suppose that $\pi$ does not have a non-trivial finite-dimensional subrepresentation. In particular, as amenability of $\pi$ necessitates that $H$ is non-zero, $H$ has to be infinite-dimensional. Since $G$ is $\sigma$-compact, we find an ascending chain $(E_{n})_{n \in \N}$ of compact subsets of $G$ such that $G = \bigcup_{n\in \N} E_{n}$. Recursively, we are going to choose a sequence $(p_{n})_{n \in \N} \in \Pro(H)^{\N}$ of mutually orthogonal non-zero finite-rank projections such that \begin{displaymath}
	\forall n \in \N \ \forall g \in E_{n} \colon \qquad \Vert \pi(g)p_{n}\pi(g)^{\ast}-p_{n}\Vert_{1} \, \leq \, \tfrac{1}{n+1}\Vert p_{n}\Vert_{1} .
\end{displaymath} For a start, by Lemma~\ref{lemma:folner}, find $p_{0} \in \Pro(H)$ such that $0 < \rk(p_{0}) < \infty$ and \begin{displaymath}
	\forall g \in E_{0}\colon \qquad \lVert \pi(g)p_{0}\pi(g)^{\ast}-p_{0}\rVert_{1} \leq \lVert p_{0}\rVert_{1}.
\end{displaymath} For the recursive step, let $n \in \N_{>0}$ and suppose that $p_{0},\ldots,p_{n-1} \in \Pro(H)\setminus \{ 0 \}$ are mutually orthogonal, of finite rank, and such that \begin{displaymath}
	\forall i \in \{0,\ldots,n-1\} \ \forall g \in E_{i}\colon \qquad \Vert \pi(g)p_{i}\pi(g)^{\ast}-p_{i}\Vert_{1} \, \leq \, \tfrac{1}{i+1}\Vert p_{i}\Vert_{1} .
\end{displaymath} Note that $p \defeq \sum_{i=0}^{n-1} p_{i} \in \Pro(H)$ and $\rk(p) = \sum_{i=0}^{n-1} \rk(p_{i})$. By Lemma~\ref{lemma:pestov}, there exists $q \in \Pro(H)$ such that \begin{displaymath}
	24(n+1)\rk(p) \, \leq \, \rk(q) \, < \, \infty
\end{displaymath} and $\Vert \pi(g)q\pi(g)^{\ast}-q\Vert_{1}\leq \tfrac{1}{2(n+1)}\lVert q \rVert_{1}$ for every $g\in E_{n}$. Since $H$ is of infinite dimension, we find some element\footnote{In fact, one may easily arrange for $u$ to be an involution.} $u \in \U(H)$ such that $p \perp upu^{\ast} \perp q$. Consider \begin{displaymath}
	p_{n} \, \defeq \, (1-p+up)q(1-p+pu^{\ast}) \, \in \, \B(H) 
\end{displaymath} and note that $p_{n}^{\ast} = p_{n}$. As $pup = pupu^{\ast}u = 0$ and $pu^{\ast}p = u^{\ast}upu^{\ast}p = 0$, we see that \begin{equation}\label{eq:vanishing}
	(p-pu^{\ast})(p-up) \, = \, p-pup-pu^{\ast}p+\smallunderbrace{pu^{\ast}up}_{=\,pp \, = \,p} \, = \, 2p.
\end{equation} Thus, \begin{align*}
	(1-p+pu^{\ast})(1-p+up) \, &= \, 1-p+up-p+pu^{\ast}+(p-pu^{\ast})(p-up) \\
	&\stackrel{\eqref{eq:vanishing}}{=} \, 1+up+pu^{\ast}
\end{align*} and therefore \begin{align*}
	q(1-p+pu^{\ast})(1-p+up)q \, &= \, q(1+up+pu^{\ast})q \, = \, q+qupq+qpu^{\ast}q \\
	&= \, q+\smallunderbrace{qupu^{\ast}}_{=\,0}uq+qu^{\ast}\smallunderbrace{upu^{\ast}q}_{=\,0} \, = \, q ,
\end{align*} from which we infer that $p_{n}p_{n} = p_{n}$, i.e., $p_{n} \in \Pro(H)$, and moreover \begin{equation}\label{eq:rank}
	\rk(p_{n}) \, = \, \tr(p_{n}) \, = \, \tr (q(1-p+pu^{\ast})(1-p+up)q) \, = \, \tr(q) \, = \, \rk(q) .
\end{equation} Finally, since \begin{align*}
	\lVert (p-up)q(p-pu^{\ast}) \rVert_{1} \, &= \, \tr((p-up)q(p-pu^{\ast})) \\
		& = \, \tr(q(p-pu^{\ast})(p-up)) \, \stackrel{\eqref{eq:vanishing}}{=} \, 2\tr(qp)
\end{align*} and so \begin{align*}
	\Vert p_{n}-q\Vert_{1} \, &= \, \lVert qpu^{\ast}-qp+upq-pq + (p-up)q(p-pu^{\ast}) \rVert_{1}\\
		&\leq \, \lVert qpu^{\ast} \rVert_{1} + \lVert qp \rVert_{1} + \lVert uqp \rVert_{1} + \lVert pq \rVert_{1} + \underbrace{\lVert (p-up)q(p-pu^{\ast}) \rVert_{1}}_{=\,2\tr(qp)} \\
		& \stackrel{\eqref{eq:1norm}}{\leq} \, 6 \lVert p \rVert_{1} \, = \, 6\tr(p) \, = \, 6\rk(p) \, \leq \, \tfrac{6}{24(n+1)}\rk(q) \, \stackrel{\eqref{eq:rank}}{=} \, \tfrac{1}{4(n+1)}\rk(p_{n}),
\end{align*} we conclude that \begin{align*}
	&\Vert \pi(g)p_{n}\pi(g)^{\ast}-p_{n}\Vert_{1} \\
	&\qquad \leq \, \Vert \pi(g)p_{n}\pi(g)^{\ast}-\pi(g)q\pi(g)^{\ast}\Vert_{1} + \Vert \pi(g)q\pi(g)^{\ast}-q\Vert_{1} + \Vert q-p_{n}\Vert_{1} \\
	&\qquad = \, \Vert \pi(g)(p_{n}-q)\pi(g)^{\ast}\Vert_{1} + \Vert \pi(g)q\pi(g)^{\ast}-q\Vert_{1} + \Vert q-p_{n}\Vert_{1} \\
	&\qquad = \, 2\Vert p_{n}-q\Vert_{1} +\Vert \pi(g)q\pi(g)^{\ast}-q\Vert_{1} \\
	&\qquad \leq \, \tfrac{1}{2(n+1)}\rk(p_{n}) + \tfrac{1}{2(n+1)}\lVert q \rVert_{1} \, \leq \, \tfrac{1}{n+1}\rk(p_{n})
\end{align*} for every $g\in E_{n}$, which completes our recursion.

Now, for each $n \in \N$, pick an orthonormal basis $X_{n}$ for $p_{n}(H)$. Then $\bigcup_{n \in \N} X_{n}$ is an orthonormal system in $H$, therefore we find an orthonormal basis $X$ for $H$ such that $\bigcup_{n \in \N} X_{n} \subseteq X$. Choose any non-principal ultrafilter $\mathcal{F}$ on $\N$ and consider the probability charge \begin{displaymath}
	\mu \colon \, \Pow(X) \, \longrightarrow \, [0,1], \quad A \, \longmapsto \, \lim\nolimits_{n \to \mathcal{F}} \tfrac{\vert A\cap X_{n}\vert}{\vert X_{n}\vert}.
\end{displaymath} We observe that \begin{displaymath}
	\mu_{\bullet}(a) \, = \, \int \langle x,ax \rangle \, \mathrm{d}\mu(x) \, = \, \lim\nolimits_{n \to \mathcal{F}} \tfrac{1}{\vert X_{n}\vert}\sum\nolimits_{x\in X_{n}}\langle x,ax\rangle \, = \, \lim\nolimits_{n \to \mathcal{F}} \tfrac{1}{\tr(p_{n})}\tr(ap_{n})
\end{displaymath} for every $a \in \B(H)$. Finally, if $g \in G$, then there exists $m \in \N$ with $g \in E_{m}$, whence \begin{displaymath}
	\left\{ n \in \N \left\vert \, \Vert \pi(g)p_{n}\pi(g)^{\ast}-p_{n}\Vert_{1}\leq \tfrac{1}{n+1}\Vert p_{n}\Vert_{1}\right\}\right.\! \supseteq \{ n \in \N \mid g \in E_{n}\} = \{ n\in \N \mid n\geq m\}
\end{displaymath} belongs to $\mathcal{F}$, and we conclude that \begin{allowdisplaybreaks} \begin{align*}
	&\vert (\pi(g).(\mu_{\bullet}))(a) - \mu_{\bullet}(a)\vert \, = \, \vert \mu_{\bullet}(\pi(g)^{\ast}a\pi(g)) - \mu_{\bullet}(a)\vert \\
	&\qquad \qquad = \, \left\lvert \lim\nolimits_{n\to \mathcal{F}}\tfrac{1}{\tr(p_{n})}\tr(\pi(g)^{\ast}a\pi(g)p_{n}) -\lim\nolimits_{n \to \mathcal{F}}\tfrac{1}{\tr(p_{n})}\tr(ap_{n}) \right\rvert \\
	&\qquad \qquad \leq \, \lim\nolimits_{n\to \mathcal{F}} \tfrac{1}{\tr(p_{n})}\vert \tr(\pi(g)^{\ast}a\pi(g)p_{n})-\tr(ap_{n})\vert \\
	&\qquad \qquad = \, \lim\nolimits_{n\to \mathcal{F}} \tfrac{1}{\tr(p_{n})} \vert \tr(a(\pi(g)p_{n}\pi(g)^{\ast}-p_{n}))\vert \\
	&\qquad \qquad \stackrel{\eqref{eq:1norm}}{\leq} \, \lim\nolimits_{n\to \mathcal{F}} \tfrac{1}{\tr(p_{n})} \Vert a \Vert \Vert \pi(g)p_{n}\pi(g)^{\ast}-p_{n}\Vert_{1} \\
	&\qquad \qquad = \, \Vert a \Vert \lim\nolimits_{n\to \mathcal{F}} \tfrac{1}{\Vert p_{n}\Vert_{1}} \Vert \pi(g)p_{n}\pi(g)^{\ast}-p_{n}\Vert_{1} \, = \, 0
\end{align*} for all $a \in \B(H)$, which means that $\pi(g).(\mu_{\bullet}) = \mu_{\bullet}$. This completes the proof. \end{allowdisplaybreaks} \end{proof}

\section{Spherical measures and concentration}\label{section:spherical.measures}

The purpose of this section is to provide some background on concentration of spherical measures relevant for our subsequent study of hyperlinear groups. Along the way, we also record a characterization of the amenability of unitary representations in terms of the existence of non-zero finite-dimensional linear subspaces whose spherical measures are asymptotically invariant with respect to the mass transportation distance (Proposition~\ref{proposition:mass.transportation}).

Given a non-zero finite-dimensional Hilbert space $H$, we let $\nu_{H}$ denote the Haar measure on its unit sphere \begin{displaymath}
	\Sph(H) \, \defeq \, \{ x \in H \mid \Vert x \Vert = 1 \} ,
\end{displaymath} i.e., the unique probability measure on the Borel $\sigma$-algebra $\Borel(\Sph(H))$ invariant under the natural action of $\U(H)$, and we consider the normalized Hilbert--Schmidt norm \begin{displaymath}
	\Vert a \Vert_{\HS} \, = \, \tfrac{1}{\sqrt{\dim(H)}} \Vert a \Vert_{2} \qquad (a \in \B(H)). 
\end{displaymath} The following is well known.

\begin{prop}[see, e.g.,~\cite{kania}]\label{proposition:trace} Let $H$ be a non-zero finite-dimensional Hilbert space. For every $a \in \B(H)$, \begin{displaymath}
	\tfrac{1}{\dim(H)}\tr(a) \, = \, \int \langle x,ax \rangle \, \mathrm{d}\nu_{H}(x) .
\end{displaymath} In particular, for every $a \in \B(H)$, \begin{displaymath}
	\Vert a \Vert_{\HS}^{2} \, = \, \int \Vert ax \Vert^{2} \, \mathrm{d}\nu_{H}(x) .
\end{displaymath} \end{prop}

\begin{proof} For instance, see~\cite{kania} or~\cite[Proposition~3.2]{burtonetal} for the first assertion, which readily entails that, for every $a \in \B(H)$, \begin{displaymath}
	\Vert a \Vert_{\HS}^{2} \, = \, \tfrac{1}{\dim(H)}\tr(a^{\ast}a) \, = \, \int \langle x,a^{\ast}ax \rangle \, \mathrm{d}\nu_{H}(x) \, = \, \int \Vert ax \Vert^{2} \, \mathrm{d}\nu_{H}(x) . \qedhere
\end{displaymath} \end{proof}

\begin{cor}\label{corollary:trace} Let $H$ be a Hilbert space and let $p \in \Pro(H)$ with $0 < \rk(p) < \infty$. For every $a \in \B(H)$, \begin{align*}
	\tfrac{1}{\rk(p)}\tr(ap) \, &= \, \int \langle x,ax \rangle \, \mathrm{d}\nu_{p(H)}(x) , \\
	\tfrac{1}{\rk(p)}\Vert ap \Vert_{2}^{2} \, &= \, \int \Vert ax \Vert^{2} \, \mathrm{d}\nu_{p(H)}(x) .
\end{align*} \end{cor}

\begin{proof} For every $a \in \B(H)$, we see that \begin{align*}
	\tfrac{1}{\rk(p)}\tr(ap) \, &= \, \tfrac{1}{\dim (p(H))}\tr(pap) \\
	& \stackrel{\ref{proposition:trace}}{=} \, \int \langle x,papx \rangle \, \mathrm{d}\nu_{p(H)}(x) \, = \, \int \langle x,ax \rangle \, \mathrm{d}\nu_{p(H)}(x) , \\
	\tfrac{1}{\rk(p)}\Vert ap \Vert_{2}^{2} \, &= \, \tfrac{1}{\dim (p(H))}\tr(pa^{\ast}ap) \\
	& \stackrel{\ref{proposition:trace}}{=} \, \int \langle x,pa^{\ast}apx \rangle \, \mathrm{d}\nu_{p(H)}(x) \, = \, \int \Vert ax \Vert^{2} \, \mathrm{d}\nu_{p(H)}(x) . \qedhere
\end{align*} \end{proof}

Before proceeding to the phenomenon of measure concentration, we point out another characterization of the amenability of unitary representations, which is essentially contained in~\cite{Pestov00} and has been suggested to the authors by Vladimir Pestov. While intended for potential future application, this observation will not be needed in the remainder of the present manuscript.

\begin{lem}\label{lemma:mass.transportation} Let $H$ be a Hilbert space, let $p,q \in \Pro(H)$ with $0 < \rk(p) = \rk(q) < \infty$. Then \begin{displaymath}
	\sup\nolimits_{f \in \Lip_{1}(\Sph(H),[-1,1])} \left\lvert \int f \, \mathrm{d}\nu_{p(H)} - \int f \, \mathrm{d}\nu_{q(H)} \right\rvert \, \leq \, \tfrac{\sqrt{2}}{\Vert p \Vert_{2}}\Vert q-p \Vert_{2} .
\end{displaymath} \end{lem}

\begin{proof} By Lemma~\ref{lemma:low.energy.conjugation}, we find $u \in \U(H)$ with $q = upu^{\ast}$ and $\Vert (1-u)p \Vert_{2} = \sqrt{2}\Vert q - p \Vert_{2}$. It follows that $\nu_{q(H)} = u_{\ast}(\nu_{p(H)})$ and then, due to the Cauchy--Schwartz inequality, \begin{align*}
	\left\lvert \int f\, \mathrm{d}\nu_{p(H)} - \int f\, \mathrm{d}\nu_{q(H)} \right\rvert \, &= \, \left\lvert \int f \, \mathrm{d}\nu_{p(H)} - \int f \circ u \, \mathrm{d}\nu_{p(H)} \right\rvert \\
		& \leq \, \int \vert f(x) - f(ux) \vert \, \mathrm{d}\nu_{p(H)}(x) \\
		& \leq \, \int \Vert x-ux \Vert \, \mathrm{d}\nu_{p(H)}(x) \\
		& \leq \, \sqrt{\int \Vert (1-u)x \Vert^{2} \, \mathrm{d}\nu_{p(H)}(x)} \\
		& \stackrel{\ref{corollary:trace}}{=} \, \tfrac{1}{\Vert p \Vert_{2}} \Vert (1-u)p \Vert_{2} \, = \, \tfrac{\sqrt{2}}{\Vert p \Vert_{2}}\Vert q - p \Vert_{2}
\end{align*} for every $f \in \Lip_{1}(\Sph(H),[-1,1])$, as desired. \end{proof}

\begin{prop}\label{proposition:mass.transportation} Let $\pi$ be a strongly continuous unitary representation of a locally compact group $G$ on a Hilbert space $H$. Then $\pi$ is amenable if and only if, for every compact subset $E \subseteq G$ and every $\epsilon \in \R_{>0}$, there exists a non-zero finite-dimensional linear subspace $H_{0}$ of $H$ such that \begin{displaymath}
    \forall g \in E \colon \qquad \sup\nolimits_{f \in \Lip_{1}(\Sph(H),[-1,1])} \left\lvert \int f \, \mathrm{d}\nu_{H_{0}} - \int f \, \mathrm{d}\nu_{\pi(g)H_{0}} \right\rvert \, \leq \, \epsilon .
\end{displaymath} \end{prop}

\begin{proof} ($\Longrightarrow$) Consider any compact subset $E \subseteq G$ and any $\epsilon \in \R_{>0}$. By Lemma~\ref{lemma:folner}, there exists a non-zero orthogonal projection $p \in \Pro(H)$ of finite rank such that \begin{displaymath}
     \forall g \in E \colon \qquad \Vert \pi(g)p\pi(g)^{\ast} - p \Vert_{1} \, \leq \, \tfrac{\epsilon^{2}}{2} \Vert p \Vert_{1} ,
\end{displaymath} which, by~\cite[Lemma~4.2(ii)]{bekka}, immediately implies that \begin{displaymath}
     \forall g \in E \colon \qquad \Vert \pi(g)p\pi(g)^{\ast} - p \Vert_{2} \, \leq \, \tfrac{\epsilon}{\sqrt{2}} \Vert p\Vert_{2} .
\end{displaymath} Hence, the desired conclusion follows by Lemma~\ref{lemma:mass.transportation}.

($\Longleftarrow$) By assumption, we find a net $(H_{i})_{i \in I}$ of non-zero finite-dimensional linear subspaces of $H$ such that, for every $g \in G$, \begin{equation}\label{eq:asymptotic.invariance}
    \sup\nolimits_{f \in \Lip_{1}(\Sph(H),[-1,1])} \left\lvert \int f \, \mathrm{d}\nu_{H_{i}} - \int f \, \mathrm{d}\nu_{\pi(g)H_{i}} \right\rvert \, \longrightarrow \, 0 \qquad (i \to I).
\end{equation} Pick an ultrafilter $\mathcal{F}$ on $I$ containing $\{ \{ i \in I \mid i_{0} \preceq i \} \mid i_{0} \in I \}$. Consider the state \begin{displaymath}
    \mu \colon \, \UCB(\Sph(H)) \, \longrightarrow \, \C, \quad f \, \longmapsto \, \lim\nolimits_{i\to \mathcal{F}} \int f \, \mathrm{d}\nu_{H_{i}} .
\end{displaymath} If $g \in G$, then $\nu_{\pi(g)H_{i}} = \pi(g)_{\ast}(\nu_{H_{i}})$ for every $i \in I$, thus \begin{align*}
    \mu(f \circ \pi(g)) \, &= \, \lim\nolimits_{i\to \mathcal{F}} \int f \circ \pi(g) \, \mathrm{d}\nu_{H_{i}} \\
        & = \, \lim\nolimits_{i\to \mathcal{F}} \int f \, \mathrm{d}\nu_{\pi(g)H_{i}} \, \stackrel{\eqref{eq:asymptotic.invariance}}{=} \, \lim\nolimits_{i\to \mathcal{F}} \int f \, \mathrm{d}\nu_{H_{i}} \, = \, \mu(f)
\end{align*} for every $f \in \Lip_{1}(\Sph(H),[-1,1])$. As $\Lip_{1}(\Sph(H),[-1,1])$ spans a $\Vert \cdot \Vert_{\infty}$-dense linear subspace of $\UCB(\Sph(H))$ (see, e.g.,~\cite[Lemma~5.20, p.~68]{PachlBook}), thus $\mu(f \circ \pi(g)) = \mu(f)$ for all $g \in G$ and $f \in \UCB(\Sph(H))$. Hence, $\pi$ is amenable by~\cite[Proposition~3.1]{Pestov00}. \end{proof}

One major ingredient in our analysis of hyperlinear groups is the well-known phenomenon of concentration of measure in finite-dimensional spheres, as discovered by L\'evy~\cite{Levy} and used extensively by Milman studying the asymptotic geometry of Banach spaces~\cite{Milman67,Milman71}. We refer to~\cite{MilmanSchechtman,milman,Ledoux,ConvexBodies} for further background.

\begin{thm}[see, e.g.,~{\cite[Theorem~1.7.9]{ConvexBodies}}]\label{theorem:spherical.concentration} Let $H$ be a non-zero finite-dimensional Hilbert space, let $\ell \in \R_{\geq 0}$ and let $f \in \Lip_{\ell}(\Sph(H),\R)$. Then, for every $\epsilon \in \R_{>0}$, \begin{displaymath}
	\nu_{H}\!\left(\left\{ x \in \Sph(H) \left\vert \, \left\lvert f(x) - \int f \, \mathrm{d}\nu_{H} \right\rvert > \epsilon \right\} \right)\!\right. \, \leq \, 2\exp\!\left(-\tfrac{\epsilon^{2}(2\dim(H)-1)}{2\ell^{2}}\right)\! .
\end{displaymath} \end{thm}

\begin{proof} The result follows from~\cite[Theorem~1.7.9]{ConvexBodies} by viewing $H$ as a $2\dim(H)$-dimensional Euclidean space.  \end{proof}

We record the following consequence for complex-valued Lipschitz functions.

\begin{cor}\label{corollary:spherical.concentration} Let $H$ be a non-zero finite-dimensional Hilbert space, let $\ell \in \R_{\geq 0}$ and let $f \in \Lip_{\ell}(\Sph(H),\C)$. Then, for every $\epsilon \in \R_{>0}$, \begin{displaymath}
	\nu_{H}\!\left(\left\{ x \in \Sph(H) \left\vert \, \left\lvert f(x) - \int f \, \mathrm{d}\nu_{H} \right\rvert > \epsilon \right\} \right)\!\right. \, \leq \, 4\exp\!\left(-\tfrac{\epsilon^{2}(2\dim(H)-1)}{4\ell^{2}}\right)\! .
\end{displaymath} \end{cor}

\begin{proof} Since $\Re f, \, \Im f \in \Lip_{\ell}(\Sph(H),\R)$ and $\int f \, \mathrm{d}\nu_{H} = \int \Re f \, \mathrm{d}\nu_{H} + i \int \Im f \, \mathrm{d}\nu_{H}$, we see that \begin{align*}
	&\nu_{H}\!\left(\left\{ x \in \Sph(H) \left\vert \, \left\lvert f(x) - \int f \, \mathrm{d}\nu_{H} \right\rvert > \epsilon \right\} \right)\!\right. \\
	& \qquad \leq \, \nu_{H}\!\left(\left\{ x \in \Sph(H) \left\vert \, \left\lvert \Re f(x) - \int \Re f \, \mathrm{d}\nu_{H} \right\rvert > \tfrac{\epsilon}{\sqrt{2}} \right\} \right)\!\right. \\
	& \qquad \qquad \qquad \qquad + \nu_{H}\!\left(\left\{ x \in \Sph(H) \left\vert \, \left\lvert \Im f(x) - \int \Im f \, \mathrm{d}\nu_{H} \right\rvert > \tfrac{\epsilon}{\sqrt{2}} \right\} \right)\!\right. \\
	& \qquad \stackrel{\ref{theorem:spherical.concentration}}{\leq} \, 4\exp\!\left(-\tfrac{\epsilon^{2}(2\dim(H)-1)}{4\ell^{2}}\right)\! . \qedhere
\end{align*} \end{proof}

Our first application of measure concentration, in the proof of Proposition~\ref{proposition:hyperlinear.groups}, will come about as a combination of the following two lemmata.

\begin{lem}[cf.~{\cite[Proposition~3.3 + Corollary~3.2]{burtonetal}}]\label{lemma:burton.concentration} Let $H$ be a Hilbert space, let $p \in \Pro(H)$ with $0 < \rk(p) < \infty$, and let $a \in \B(H)$. Then, for every $\epsilon \in \R_{>0}$, \begin{align*}
	\nu_{p(H)}\!\left(\left\{ x \in \Sph(p(H)) \left\vert \, \left\lvert \Vert ax \Vert^{2} - \tfrac{\Vert ap \Vert_{2}^{2}}{\tr(p)} \right\rvert > \epsilon \right\} \right)\!\right. \, &\leq \, 2\exp\!\left(-\tfrac{\epsilon^{2}(2\rk(p)-1)}{8\Vert a\Vert^{4}}\right)\! , \\
	\nu_{p(H)}\!\left(\left\{ x \in \Sph(p(H)) \left\vert \, \left\lvert \langle x,ax \rangle - \tfrac{\tr(ap)}{\rk(p)} \right\rvert > \epsilon \right\} \right)\!\right. \, &\leq \, 4\exp\!\left(-\tfrac{\epsilon^{2}(2\rk(p)-1)}{16\Vert a\Vert^{2}}\right)\! .
\end{align*} \end{lem}

\begin{proof} The first item follows from Corollary~\ref{corollary:trace} and Theorem~\ref{theorem:spherical.concentration} upon considering the $2\Vert a\Vert^{2}$-Lipschitz function \begin{displaymath}
	f \colon \, \Sph(p(H)) \, \longrightarrow \, \R, \quad x \, \longmapsto \, \langle x,a^{\ast}ax\rangle = \Vert ax \Vert^{2} .
\end{displaymath} The second assertion is a consequence of Corollary~\ref{corollary:trace} and Corollary~\ref{corollary:spherical.concentration}, applied to the $2\Vert a\Vert$-Lipschitz function $f \colon \Sph(p(H)) \to \C, \, x \mapsto \langle x,ax \rangle$. \end{proof} 

\begin{lem}[cf.~{\cite[Proof of Lemma~5.2]{Pestov00}}\footnote{See also~\cite[p.~276]{milman}.}]\label{lemma:milman} Let $H$ be a non-zero finite-dimensional Hilbert space. If $B \in \Borel(\Sph(H))$ satisfies $\dim(H) \nu_{H}(\Sph(H)\setminus B) < 1$, then $B$ contains an orthonormal basis for $H$. \end{lem}

\begin{proof} Let $\lambda$ denote the Haar measure on $\U(H)$. For each $x \in \Sph(H)$, consider the map $f_{x} \colon \U(H) \to \Sph(H), \, u \mapsto ux$ and recall that $\nu_{H} = (f_{x})_{\ast}(\lambda)$. Now, fix some orthonormal basis $X$ for $H$. Then \begin{align*}
	&\lambda(\{ u \in \U(H) \mid \forall x \in X \colon \, ux \in B \}) \, = \, \lambda\!\left( \bigcap\nolimits_{x \in X} f_{x}^{-1}(B) \right) \\
	&\qquad = \, 1-\lambda\!\left( {\U(H)} \setminus {\bigcap\nolimits_{x \in X} f_{x}^{-1}(B)} \right) \! \, = \, 1-\lambda\!\left( \bigcup\nolimits_{x \in X} {\U(H)} \setminus f_{x}^{-1}(B) \right) \\
	&\qquad = \, 1-\lambda\!\left( \bigcup\nolimits_{x \in X} f_{x}^{-1}({\Sph(H)} \setminus B) \right) \! \, \geq \, 1-\sum\nolimits_{x \in X}\lambda\!\left( f_{x}^{-1}({\Sph(H)} \setminus B) \right) \\
	&\qquad = \, 1 - \vert X \vert \nu_{H}({\Sph(H)} \setminus B) \, = \, 1 - \dim(H) \nu_{H}({\Sph(H)} \setminus B) \, > \, 0 ,
\end{align*} so there exists $u \in \U(H)$ such that the resulting orthonormal basis $\{ ux \mid x \in X \}$ is contained in $B$. \end{proof}

\section{Hyperlinear groups}\label{section:hyperlinear.groups}

We now turn to hyperlinear groups. For a set $X$, we denote by $\Pfin(X)$ the set of all finite subsets of $X$. A group $G$ is called \emph{hyperlinear} if there exists a function $\gamma \colon G\setminus \{e\} \to \R_{>0}$ such that, for every $E \in \Pfin(G)$ and every $\epsilon \in \R_{>0}$, there exist a non-zero finite-dimensional Hilbert space $H$ and a map $\pi \colon G \to \U(H)$ such that \begin{itemize}
	\item[---\,] $\lVert \pi(gh) - \pi(g)\pi(h) \rVert_{\HS} \leq \epsilon$ for all $g,h \in E$,
	\item[---\,] $\lVert 1-\pi(g) \rVert_{\HS} \geq \gamma(g)$ for every $g \in E\setminus \{ e \}$.
\end{itemize} The reader is referred to~\cite{Pestov08,PestovKwiatkowska,CapraroLupini} for comprehensive accounts on such groups. For our purposes, the following characterization extracted from the work of Pestov and Kwiatkowska~\cite{PestovKwiatkowska} will be of central relevance.

\begin{thm}[{\cite[Theorem~4.2, (1)$\Longleftrightarrow$(2)]{PestovKwiatkowska}}]\label{theorem:kwiatkowska.pestov} Let $G$ be a group. Then the following are equivalent. \begin{enumerate}
	\item\label{theorem:kwiatkowska.pestov.1} $G$ is hyperlinear.
	\item\label{theorem:kwiatkowska.pestov.2} For all $E \in \Pfin(G)$ and $\epsilon \in \R_{>0}$, there exist a non-zero finite-dimensional Hilbert space $H$ and a map $\pi \colon G \to \U(H)$ such that \begin{itemize}
			\item[---\,] $\dim(H) \geq \tfrac{1}{\epsilon}$,
			\item[---\,] $\lVert \pi(gh) - \pi(g)\pi(h) \rVert_{\HS} \leq \epsilon$ for all $g,h \in E$,
			\item[---\,] $\tfrac{\vert\tr(\pi(g))\vert}{\dim(H)} \leq \epsilon$ for every $g \in E\setminus \{ e \}$.
		\end{itemize}
	\item\label{theorem:kwiatkowska.pestov.3} For all $E \in \Pfin(G)$ and $\epsilon \in \R_{>0}$, there exist a non-zero finite-dimensional Hilbert space $H$ and a map $\pi \colon G \to \U(H)$ such that \begin{itemize}
			\item[---\,] $\lVert \pi(gh) - \pi(g)\pi(h) \rVert_{\HS} \leq \epsilon$ for all $g,h \in E$,
			\item[---\,] $\tfrac{\vert\tr(\pi(g))\vert}{\dim(H)} \leq \epsilon$ for every $g \in E\setminus \{ e \}$.
		\end{itemize}
\end{enumerate} \end{thm}

\begin{proof} The implications \ref{theorem:kwiatkowska.pestov.2}$\Longrightarrow$\ref{theorem:kwiatkowska.pestov.3} and \ref{theorem:kwiatkowska.pestov.3}$\Longrightarrow$\ref{theorem:kwiatkowska.pestov.1} are obvious, while the implication \ref{theorem:kwiatkowska.pestov.1}$\Longrightarrow$\ref{theorem:kwiatkowska.pestov.2} is established in~\cite[Proof of Theorem~4.2, (1)$\Longrightarrow$(2)]{PestovKwiatkowska} following an amplification argument due to R\u{a}dulescu~\cite[Proposition~2.5]{Radulescu}\footnote{This argument extends an idea by Kirchberg~\cite[Proof of Corollary~1.2, (ii)$\Longrightarrow$(i), p.~561--562]{Kirchberg}, as noted by Ozawa~\cite[Proposition~7.1]{Ozawa}.}. \end{proof}

Everything is prepared for this section's main observation.

\begin{prop}\label{proposition:hyperlinear.groups} Let $G$ be a group. The following are equivalent. \begin{enumerate}
	\item\label{proposition:hyperlinear.groups.1} $G$ is hyperlinear.
	\item\label{proposition:hyperlinear.groups.2} For all $E \in \Pfin(G)$ and $\epsilon \in \R_{>0}$, there exist a non-zero finite-dimensional Hilbert space $H$ and a map $\pi \colon G \to \U(H)$ such that \begin{align*}
				\qquad &\forall g,h \in E \colon \ \, \nu_{H}(\{ x \in \Sph(H) \mid \lVert \pi(gh)x - \pi(g)\pi(h)x \rVert > \epsilon \}) \leq \epsilon, \\
				&\forall g \in E\setminus \{ e \} \colon \ \, \nu_{H}(\{ x \in \Sph(H) \mid \lvert \langle x, \pi(g)x \rangle \rvert > \epsilon \}) \leq \epsilon .
			\end{align*}
	\item\label{proposition:hyperlinear.groups.3} For all $E \in \Pfin(G)$ and $\epsilon \in \R_{>0}$, there exist a non-zero finite-dimensional Hilbert space $H$ with an orthonormal basis $X$ and a map $\pi \colon G \to \U(H)$ so~that \begin{align*}
				\qquad &\forall g,h \in E \ \forall x \in X \colon \ \, \lVert \pi(gh)x - \pi(g)\pi(h)x \rVert \leq \epsilon , \\
				&\forall g \in E\setminus \{ e \} \ \forall x \in X\colon \ \, \lvert \langle x, \pi(g)x \rangle \rvert \leq \epsilon .
			\end{align*}
\end{enumerate} \end{prop}

\begin{proof} \ref{proposition:hyperlinear.groups.1}$\Longrightarrow$\ref{proposition:hyperlinear.groups.2}$\wedge$\ref{proposition:hyperlinear.groups.3}. Suppose that $G$ is a hyperlinear group. Let $E \in \Pfin(G)$ and let $\epsilon \in (0,1]$. According to Theorem~\ref{theorem:kwiatkowska.pestov}, there exist a non-zero finite-dimensional Hilbert space $H$ and a map $\pi \colon G \to \U(H)$ such that \begin{align}
	&2\exp\!\left(-\tfrac{\epsilon^{4} (2\dim(H)-1)}{512}\right) < \min \! \left\{ \epsilon, \tfrac{1}{(2 \vert E \vert^{2} +1)\dim(H)} \right\}\!,\label{eq:hyperlinear.dimension.1} \\
	&4\exp\!\left(-\tfrac{\epsilon^{2} (2\dim(H)-1)}{64}\right) < \min \! \left\{ \epsilon, \tfrac{1}{(2 \vert E \vert +1)\dim(H)} \right\} \!,\label{eq:hyperlinear.dimension.2} \\
	&\forall g,h \in E \colon \ \, \lVert \pi(gh) - \pi(g)\pi(h) \rVert_{\HS}^{2} \leq \tfrac{\epsilon^{2}}{2} ,\label{eq:hyperlinear.homomorphism} \\
	&\forall g \in E\setminus \{ e \} \colon \ \, \left\lvert \tfrac{\tr(\pi(g))}{\dim(H)} \right\rvert \leq \tfrac{\epsilon}{2}.\label{eq:hyperlinear.freeness}
\end{align} Consider \begin{align*}
	B_{g,h} \, &\defeq \, \{ x \in \Sph(H) \mid \Vert \pi(gh)x-\pi(g)\pi(h)x \Vert \leq \epsilon \} \, \in \, \Borel(\Sph(H)), \\
	C_{g} \, &\defeq \, \{ x \in \Sph(H) \mid \lvert \langle x, \pi(g)x \rangle \rvert \leq \epsilon \} \, \in \, \Borel(\Sph(H))
\end{align*} for any $g,h \in G$, as well as \begin{align*}
	B \, &\defeq \, {\bigcap\nolimits_{g,h \in E} B_{g,h}} \cap {\bigcap\nolimits_{g \in E\setminus\{ e\}} C_{g}} \, \in \, \Borel(\Sph(H)).
\end{align*} If $g,h \in E$, then \begin{align*}
	&\Sph(H)\setminus {B_{g,h}} \, = \, \{ x \in \Sph(H) \mid \Vert \pi(gh)x-\pi(g)\pi(h)x \Vert > \epsilon \} \\
	&\qquad = \, \left\{ x \in \Sph(H) \left\vert \, \Vert (\pi(gh)-\pi(g)\pi(h))x \Vert^{2} > \epsilon^{2} \right\} \right. \\
	&\qquad \stackrel{\eqref{eq:hyperlinear.homomorphism}}{\subseteq} \, \left\{ x \in \Sph(H) \left\vert \, \left\lvert \Vert (\pi(gh)-\pi(g)\pi(h))x \Vert^{2} - \Vert \pi(gh) -\pi(g)\pi(h) \Vert_{\HS}^{2} \right\rvert > \tfrac{\epsilon^{2}}{2} \right\} \right.
\end{align*} and therefore \begin{align*}
	\nu_{H}(\Sph(H)\setminus {B_{g,h}}) \, &\stackrel{\ref{lemma:burton.concentration}}{\leq} \, 2\exp\!\left(-\tfrac{(\epsilon^{2}/2)^{2}(2\dim(H)-1)}{8\Vert \pi(gh) -\pi(g)\pi(h)\Vert^{4}}\right)\! \, \leq \, 2\exp\!\left(-\tfrac{\epsilon^{4}(2\dim(H)-1)}{512}\right) \\
	& \stackrel{\eqref{eq:hyperlinear.dimension.1}}{<} \, \min \! \left\{ \epsilon, \tfrac{1}{(2\vert E \vert^{2} + 1)\dim(H)} \right\} .
\end{align*} Moreover, if $g \in E\setminus \{ e\}$, then \begin{align*}
	\Sph(H)\setminus {C_{g}} \, &= \, \{ x \in \Sph(H) \mid \vert \langle x,\pi(g)x\rangle \vert > \epsilon \} \\
	&\stackrel{\eqref{eq:hyperlinear.freeness}}{\subseteq} \, \left\{ x \in \Sph(H) \left\vert \, \left\lvert \langle x,\pi(g)x\rangle - \tfrac{\tr(\pi(g))}{\dim(H)} \right\rvert > \tfrac{\epsilon}{2} \right\} \right.
\end{align*} and hence \begin{align*}
	\nu_{H}(\Sph(H)\setminus {C_{g}}) \, &\stackrel{\ref{lemma:burton.concentration}}{\leq} \, 4\exp\!\left(-\tfrac{(\epsilon/2)^{2}(2\dim(H)-1)}{16\Vert \pi(g)\Vert^{2}}\right)\! \, = \, 4\exp\!\left(-\tfrac{\epsilon^{2}(2\dim(H)-1)}{64}\right) \\
	& \stackrel{\eqref{eq:hyperlinear.dimension.2}}{<} \, \min \! \left\{ \epsilon,\tfrac{1}{(2\vert E \vert +1)\dim(H)} \right\}
\end{align*} This already proves~\ref{proposition:hyperlinear.groups.2}. Furthermore, we conclude that \begin{align*}
	&\nu_{H}(\Sph(H)\setminus B) \, = \, \nu_{H}\!\left( {\bigcup\nolimits_{g,h \in E} \Sph(H) \setminus B_{g,h}} \cup {\bigcup\nolimits_{g \in E\setminus\{ e\}} \Sph(H)\setminus C_{g}} \right) \\
	& \qquad \leq \, \sum\nolimits_{g,h \in E} \nu_{H}(\Sph(H) \setminus B_{g,h}) + \sum\nolimits_{g \in E\setminus \{ e\}} \nu_{H}(\Sph(H) \setminus C_{g}) \, < \, \tfrac{1}{\dim(H)} ,
\end{align*} that is, $\dim(H) \nu_{H}(\Sph(H)\setminus B) < 1$. Due to Lemma~\ref{lemma:milman}, thus $B$ contains an orthonormal basis for $H$, which completes the proof of~\ref{proposition:hyperlinear.groups.3}. 

\ref{proposition:hyperlinear.groups.2}$\Longrightarrow$\ref{proposition:hyperlinear.groups.1}. We verify condition~\ref{theorem:kwiatkowska.pestov.3} of Theorem~\ref{theorem:kwiatkowska.pestov}. To this end, let $E \in \Pfin(G)$ and $\epsilon \in \R_{>0}$. By our hypothesis, there exist a non-zero finite-dimensional Hilbert space $H$ and a map $\pi \colon G \to \U(H)$ such that \begin{align}
	&\forall g,h \in E \colon \ \, \nu_{H}\!\left(\left\{ x \in \Sph(H) \left\vert \, \lVert \pi(gh)x - \pi(g)\pi(h)x \rVert > \tfrac{\epsilon}{\sqrt{2}} \right\}\right)\right.\! \leq \tfrac{\epsilon^{2}}{8} , \label{eq:measure.homomorphism} \\
	&\forall g \in E\setminus \{ e \} \colon \ \, \nu_{H}\!\left(\left\{ x \in \Sph(H) \left\vert \, \lvert \langle x,\pi(g)x \rangle \rvert > \tfrac{\epsilon}{2} \right\}\right)\right.\! \leq \tfrac{\epsilon}{2} . \label{eq:measure.freeness}
\end{align} For all $g,h \in E$, \begin{displaymath}
	\Vert \pi(gh)-\pi(g)\pi(h) \Vert_{\HS}^{2} \, \stackrel{\ref{proposition:trace}}{=} \, \int \Vert \pi(gh)x-\pi(g)\pi(h)x \Vert^{2} \, \mathrm{d}\nu_{H}(x) \, \stackrel{\eqref{eq:measure.homomorphism}}{\leq} \, \epsilon^{2}
\end{displaymath} and so $\Vert \pi(gh)-\pi(g)\pi(h) \Vert_{\HS} \leq \epsilon$. Moreover, if $g \in E\setminus\{e\}$, then \begin{displaymath}
	\tfrac{\vert \tr(\pi(g)) \vert}{\dim(H)} \, \stackrel{\ref{proposition:trace}}{=} \,  \left\lvert \int \langle x,\pi(g)x \rangle \, \mathrm{d}\nu_{H}(x) \right\rvert \, \leq \, \int \vert \langle x,\pi(g)x \rangle \vert \, \mathrm{d}\nu_{H}(x) \, \stackrel{\eqref{eq:measure.freeness}}{\leq} \, \epsilon .
\end{displaymath} This proves hyperlinearity of $G$.

\ref{proposition:hyperlinear.groups.3}$\Longrightarrow$\ref{proposition:hyperlinear.groups.1}. We verify condition~\ref{theorem:kwiatkowska.pestov.3} of Theorem~\ref{theorem:kwiatkowska.pestov}. To this end, let $E \in \Pfin(G)$ and $\epsilon \in \R_{>0}$. According to our hypothesis, there exist a non-zero finite-dimensional Hilbert space $H$ with an orthonormal basis $X$ and a map $\pi \colon G \to \U(H)$ such that \begin{align}
	&\forall g,h \in E \ \forall x \in X \colon \ \, \lVert \pi(gh)x - \pi(g)\pi(h)x \rVert \leq \epsilon , \label{eq:onb.homomorphism} \\
	&\forall g \in E\setminus \{ e \} \ \forall x \in X\colon \ \, \vert \langle x,\pi(g)x \rangle \vert \leq \epsilon . \label{eq:onb.freeness}
\end{align} For all $g,h \in E$, \begin{align*}
	\Vert \pi(gh)-\pi(g)\pi(h) \Vert_{\HS}^{2} \, &= \, \tfrac{1}{\dim(H)}\tr((\pi(gh)-\pi(g)\pi(h))^{\ast}(\pi(gh)-\pi(g)\pi(h))) \\
	& = \, \tfrac{1}{\vert X \vert} \sum\nolimits_{x \in X} \langle x,(\pi(gh)-\pi(g)\pi(h))^{\ast}(\pi(gh)-\pi(g)\pi(h))x\rangle \\
	& = \, \tfrac{1}{\vert X \vert} \sum\nolimits_{x \in X} \Vert \pi(gh)x-\pi(g)\pi(h)x \Vert^{2} \, \stackrel{\eqref{eq:onb.homomorphism}}{\leq} \, \epsilon^{2} .
\end{align*} and so $\Vert \pi(gh)-\pi(g)\pi(h) \Vert_{\HS} \leq \epsilon$. Moreover, if $g \in E\setminus\{e\}$, then \begin{displaymath}
	\tfrac{\vert \tr(\pi(g)) \vert}{\dim(H)} \, = \, \tfrac{1}{\vert X \vert} \left\lvert \sum\nolimits_{x \in X} \langle x,\pi(g)x \rangle \right\rvert \, \leq \, \tfrac{1}{\vert X \vert} \sum\nolimits_{x \in X} \lvert \langle x,\pi(g)x \rangle \rvert \, \stackrel{\eqref{eq:onb.freeness}}{\leq} \, \epsilon .
\end{displaymath} This shows that $G$ is hyperlinear. \end{proof}

\begin{remark}\label{remark:hyperlinear.groups} As the proof of Proposition~\ref{proposition:hyperlinear.groups} shows, one may additionally require that $\dim(H) \geq \tfrac{1}{\epsilon}$ in items~\ref{proposition:hyperlinear.groups.2} and~\ref{proposition:hyperlinear.groups.3}. \end{remark}

\section{Near representations}\label{section:near.representations}

Let $H$ be a Hilbert space. The embedding \begin{displaymath}
	\U(H) \, \lhook\joinrel\longrightarrow \, \Iso(\Sph(H)) , \quad u \, \longmapsto \, u\vert_{\Sph(H)} 
\end{displaymath} gives rise to actions of $\U(H)$ on both the Borel $\sigma$-algebra $\Borel(\Sph(H))$ and the algebra $\UCB(\Sph(H))$ of uniformly continuous bounded complex-valued functions. The focus of this section lies on invariance properties of probability charges on $\Borel(\Sph(H))$ and states on $\UCB(\Sph(H))$ or $\B(H)$, respectively, and corresponding notions of amenable near representations. The section's main result is Theorem~\ref{theorem:amenable.trace}, its technical core is Lemma~\ref{lemma:key}. We start off with a general remark.

\begin{remark}\label{remark:means.from.charges} Let $X$ be a metric space. Given $g \in \Iso(X)$, a probability charge $\mu$ on $\Borel(X)$, and a state $\nu$ on $\UCB(X)$, we will consider the probability charge \begin{displaymath}
	g.\mu \colon \, \Borel(X) \, \longrightarrow \, [0,1], \quad B \, \longmapsto \, \mu\!\left(g^{-1}B\right)
\end{displaymath} as well as the state \begin{displaymath}
	g.\nu \colon \, \UCB(X) \, \longrightarrow \, \C, \quad f \, \longmapsto \, \nu(f \circ g) .
\end{displaymath} If $\mu$ is a probability charge on $\Borel(X)$, then \begin{displaymath}
	\mu_{\circ} \colon \, \UCB(X) \, \longrightarrow \, \C , \quad f \, \longmapsto \, \int f \, \mathrm{d}\mu
\end{displaymath} is a state on $\UCB(X)$ (cf.~\cite[363L]{Fremlin3}) and $g.(\mu_{\circ}) = (g.\mu)_{\circ}$ for every $g \in \Iso(X)$. \end{remark}

\begin{lem}[cf.~{\cite[Proposition~3.1]{Pestov00}}]\label{lemma:states.from.means} Let $H$ be a Hilbert space. If $\mu$ is a state on $\UCB(\Sph(H))$, then \begin{displaymath}
	\mu_{\bullet} \colon \, \B(H) \, \longrightarrow \, \C , \quad a \, \longmapsto \, \mu(x \mapsto \langle x,ax \rangle)
\end{displaymath} is a state on $\B(H)$ and $u.(\mu_{\bullet}) = (u.\mu)_{\bullet}$ for every $u \in \U(H)$. \end{lem}

\begin{proof} This is proved in~\cite[Proof of Proposition~3.1]{Pestov00}. \end{proof}

We arrive at the anticipated notions of amenable near representations.

\begin{definition}\label{definition:near.representation.variations} Let $G$ be a group, let $H$ be a Hilbert space and let $\pi \colon G \to \U(H)$. \begin{enumerate}
	\item Let $\mu$ be a probability charge on $\Borel(\Sph(H))$. Then $\pi$ is said to be a \emph{$\mu$-near representation} if $\mu$ is $\pi(G)$-invariant and, for all $g,h \in G$ and $\epsilon \in \R_{>0}$, \begin{displaymath}
			\qquad \mu(\{ x \in \Sph(H) \mid \lVert \pi(gh)x - \pi(g)\pi(h)x \rVert \leq \epsilon \}) \, = \, 1 .
		\end{displaymath}
	\item Let $\mu$ be a state on $\UCB(\Sph(H))$. Then $\pi$ is said to be a \emph{$\mu$-near representation} if $\mu$ is $\pi(G)$-invariant and, for all $g,h \in G$, \begin{displaymath}
		\qquad \mu( x \mapsto \lVert \pi(gh)x - \pi(g)\pi(h)x \rVert) \, = \, 0 .
	\end{displaymath}
	\item Let $\phi$ be a state on $\B(H)$. Then $\pi$ is said to be a \emph{$\phi$-near representation} if $\phi$ is $\pi(G)$-invariant and, for all $g,h \in G$, \begin{displaymath}
		\qquad \phi((\pi(gh)-\pi(g)\pi(h))^{\ast}(\pi(gh)-\pi(g)\pi(h))) \, = \, 0 .
	\end{displaymath}
\end{enumerate} \end{definition}

Let us clarify the relations between the concepts just introduced.

\begin{lem}\label{lemma:near.representations} Let $G$ be a group, let $H$ be a Hilbert space and let $\pi \colon G \to \U(H)$. \begin{enumerate}
	\item\label{lemma:near.representations.1} Let $\mu$ be a $\pi(G)$-invariant probability charge on $\Borel(\Sph(H))$. Then $\pi$ is a $\mu$-near representation if and only if $\pi$ is a $\mu_{\circ}$-near representation.
	\item\label{lemma:near.representations.2} Let $\mu$ be a $\pi(G)$-invariant state on $\UCB(\Sph(H))$. Then $\pi$ is a $\mu$-near representation if and only if $\pi$ is a $\mu_{\bullet}$-near representation.
\end{enumerate} \end{lem}

\begin{proof} \ref{lemma:near.representations.1} First note that $\mu_{\circ}$ is $\pi(G)$-invariant due to Remark~\ref{remark:means.from.charges}. For any $g,h \in G$, we let $f_{g,h} \colon \Sph(H) \to [0,2], \, x \mapsto \lVert \pi(gh)x - \pi(g)\pi(h)x \rVert$.
	
($\Longrightarrow$) Suppose that $\pi$ is a $\mu$-near representation. If $g,h \in G$, then we see that\footnote{Here and in the following, if $X$ is a set and $A\subseteq X$, then we let \begin{displaymath} \chi_{A} \colon \, X \, \longrightarrow \, \{0,1\}, \quad x \,\longmapsto \, \begin{cases}
			\,1 & \text{if } x \in A, \\
			\,0 & \text{otherwise.}
\end{cases}\end{displaymath}} \begin{align*}
	\mu_{\circ}(f_{g,h}) \, &= \, \int f_{g,h} \, \mathrm{d}\mu \, = \, \int f_{g,h}\chi_{f_{g,h}^{-1}([0,\epsilon])} \, \mathrm{d}\mu + \int f_{g,h}\chi_{f_{g,h}^{-1}((\epsilon,2])} \, \mathrm{d}\mu \\
	&\leq \, \epsilon \mu\!\left( f_{g,h}^{-1}([0,\epsilon])\right)\! + 2 \mu\!\left( f_{g,h}^{-1}((\epsilon,2]) \right)\! \, = \, \epsilon
\end{align*} for every $\epsilon \in \R_{>0}$, that is, $\mu_{\circ}(f_{g,h}) = 0$. Hence, $\pi$ is a $\mu_{\circ}$-near representation. 

($\Longleftarrow$) Assume now that $\pi$ is a $\mu_{\circ}$-near representation. For all $g,h \in G$ and $\epsilon \in \R_{>0}$, from $\chi_{f_{g,h}^{-1}((\epsilon,2])} \leq \epsilon^{-1}f_{g,h}$ we infer that \begin{displaymath}
	\mu\!\left( f_{g,h}^{-1}([0,\epsilon])\right)\! \, = \, 1- \mu\!\left( f_{g,h}^{-1}((\epsilon,2]) \right)\! \, = \, 1- \int \chi_{f_{g,h}^{-1}((\epsilon,2])} \, \mathrm{d}\mu\, \geq \, 1- \epsilon^{-1}\int f_{g,h} \, \mathrm{d}\mu \, = \, 1 .
\end{displaymath} This shows that $\pi$ is a $\mu$-near representation.

\ref{lemma:near.representations.2} Note that $\phi \defeq \mu_{\bullet}$ is $\pi(G)$-invariant by Lemma~\ref{lemma:states.from.means}. For all $g,h \in G$, we see that \begin{align*}
	&\phi((\pi(gh) - \pi(g)\pi(h))^{\ast}(\pi(gh) - \pi(g)\pi(h))) \\
	&\qquad = \, \mu (x \mapsto \langle x,( \pi(gh) - \pi(g)\pi(h))^{\ast}(\pi(gh) - \pi(g)\pi(h))x \rangle) \\
	&\qquad = \, \mu (x \mapsto \langle (\pi(gh) - \pi(g)\pi(h))x, (\pi(gh) - \pi(g)\pi(h))x \rangle) \\
	&\qquad = \, \mu \!\left( x \mapsto \Vert (\pi(gh) - \pi(g)\pi(h))x \Vert^{2} \right) ,
\end{align*} wherefore both \begin{align*}
	\phi((\pi(gh) - \pi(g)\pi(h))^{\ast}(\pi(gh) - \pi(g)\pi(h))) \, &\leq \, \mu(x \mapsto 2\Vert (\pi(gh) - \pi(g)\pi(h)) x \Vert) \\
	& = \, 2\mu(x \mapsto \Vert (\pi(gh) - \pi(g)\pi(h)) x \Vert)
\end{align*} and, thanks to the Cauchy--Schwartz inequality, \begin{align*}
	\mu(x \mapsto \Vert (\pi(gh) - \pi(g)\pi(h)) x \Vert) \, &\leq \, \sqrt{\mu \!\left( x \mapsto \Vert (\pi(gh) - \pi(g)\pi(h))x \Vert^{2} \right)} \\
	& = \, \sqrt{\phi((\pi(gh) - \pi(g)\pi(h))^{\ast}(\pi(gh) - \pi(g)\pi(h)))} .
\end{align*} The desired equivalence now follows at once. \end{proof}

Let us briefly elaborate on some general properties of amenable near representations. Recall that a \emph{character} on a group $G$ is a function $\phi \colon G \to \C$ such that \begin{itemize}
	\item[---\,] $\phi(e) = 1$,
	\item[---\,] $\phi$ is \emph{of positive type}, i.e., \begin{displaymath}
		\qquad \forall n \in \N \ \forall c \in \C^{n} \ \forall g \in G^{n} \colon \quad \sum\nolimits_{i,j=1}^{n} \overline{c_{i}}c_{j} \phi\bigl(g_{i}^{-1}g_{j}\bigr) \, \geq \, 0 ,
	\end{displaymath}
	\item[---\,] and $\phi(h^{-1}gh) = \phi(g)$ for all $g,h \in G$.
\end{itemize} The \emph{regular character} on a group $G$ is the map \begin{displaymath}
	G \, \longrightarrow \, \C, \quad g \, \longmapsto \, \begin{cases}
		\, 1 & \text{if } g=e, \\
		\, 0 & \text{otherwise}.
	\end{cases}
\end{displaymath}

\begin{prop}\label{proposition:near.rep} Let $G$ be a group, $H$ be a Hilbert space, $\phi$ be a state on $\B(H)$, and $\pi \colon G \to \U(H)$ be a $\phi$-near representation. Then the following hold. \begin{enumerate}
	\item\label{proposition:near.rep.pilinearinphi} $\phi(\pi(g)^{\ast}\pi(h)) = \phi(\pi(g^{-1}h))$ and $\phi(\pi(g)\pi(h)^{\ast}) = \phi(\pi(gh^{-1}))$ for all $g,h \in G$.
	\item\label{proposition:near.rep.character} $\phi \circ \pi$ is a character on $G$.
	\item\label{proposition:near.rep.regular} $\phi \circ \pi$ is the regular character on $G$ if and only if $\phi(\pi(g)^{\ast}\pi(h)) = 0$ for any two distinct $g,h \in G$.
	\item\label{proposition:near.rep.e} The map \begin{displaymath}
			\qquad \tilde{\pi} \colon \, G \, \longrightarrow \, \U(H), \quad g \, \longmapsto \, \begin{cases}
				\, 1 & \text{if } g = e, \\
				\, \pi(g) & \text{otherwise}
			\end{cases}
		\end{displaymath} is a $\phi$-near representation and, for every $g \in G$, \begin{displaymath}
			\qquad \phi((1-\tilde{\pi}(g))^{\ast}(1-\tilde{\pi}(g))) \, = \, \phi((1-\pi(g))^{\ast}(1-\pi(g))).
		\end{displaymath}
	\end{enumerate}
\end{prop}

\begin{proof} \ref{proposition:near.rep.pilinearinphi} For all $g,h \in G$, \begin{align*}
	&\vert \phi(\pi(g)^{\ast}\pi(h)) - \phi(\pi(g^{-1}h)) \vert \, = \, \lvert \phi(\pi(g)^{\ast}(\pi(h) - \pi(g)\pi(g^{-1}h))) \rvert \\
	&\quad \leq \, \sqrt{\phi(\pi(g)^{\ast}\pi(g))}\sqrt{\phi((\pi(h)-\pi(g)\pi(g^{-1}h))^{\ast}(\pi(h)-\pi(g)\pi(g^{-1}h)))} \, = \, 0
\end{align*} by the Cauchy--Schwartz inequality and $\pi$ being a $\phi$-near representation, hence \begin{displaymath}
	\phi(\pi(g)^{\ast}\pi(h)) \, = \, \phi(\pi(g^{-1}h)) .
\end{displaymath} Combining the same argument with the $\pi(G)$-invariance of $\phi$, we see that \begin{align*}
	&\vert \phi(\pi(g)\pi(h)^{\ast}) - \phi(\pi(gh^{-1})) \vert \, = \, \lvert \phi((\pi(g) - \pi(gh^{-1})\pi(h))\pi(h)^{\ast}) \rvert \\
	& \quad = \, \lvert \phi(\pi(h)^{\ast}(\pi(g) - \pi(gh^{-1})\pi(h))) \rvert \\
	&\quad \leq \, \sqrt{\phi(\pi(h)^{\ast}\pi(h))}\sqrt{\phi((\pi(g) - \pi(gh^{-1})\pi(h))^{\ast}(\pi(g) - \pi(gh^{-1})\pi(h)))} \, = \, 0
\end{align*} and thus $\phi(\pi(g)\pi(h)^{\ast}) = \phi(\pi(gh^{-1}))$ for all $g,h \in G$.
	
\ref{proposition:near.rep.character} First of all, \begin{displaymath}
	\phi(\pi(e)) \, = \, \phi(\pi(e^{-1}e)) \, \stackrel{\ref{proposition:near.rep.pilinearinphi}}{=} \, \phi(\pi(e)^{\ast}\pi(e)) \, = \, \phi(1) \, = \, 1 .
\end{displaymath} Also, if $n \in \N$, $c \in \C^{n}$ and $g \in G^{n}$, then \begin{align*}
	\sum\nolimits_{i,j=1}^{n} \overline{c_{i}}c_{j}\phi(\pi(g_{i}^{-1}g_{j})) \, &\stackrel{\ref{proposition:near.rep.pilinearinphi}}{=} \, \sum\nolimits_{i,j=1}^{n} \overline{c_{i}}c_{j}\phi(\pi(g_{i})^{\ast}\pi(g_{j})) \\
		& = \, \phi\!\left(\left( \sum\nolimits_{i=1}^{n} c_{i}\pi(g_{i})\right)^{\ast} \! \left( \sum\nolimits_{i=1}^{n} c_{i}\pi(g_{i}) \right)\right)\! \, \geq \, 0
\end{align*} due to positivity of $\phi$. And for all $g,h\in G$, the $\pi(G)$-invariance of $\phi$ entails that \begin{displaymath}
	\phi(\pi(h^{-1}gh)) \, \stackrel{\ref{proposition:near.rep.pilinearinphi}}{=} \, \phi(\pi(h)^{\ast}\pi(gh)) \, = \, \phi(\pi(gh)\pi(h)^{\ast}) \, \stackrel{\ref{proposition:near.rep.pilinearinphi}}{=} \, \phi(\pi(ghh^{-1})) \, = \, \phi(\pi(g)) .
\end{displaymath}

\ref{proposition:near.rep.regular} This is an immediate consequence of~\ref{proposition:near.rep.pilinearinphi} and~\ref{proposition:near.rep.character}.
	
\ref{proposition:near.rep.e} Clearly, $\tilde{\pi}(e).\phi=\phi$ and $\tilde{\pi}(ge)-\tilde{\pi}(g)\tilde{\pi}(e)=\tilde{\pi}(eg)-\tilde{\pi}(e)\tilde{\pi}(g)=0$ for all~$g\in G$. Hence, due to $\pi$ being a $\phi$-near representation, $\tilde{\pi}$ is such, too. From~\ref{proposition:near.rep.character} we know that $\phi(\pi(e)) = 1 = \phi(\tilde{\pi}(e))$, which entails that, for every $g \in G$, \begin{displaymath}
	\phi((1-\tilde{\pi}(g))^{\ast}(1-\tilde{\pi}(g))) \, = \,\phi((1-\pi(g))^{\ast}(1-\pi(g))) .\qedhere
\end{displaymath} \end{proof}

Preparing the proof of Theorem~\ref{theorem:amenable.trace}, let us recollect the following fact about finite-rank projections almost commuting with a unitary.

\begin{lem}[cf.~{\cite[Proposition~2.4]{burtonetal}}]\label{lemma:almost.commuting} Let $H$ be a Hilbert space, $p \in \Pro(H)$ with $\rk(p) < \infty$, and $u \in \U(H)$. If $\epsilon \in \left(0,\tfrac{1}{2}\right)$ is such that $\Vert (1-p)up \Vert_{2}^{2} \leq \epsilon \rk(p)$, then there exists $v \in \U(H)$ such that $vp=pv$ and $\Vert (u-v)p \Vert_{2}^{2}\leq 4\epsilon\rk(p)$. \end{lem}

\begin{proof} Evidently, $H$ admits a linear subspace $H_{0} \defeq p(H) + up(H)$ and note that $\dim(H_{0}) \leq 2\rk(p)$. Consider $p_{0} \defeq p\vert_{H_{0}} \in \Pro(H_{0})$. Moreover, we find $u_{0} \in \U(H_{0})$ with $u_{0}\vert_{p(H)} = u\vert_{p(H)}$. Now, let $\epsilon \in \left(0,\tfrac{1}{2}\right)$ be such that $\Vert (1-p)up \Vert_{2}^{2} \leq \epsilon \rk(p)$. Then \begin{displaymath}
	\Vert (1-p_{0})u_{0}p_{0} \Vert_{2}^{2} \, = \, \Vert (1-p)up \Vert_{2}^{2} \, \leq \, \epsilon \rk(p) \, = \, \tfrac{\epsilon \rk(p)}{\dim(H_{0})} \dim(H_{0}) .
\end{displaymath} Hence, by \cite[Proposition~2.4]{burtonetal}, there exists some $v_{0} \in \B(H_{0})$ such that $v_{0}p_{0}=p_{0}v_{0}$, $v_{0}v_{0}^{\ast} = v_{0}^{\ast}v_{0} = p_{0}$ and $\Vert (u_{0}-v_{0})p_{0} \Vert_{2}^{2} \leq 4\tfrac{\epsilon \rk(p)}{\dim(H_{0})} \dim(H_{0})$. We conclude that \begin{displaymath}
	v \, \defeq \, v_{0}p + 1-p \, \in \, \U(H) ,
\end{displaymath} $vp=pv$ and $\Vert (u-v)p \Vert_{2}^{2}\leq 4\tfrac{\epsilon \rk(p)}{\dim(H_{0})} \dim(H_{0}) = 4\epsilon\rk(p)$. \end{proof}

We also recall that a linear map $\sigma \colon A \to B$ between $\Cstar$-algebras $A$ and $B$ is said to be \emph{completely positive} (see, e.g.,~\cite[Definition 1.5.1, p.~9]{BrownOzawa}) if, for every $n\in\N$, the map \begin{displaymath}
	\Mat_{n}(A) \, \longrightarrow \, \Mat_{n}(B), \quad (a_{ij})_{i,j \in \{ 1,\ldots,n\}} \, \longmapsto \, (\sigma(a_{ij}))_{i,j \in \{ 1,\ldots,n\}}
\end{displaymath} between the associated matrix algebras is positive. Of course, the following basic fact applies to completely positive linear maps in particular.

\begin{remark}\label{remark:positive} Let $\sigma \colon A \to B$ be a positive linear map between $\Cstar$-algebras $A$ and~$B$. \begin{enumerate}
	\item\label{remark:positive.1} Since the $\C$-vector space $A$ is spanned by the set of positive elements of $A$ (see, e.g.,~\cite[Proposition~1.7(b), p.~4]{ConwayBook} and~\cite[Proposition~3.2, p.~12]{ConwayBook}), it follows that $\sigma(a^{\ast}) = \sigma(a)^{\ast}$ for every $a \in A$.
	\item\label{remark:positive.2} From~\ref{remark:positive.1} and boundedness of $\sigma$ (for instance, see~\cite[II.6.9.2]{Blackadar}), we infer that the set $\{ a \in A \mid \forall b \in A \colon \, \sigma(ab) = \sigma(ba) \}$ is a $C^{\ast}$-subalgebra of~$A$.
\end{enumerate} \end{remark}

The following classical result about completely positive linear maps constitutes another important ingredient in the proof of Theorem~\ref{theorem:amenable.trace}.

\begin{thm}[Stinespring dilation theorem~\cite{Stinespring}]\label{theorem:stinespring} Let $A$ be a unital $\Cstar$-algebra, let $H$ be a Hilbert space and let $\sigma \colon A \to \B(H)$ be a unital completely positive linear map. Then there exist a Hilbert space $K$, a unital ${}^{\ast}$-algebra homomorphism $\rho \colon A \to \B(K)$ and a linear isometry $v \colon H \to K$ such that \begin{displaymath}
	\forall a \in A \colon \qquad \sigma(a) = v^{\ast}\rho(a)v .
\end{displaymath} \end{thm}

\begin{proof} This is due to~\cite[Theorem~1\,$+$\,Remark on p.~213]{Stinespring}. (Alternatively, a proof is to be found in~\cite[Theorem~1.5.3\,$+$\,Remark~1.5.5, p.~10--11]{BrownOzawa}). \end{proof}

Concerning a Hilbert space $H$ and a unital $C^{\ast}$-subalgebra $A$ of $\B(H)$, we recall some further terminology: a \emph{hypertrace} for $A$ is a state $\phi$ on $\B(H)$ such that $\phi(ab) = \phi(ba)$ for all $a \in A$ and $b \in \B(H)$. The refinement due to Ozawa~\cite[Theorem~6.1]{Ozawa}\footnote{The proof is also given in~\cite[Theorem~6.2.7, p.~219--220]{BrownOzawa}.} of a result by Kirchberg~\cite[Proposition~3.2]{Kirchberg} entails the following approximation lemma for hypertraces, which is fundamental to the proof of the subsequent Lemma~\ref{lemma:key}, the technical key to Theorem~\ref{theorem:amenable.trace}.

\begin{lem}[cf.~{\cite[Theorem~6.1]{Ozawa}}]\label{lemma:brown.ozawa} Let $H$ be a Hilbert space, let $A$ be a unital $C^{\ast}$-subalgebra of $\B(H)$ and let $\phi \colon \B(H) \to \C$ be a hypertrace for $A$. Then, for every $E \in \Pfin(A)$ and every $\epsilon \in \R_{>0}$, there exist a non-zero finite-dimensional Hilbert space $H_{0}$ and a unital completely positive linear map $\sigma \colon A \to \B(H_{0})$ such that \begin{displaymath}
	\forall a,b \in E \colon \qquad \dim(H_{0})^{-1}\lvert \tr(\sigma(a^{\ast}b)) - \tr(\sigma(a)^{\ast}\sigma(b)) \rvert \leq \epsilon
\end{displaymath} and \begin{displaymath}
	\forall a \in E \colon \qquad \left\lvert \phi(a) - \dim(H_{0})^{-1}\tr(\sigma(a)) \right\rvert \leq \epsilon .
\end{displaymath} \end{lem}

\begin{proof} This is due to~\cite[Theorem~6.1, (1)$\Longrightarrow$(2)]{Ozawa} and Remark~\ref{remark:positive}\ref{remark:positive.1}. \end{proof}
	
\begin{lem}\label{lemma:key} Let $G$ be a group, $H$ be a Hilbert space, $\phi \colon \B(H) \to \C$ be a state, $\pi \colon G \to \U(H)$ be a $\phi$-near representation, and \begin{displaymath}
	\gamma \colon \, G \, \longrightarrow \, [0,2], \quad g \, \longmapsto \, \phi((1-\pi(g))^{\ast}(1-\pi(g)))^{1/2} .
\end{displaymath} For every $E \in \Pfin(G)$ and every $\epsilon \in \R_{>0}$, there exist a non-zero finite-dimensional Hilbert space $H_{0}$ and a map $\psi \colon G \to \U(H_{0})$ such that, for all $g,h \in E$, \begin{enumerate}
	\item[---\,] $\Vert \psi(gh) - \psi(g)\psi(h) \Vert_{\HS} \leq \epsilon$,
	\item[---\,] $\Vert 1-\psi(g) \Vert_{\HS} \in [\gamma(g)-\epsilon, \gamma(g)+\epsilon]$.
\end{enumerate} \end{lem}

\begin{proof} From Remark~\ref{remark:positive}\ref{remark:positive.2}, we know that \begin{displaymath}
	A \, \defeq \, \{ a \in \B(H) \mid \forall b \in \B(H) \colon \, \phi(ab) = \phi(ba) \}
\end{displaymath} constitutes a unital $C^{\ast}$-subalgebra of~$\B(H)$. By assumption, $\pi(G) \subseteq \U(A)$. Now, consider $E \in \Pfin(G)$ and $\epsilon \in (0,1]$. Recall that \begin{displaymath}
	\phi((\pi(gh)-\pi(g)\pi(h))^{\ast}(\pi(gh)-\pi(g)\pi(h))) \, = \, 0
\end{displaymath} for all $g,h \in E$. Hence, by Lemma~\ref{lemma:brown.ozawa}, there exist a non-zero finite-dimensional Hilbert space $H_{0}$ and a unital completely positive linear map $\sigma \colon A \to \B(H_{0})$ such that, for all $g,h \in E$,\footnote{For~\eqref{almost.invariance}, note that $\tr(\sigma(\pi(g)^{\ast}\pi(g))) = \tr(\sigma(1)) = \tr(1) = \dim(H_{0})$.} \begin{align}
	&\dim(H_{0})^{-1}\tr(\sigma(\pi(g))^{\ast}\sigma(\pi(g))) \in \left[1-\tfrac{\epsilon^{2}}{144},\,1+\tfrac{\epsilon^{2}}{144}\right], \label{almost.invariance} \\
	&\dim(H_{0})^{-1}\tr(\sigma((\pi(gh)-\pi(g)\pi(h))^{\ast}(\pi(gh)-\pi(g)\pi(h)))) \in \left[0,\tfrac{\epsilon^{2}}{4} \right], \label{approximate.homomorphism} \\
	&\left( \dim(H_{0})^{-1}\tr(\sigma((1-\pi(g))^{\ast}(1-\pi(g))))\right)^{1/2} \in \left[\gamma(g)-\tfrac{5\epsilon}{6},\,\gamma(g)+\tfrac{5\epsilon}{6}\right]. \label{approximate.freeness}
\end{align} Thanks to Theorem~\ref{theorem:stinespring}, we find a Hilbert space $K$, a unital ${}^{\ast}$-algebra homomorphism $\rho \colon A \to \B(K)$ and a linear isometry $v \colon H_{0} \to K$ such that \begin{equation}\label{dilation}
	\forall a \in A \colon \qquad \sigma(a) = v^{\ast}\rho(a)v .
\end{equation} Note that $v^{\ast}v = \id_{H_{0}}$ as $v$ is a linear isometry. From this we infer that $p \defeq vv^{\ast} \in \Pro(K)$ and $\rk(p) = \dim(H_{0})$. For every $g \in E$, \begin{align*}
	\Vert p\rho(\pi(g))p \Vert_{2}^{2} \, &= \, \Vert vv^{\ast}\rho(\pi(g))vv^{\ast} \Vert_{2}^{2} \, \stackrel{\eqref{dilation}}{=} \, \Vert v\sigma(\pi(g))v^{\ast} \Vert_{2}^{2} \, \stackrel{v\,\text{isom.}}{=} \, \Vert \sigma(\pi(g)) \Vert_{2}^{2} \\
	& = \, \tr(\sigma(\pi(g))^{\ast}\sigma(\pi(g))) \, \stackrel{\eqref{almost.invariance}}{\geq} \, \left(1-\tfrac{\epsilon^{2}}{144}\right)\!\dim(H_{0}) \, = \, \left(1-\tfrac{\epsilon^{2}}{144}\right)\!\rk(p)
\end{align*} and thus \begin{align*}
	\Vert (1-p)\rho(\pi(g))p \Vert_{2}^{2} \, &\stackrel{p \perp (1-p)}{=} \, \Vert \rho(\pi(g))p \Vert_{2}^{2} - \Vert p\rho(\pi(g))p \Vert_{2}^{2} \\
	&\stackrel{\rho(\pi(g)) \in \U(K)}{=} \, \Vert p \Vert_{2}^{2} - \Vert p\rho(\pi(g))p \Vert_{2}^{2} \\
	& \leq \, \rk(p) - \left(1-\tfrac{\epsilon^{2}}{144}\right)\!\rk(p) \, = \, \tfrac{\epsilon^{2}}{144}\rk(p) ,
\end{align*} so that Lemma~\ref{lemma:almost.commuting} asserts the existence of some $\tilde{\pi}(g) \in \U(K)$ such that \begin{align}
	&\tilde{\pi}(g)p = p\tilde{\pi}(g), \label{commuting} \\
	&\lVert (\tilde{\pi}(g)-\rho(\pi(g)))p \rVert_{2}^{2} \leq \tfrac{\epsilon^{2}}{36}\rk(p). \label{approximation}
\end{align} Define \begin{displaymath}
	\psi \colon \, G \, \longrightarrow \, \U(H_{0}) , \quad g \, \longmapsto \, \begin{cases}
		\, v^{\ast}\tilde{\pi}(g)v & \text{if } g \in E, \\
		\, 1 & \text{otherwise.}
	\end{cases}
\end{displaymath} Now, if $g,h \in E$ and $gh \in E$, then \begin{align*}
	&(\psi(gh)-\psi(g)\psi(h))^{\ast}(\psi(gh)-\psi(g)\psi(h)) \\
	&\qquad = \, (v^{\ast}\tilde{\pi}(gh)v-v^{\ast}\tilde{\pi}(g)\smallunderbrace{vv^{\ast}}_{=\,p}\tilde{\pi}(h)v)^{\ast}(v^{\ast}\tilde{\pi}(gh)v-v^{\ast}\tilde{\pi}(g)\smallunderbrace{vv^{\ast}}_{=\,p}\tilde{\pi}(h)v) \\
	&\qquad \stackrel{\eqref{commuting}}{=} \, (v^{\ast}\tilde{\pi}(gh)v-v^{\ast}\tilde{\pi}(g)\tilde{\pi}(h)\smallunderbrace{pv}_{=\,v})^{\ast}(v^{\ast}\tilde{\pi}(gh)v-v^{\ast}\tilde{\pi}(g)\tilde{\pi}(h)\smallunderbrace{pv}_{=\,v}) \\
	&\qquad = \, v^{\ast}(\tilde{\pi}(gh)-\tilde{\pi}(g)\tilde{\pi}(h))^{\ast}\smallunderbrace{vv^{\ast}}_{=\,p}(\tilde{\pi}(gh)-\tilde{\pi}(g)\tilde{\pi}(h))v \\
	&\qquad \stackrel{\eqref{commuting}}{=} \, v^{\ast}(\tilde{\pi}(gh)-\tilde{\pi}(g)\tilde{\pi}(h))^{\ast}(\tilde{\pi}(gh)-\tilde{\pi}(g)\tilde{\pi}(h))v
\end{align*} and \begin{allowdisplaybreaks}\begin{align*}
	&\lVert (\rho(\pi(gh))-\rho(\pi(g))\rho(\pi(h)))p \rVert_{2}^{2} \\
	&\qquad \stackrel{\rho\,\text{hom.}}{=} \, \tr(p\rho((\pi(gh)-\pi(g)\pi(h))^{\ast}(\pi(gh)-\pi(g)\pi(h)))p) \\
	&\qquad \stackrel{\eqref{dilation}}{=} \, \tr(v\sigma((\pi(gh)-\pi(g)\pi(h))^{\ast}(\pi(gh)-\pi(g)\pi(h)))v^{\ast}) \\
	&\qquad \stackrel{v\,\text{isom.}}{=} \, \tr(\sigma((\pi(gh)-\pi(g)\pi(h))^{\ast}(\pi(gh)-\pi(g)\pi(h)))) \\
	&\qquad \stackrel{\eqref{approximate.homomorphism}}{\leq} \, \tfrac{\epsilon^{2}}{4}\dim(H_{0}),
\end{align*}\end{allowdisplaybreaks} wherefore \begin{align*}
	&\lVert \psi(gh)-\psi(g)\psi(h) \rVert_{2} \\
	&\quad = \, \tr(v^{\ast}(\tilde{\pi}(gh)-\tilde{\pi}(g)\tilde{\pi}(h))^{\ast}(\tilde{\pi}(gh)-\tilde{\pi}(g)\tilde{\pi}(h))v)^{1/2} \\
	&\quad \stackrel{vv^{\ast}=p}{=} \, \tr(p(\tilde{\pi}(gh)-\tilde{\pi}(g)\tilde{\pi}(h))^{\ast}(\tilde{\pi}(gh)-\tilde{\pi}(g)\tilde{\pi}(h))p )^{1/2} \\
	&\quad = \, \lVert (\tilde{\pi}(gh)-\tilde{\pi}(g)\tilde{\pi}(h))p \rVert_{2} \\
	&\quad \leq \, \lVert (\tilde{\pi}(gh)-\rho(\pi(gh)))p \rVert_{2} + \lVert (\rho(\pi(gh))-\rho(\pi(g))\rho(\pi(h)))p \rVert_{2} \\
	&\quad \qquad + \lVert (\rho(\pi(g))\rho(\pi(h))-\rho(\pi(g))\tilde{\pi}(h))p \rVert_{2} + \lVert (\rho(\pi(g))\tilde{\pi}(h)-\tilde{\pi}(g)\tilde{\pi}(h))p \rVert_{2} \\
	&\quad \stackrel{\eqref{approximation}}{\leq} \, \left(\tfrac{\epsilon^{2}}{36}\rk(p)\right)^{1/2} + \left(\tfrac{\epsilon^{2}}{4}\dim(H_{0})\right)^{1/2} \\
	&\quad \qquad + \lVert \rho(\pi(g))(\rho(\pi(h))-\tilde{\pi}(h))p \rVert_{2} + \lVert (\rho(\pi(g))-\tilde{\pi}(g))\!\!\smallunderbrace{\tilde{\pi}(h)p}_{\stackrel{\eqref{commuting}}{=}\,p\tilde{\pi}(h)} \rVert_{2} \\
	&\quad = \, \left(\tfrac{\epsilon}{6} + \tfrac{\epsilon}{2} \right)\dim(H_{0})^{1/2} + \lVert (\rho(\pi(h))-\tilde{\pi}(h))p \rVert_{2} + \lVert (\rho(\pi(g))-\tilde{\pi}(g)) p\rVert_{2} \\
	&\quad \stackrel{\eqref{approximation}}{\leq} \, \tfrac{2\epsilon}{3}\dim(H_{0})^{1/2} + 2\!\left(\tfrac{\epsilon^{2}}{36}\rk(p)\right)^{1/2} \, = \, \epsilon \dim(H_{0})^{1/2} ,
\end{align*} that is, $\lVert \psi(gh)-\psi(g)\psi(h) \rVert_{\HS} \leq \epsilon$. Moreover, for every $g \in E$, \begin{align*}
	&(1-\psi(g))^{\ast}(1-\psi(g)) \, = \, (v^{\ast}v-v^{\ast}\tilde{\pi}(g)v)^{\ast}(v^{\ast}v-v^{\ast}\tilde{\pi}(g)v) \\
	&\qquad = \, v^{\ast}(1-\tilde{\pi}(g))^{\ast}\smallunderbrace{vv^{\ast}}_{=\,p}(1-\tilde{\pi}(g))v \, \stackrel{\eqref{commuting}}{=} \, v^{\ast}(1-\tilde{\pi}(g))^{\ast}(1-\tilde{\pi}(g))\!\smallunderbrace{pv}_{=\,v},
\end{align*} \begin{align*}
	\lVert (1-\rho(\pi(g)))p - (1-\tilde{\pi}(g))p \rVert_{2} \, &= \, \lVert (\tilde{\pi}(g)-\rho(\pi(g)))p \rVert_{2} \\
	&\stackrel{\eqref{approximation}}{\leq} \, \left(\tfrac{\epsilon^{2}}{36}\rk(p)\right)^{1/2} \, = \, \tfrac{\epsilon}{6}\dim(H_{0})^{1/2} ,
\end{align*} and \begin{align*}
	\lVert (1-\rho(\pi(g)))p \rVert_{2} \, &\stackrel{\rho\,\text{hom.}}{=} \, \tr(p\rho((1-\pi(g))^{\ast}(1-\pi(g)))p)^{1/2} \\
	&\stackrel{\eqref{dilation}}{=} \, \tr(v\sigma((1-\pi(g))^{\ast}(1-\pi(g)))v^{\ast})^{1/2} \\
	&\stackrel{v\,\text{isom.}}{=} \, \tr(\sigma((1-\pi(g))^{\ast}(1-\pi(g))))^{1/2} \\
	&\stackrel{\eqref{approximate.freeness}}{\in} \, \left[ \left(\gamma(g)-\tfrac{5\epsilon}{6}\right)\dim(H_{0})^{1/2}, \, \left(\gamma(g)+\tfrac{5\epsilon}{6}\right)\dim(H_{0})^{1/2} \right]\!,
\end{align*} whence \begin{align*}
	\lVert 1-\psi(g) \rVert_{2} \, &= \, \tr(v^{\ast}(1-\tilde{\pi}(g))^{\ast}(1-\tilde{\pi}(g))v)^{1/2} \\
	& \stackrel{vv^{\ast}=p}{=} \tr(p(1-\tilde{\pi}(g))^{\ast}(1-\tilde{\pi}(g))p)^{1/2} \, = \, \lVert (1-\tilde{\pi}(g))p \rVert_{2} \\
	&\in \, \left[ \left(\gamma(g)-\tfrac{\epsilon}{6}-\tfrac{5\epsilon}{6}\right)\dim(H_{0})^{1/2} , \left(\gamma(g)+\tfrac{\epsilon}{6}+\tfrac{5\epsilon}{6}\right)\dim(H_{0})^{1/2} \right] \\
	&= \, \left[ (\gamma(g)-\epsilon)\dim(H_{0})^{1/2} , \, (\gamma(g)+\epsilon)\dim(H_{0})^{1/2} \right]
\end{align*} and thus $\Vert 1-\psi(g) \Vert_{\HS} \in [\gamma(g)-\epsilon, \gamma(g)+\epsilon]$. \end{proof}			
		
\begin{remark}\label{remark:dank.pestov} The main idea underlying the proof of Lemma~\ref{lemma:key}, i.e., the combined application of Lemma~\ref{lemma:brown.ozawa} and Theorem~\ref{theorem:stinespring}, is inspired by the first step in the proof of~\cite[Proposition~3.1(ii), p.~167--168]{AraLledo}. The authors are grateful to Vladimir Pestov for bringing the work of Ara and Lledó~\cite{AraLledo} to their attention. \end{remark}	

We proceed to this section's main result. Note that, according to Proposition~\ref{proposition:near.rep}\ref{proposition:near.rep.regular}, the condition specified in~\ref{theorem:amenable.trace.d1} of the following theorem is equivalent to orthogonality in the inner product space naturally associated with a state.	
		
\begin{thm}\label{theorem:amenable.trace} Let $G$ be a group. The following are equivalent. \begin{enumerate}
	\item\label{theorem:amenable.trace.a} $G$ is hyperlinear.
	\item\label{theorem:amenable.trace.b1} There exist a Hilbert space $H$, a probability charge $\nu$ on $\Borel(\Sph(H))$ and a $\nu$-near representation $\pi \colon G \to \U(H)$ such that, for all $g \in G\setminus \{ e \}$ and $\epsilon \in \R_{>0}$, \begin{displaymath}
			\qquad \nu(\{ x \in \Sph(H) \mid \lvert \langle x,\pi(g)x \rangle \rvert \leq \epsilon \}) \, = \, 1 .
		\end{displaymath}
	\item\label{theorem:amenable.trace.b2} There exist a Hilbert space $H$, a probability charge $\nu$ on $\Borel(\Sph(H))$ and a $\nu$-near representation $\pi \colon G \to \U(H)$ such that \begin{displaymath}
			\qquad \forall g \in G\setminus \{ e \} \ \exists \gamma \in \R_{>0} \colon \ \ \nu(\{ x \in \Sph(H) \mid \lVert x - \pi(g)x \rVert \geq \gamma \}) = 1 .
		\end{displaymath}
	\item\label{theorem:amenable.trace.c1} There exist a Hilbert space $H$, a state $\mu$ on $\UCB(\Sph(H))$ and a $\mu$-near representation $\pi \colon G \to \U(H)$ such that \begin{displaymath}
			\qquad \forall g \in G\setminus \{ e\} \colon \quad  \mu(x \mapsto \vert\langle x,\pi(g)x \rangle\vert) = 0 .
		\end{displaymath}
	\item\label{theorem:amenable.trace.c2} There exist a Hilbert space $H$, a state $\mu$ on $\UCB(\Sph(H))$ and a $\mu$-near representation $\pi \colon G \to \U(H)$ such that \begin{displaymath}
		\qquad \forall g \in G\setminus \{ e\} \colon \quad \mu(x \mapsto \lVert x - \pi(g)x \rVert) > 0 .
	\end{displaymath}
	\item\label{theorem:amenable.trace.d1} There exist a Hilbert space $H$, a state $\phi$ on $\B(H)$ and a $\phi$-near representation $\pi \colon G \to \U(H)$ such that \begin{displaymath}
			\qquad \forall g \in G\setminus \{ e\} \colon \quad \phi(\pi(g)) = 0 .
		\end{displaymath}
	\item\label{theorem:amenable.trace.d2} There exist a Hilbert space $H$, a state $\phi$ on $\B(H)$ and a $\phi$-near representation $\pi \colon G \to \U(H)$ such that \begin{displaymath}
		\qquad \forall g \in G\setminus \{ e\} \colon \quad \phi((1-\pi(g))^{\ast}(1-\pi(g))) > 0 .
	\end{displaymath}
\end{enumerate} \end{thm}

\begin{proof} \ref{theorem:amenable.trace.a}$\Longrightarrow$\ref{theorem:amenable.trace.b1}. Suppose that $G$ is hyperlinear. Consider the directed set $(I,{\preceq})$ defined by $I \defeq \Pfin(G) \times \N_{>0}$ and \begin{displaymath}
	(E,k) \preceq (E',k') \quad :\Longleftrightarrow \quad E \subseteq E' \ \wedge \ k \leq k' \qquad ((E,k),(E',k') \in I).
\end{displaymath} According to Proposition~\ref{proposition:hyperlinear.groups}, for each $i = (E,k) \in I$, we may select a non-zero finite-dimensional Hilbert space $H_{i}$ and a map $\pi_{i} \colon G \to \U(H_{i})$ such that \begin{align}
	&\forall g,h \in E \colon \ \ \nu_{H_{i}}\!\left(\left\{ x \in \Sph(H_{i}) \left\vert \lVert \pi_{i}(gh)x - \pi_{i}(g)\pi_{i}(h)x \rVert \leq \tfrac{1}{k} \right\}\right) \right. \! \geq 1 - \tfrac{1}{k} , \label{almost.homomorphism} \\
	&\forall g \in E\setminus \{ e \} \colon \ \ \nu_{H_{i}}\!\left(\left\{ x \in \Sph(H_{i}) \left\vert \, \lvert \langle x,\pi_{i}(g)x \rangle \rvert \leq \tfrac{1}{k} \right\}\right)\right.\! \geq 1 - \tfrac{1}{k} . \label{almost.freeness}
\end{align} Consider the Hilbert space sum $H \defeq \bigoplus_{i \in I} H_{i}$. For each $i \in I$, observe that \begin{displaymath}
	\Borel (\Sph(H_{i})) \, = \, \{ B \cap \Sph(H_{i}) \mid B \in \Borel(\Sph(H)) \} .
\end{displaymath} Choose any (necessarily non-principal) ultrafilter $\mathcal{F}$ on $I$ containing \begin{displaymath}
	\{ \{ (E',k') \in I \mid (E,k) \preceq (E',k') \} \mid (E,k) \in I \} ,
\end{displaymath} and consider the probability charge \begin{displaymath}
	\nu \colon \, \Borel(\Sph(H)) \, \longrightarrow \, [0,1], \quad B \, \longmapsto \, \lim\nolimits_{i\to \mathcal{F}} \nu_{H_{i}}(B \cap \Sph(H_{i})) .
\end{displaymath} Now, let us consider the well-defined map $\pi \colon G \to \U(H)$ given by \begin{displaymath}
	\pi(g)(x) \, \defeq \, (\pi_{i}(g)x_{i})_{i \in I} \qquad (g \in G, \, x \in H) .
\end{displaymath} If $g \in G$, then \begin{align*}
	&(\pi(g).\nu)(B) \, = \, \nu(\pi(g)^{\ast}B) \, = \, \lim\nolimits_{i\to \mathcal{F}} \nu_{H_{i}}(\pi(g)^{\ast}B \cap \Sph(H_{i})) \\
	& \qquad = \, \lim\nolimits_{i\to \mathcal{F}} \nu_{H_{i}}(\pi(g)^{\ast}B \cap \pi(g)^{\ast}\Sph(H_{i})) \! \, = \, \lim\nolimits_{i\to \mathcal{F}} \nu_{H_{i}}(\pi(g)^{\ast}(B \cap \Sph(H_{i}))) \\
	& \qquad = \, \lim\nolimits_{i\to \mathcal{F}} \nu_{H_{i}}(\pi_{i}(g)^{\ast}(B \cap \Sph(H_{i}))) \! \, = \, \lim\nolimits_{i\to \mathcal{F}} \nu_{H_{i}}(B \cap \Sph(H_{i})) \, = \, \nu(B) 
\end{align*} for every $B \in \Borel(\Sph(H))$, i.e., $\pi(g).\nu = \nu$. Furthermore, if $g,h \in G$ and $\epsilon \in \R_{>0}$, then \begin{align*}
	&\nu(\{ x \in \Sph(H) \mid \lVert \pi(gh)x - \pi(g)\pi(h)x \rVert \leq \epsilon \}) \\
	&\qquad = \, \lim\nolimits_{i \to \mathcal{F}} \nu_{H_{i}}(\{ x \in \Sph(H_{i}) \mid \lVert \pi(gh)x - \pi(g)\pi(h)x \rVert \leq \epsilon \}) \\
	&\qquad = \, \lim\nolimits_{i \to \mathcal{F}} \nu_{H_{i}}(\{ x \in \Sph(H_{i}) \mid \lVert \pi_{i}(gh)x - \pi_{i}(g)\pi_{i}(h)x \rVert \leq \epsilon \}) \, \stackrel{\eqref{almost.homomorphism}}{=} \, 1 ,
\end{align*} wherefore $\pi$ is a $\nu$-near representation. Finally, if $g \in G\setminus \{ e \}$ and $\epsilon \in \R_{>0}$, then \begin{align*}
	&\nu(\{ x \in \Sph(H) \mid \lvert \langle x,\pi(g)x \rangle \rvert \leq \epsilon \}) \, = \, \lim\nolimits_{i \to \mathcal{F}} \nu_{H_{i}}(\{ x \in \Sph(H_{i}) \mid \lvert \langle x,\pi(g)x \rangle \rvert \leq \epsilon \}) \\
	& \qquad \qquad \qquad = \, \lim\nolimits_{i \to \mathcal{F}} \nu_{H_{i}}(\{ x \in \Sph(H_{i}) \mid \lvert \langle x,\pi_{i}(g)x \rangle \rvert \leq \epsilon \}) \, \stackrel{\eqref{almost.freeness}}{=} \, 1 .
\end{align*}

\ref{theorem:amenable.trace.b1}$\Longrightarrow$\ref{theorem:amenable.trace.b2}. This implication follows from the fact that, in any Hilbert space $H$, \begin{displaymath}
	\forall x,y \in \Sph(H) \colon \quad \Vert x-y \Vert \, = \, \sqrt{2-2\Re\langle x,y \rangle} \, \geq \, \sqrt{2-2\vert\langle x,y \rangle\vert} .
\end{displaymath}

\ref{theorem:amenable.trace.d1}$\Longrightarrow$\ref{theorem:amenable.trace.d2}. If $H$ is a Hilbert space and $\phi$ is a state on $\B(H)$, then \begin{displaymath}
	\phi((1-u)^{\ast}(1-u)) \, = \, 2-\phi(u)-\phi(u^{\ast}) \, \stackrel{\ref{remark:positive}\ref{remark:positive.1}}{=} \, 2-2\Re \phi(u) 
\end{displaymath} for every $u \in \U(H)$. Hence, \ref{theorem:amenable.trace.d1} implies~\ref{theorem:amenable.trace.d2}.

\ref{theorem:amenable.trace.b1}$\Longrightarrow$\ref{theorem:amenable.trace.c1}; \ref{theorem:amenable.trace.b2}$\Longrightarrow$\ref{theorem:amenable.trace.c2}. Let $H$ be a Hilbert space, let $\nu$ be a probability charge on $\Borel(\Sph(H))$, and let $\pi$ be a $\nu$-near representation of $G$ on $H$. Then $\mu \defeq \nu_{\circ}$ is a state on $\UCB(\Sph(H))$ by Remark~\ref{remark:means.from.charges}, and $\pi$ is a $\mu$-near representation by Lemma~\ref{lemma:near.representations}\ref{lemma:near.representations.1}. Now, if $g \in G$ is such that \begin{displaymath}
	\forall \epsilon \in \R_{>0} \colon \quad \nu(\underbrace{\{ x \in \Sph(H) \mid \vert \langle x,\pi(g)x\rangle \vert \leq \epsilon \}}_{\eqdef\,B(\epsilon)}) = 1 ,
\end{displaymath} then \begin{align*}
	&\mu(x \mapsto \vert \langle x,\pi(g)x\rangle \vert) \, = \, \int \vert \langle x,\pi(g)x\rangle \vert \, \mathrm{d}\nu(x) \\
	&\qquad \leq \, \int \epsilon \chi_{B(\epsilon)} + \chi_{{\Sph(H)} \setminus B(\epsilon)} \, \mathrm{d}\nu \, = \, \epsilon \nu(B(\epsilon)) + \nu({\Sph(H)} \setminus B(\epsilon)) \, = \, \epsilon
\end{align*} for every $\epsilon \in \R_{>0}$, that is, $\mu(x \mapsto \vert \langle x,\pi(g)x\rangle \vert) = 0$. This establishes the implication \ref{theorem:amenable.trace.b1}$\Longrightarrow$\ref{theorem:amenable.trace.c1}. On the other hand, if $g \in G$ and $\gamma \in \R_{\geq 0}$ satisfy \begin{displaymath}
	\nu(\underbrace{\{ x \in \Sph(H) \mid \lVert x - \pi(g)x \rVert \geq \gamma \}}_{\eqdef\,C}) \, = \, 1 ,
\end{displaymath} then \begin{align*}
	\mu( x \mapsto \lVert x-\pi(g)x \rVert) \, &= \, \int \lVert x-\pi(g)x \rVert \, \mathrm{d}\nu(x) \, \geq \, \int \gamma \chi_{C} \, \mathrm{d}\nu \, = \, \gamma\nu(C) \, = \, \gamma .
\end{align*} This observation readily entails the implication \ref{theorem:amenable.trace.b2}$\Longrightarrow$\ref{theorem:amenable.trace.c2}.

\ref{theorem:amenable.trace.c1}$\Longrightarrow$\ref{theorem:amenable.trace.d1}; \ref{theorem:amenable.trace.c2}$\Longrightarrow$\ref{theorem:amenable.trace.d2}. Let $H$ be a Hilbert space, $\mu$ be a state on~$\UCB(\Sph(H))$, and $\pi$ be a $\mu$-near representation of $G$ on $H$. Then Lemma~\ref{lemma:states.from.means} asserts that $\phi \defeq \mu_{\bullet}$ is a state on $\B(H)$. Moreover, $\pi$ is a $\phi$-near representation by Lemma~\ref{lemma:near.representations}\ref{lemma:near.representations.2}. Note that, for every $g \in G$, \begin{displaymath}
	\vert \phi(\pi(g)) \vert \, = \, \lvert \mu(x \mapsto \langle x,\pi(g)x \rangle) \rvert \, \leq \, \mu(x \mapsto \vert \langle x,\pi(g)x \rangle\vert) .
\end{displaymath} Therefore, \ref{theorem:amenable.trace.c1} implies~\ref{theorem:amenable.trace.d1}. Finally, for every $g \in G$, we see that \begin{align*}
	\phi((1-\pi(g))^{\ast}(1-\pi(g))) \, &= \, \mu (x \mapsto \langle x,(1-\pi(g))^{\ast}(1-\pi(g))x \rangle) \\
	& = \, \mu (x \mapsto \langle (1-\pi(g))x,(1-\pi(g))x \rangle) \\
	& = \, \mu \!\left(x \mapsto \Vert x-\pi(g)x \Vert^{2}\right)\! \\
	& \geq \, \mu(x \mapsto \Vert x-\pi(g)x \Vert)^{2} ,
\end{align*} where the last step is an application of the Cauchy--Schwartz inequality. This proves the implication \ref{theorem:amenable.trace.c2}$\Longrightarrow$\ref{theorem:amenable.trace.d2}.

\ref{theorem:amenable.trace.d2}$\Longrightarrow$\ref{theorem:amenable.trace.a}. Let $H$ be a Hilbert space, let $\phi$ be a state on $\B(H)$, and let $\pi \colon G \to \U(H)$ be a $\phi$-near representation such that $\phi((1-\pi(g))^{\ast}(1-\pi(g))) > 0$ for every $g \in G\setminus \{ e \}$. Then Lemma~\ref{lemma:key} entails that the function \begin{displaymath}
	\gamma \colon \, G\setminus \{ e \} \, \longrightarrow \, \R_{>0}, \quad g \, \longmapsto \, \tfrac{1}{2}\phi((1-\pi(g))^{\ast}(1-\pi(g)))^{1/2}
\end{displaymath} witnesses the hyperlinearity of $G$. \end{proof}		

We record another consequence of the concentration of spherical measures.

\begin{remark}\label{remark:ring.homomorphism} In the conditions~\ref{theorem:amenable.trace.c1} and~\ref{theorem:amenable.trace.c2} of Theorem~\ref{theorem:amenable.trace}, one may additionally require that $\mu$ be a ring homomorphism from $\UCB(\Sph(H))$ to $\C$. Due to Remark~\ref{remark:hyperlinear.groups}, we may arrange for the family $(H_{i},\pi_{i})_{i \in I}$ selected in the proof of the implication \ref{theorem:amenable.trace.a}$\Longrightarrow$\ref{theorem:amenable.trace.b1} to moreover satisfy \begin{equation}\label{eq:asymptotic.dimension}
	\dim(H_{i})\, \longrightarrow \, \infty \quad (i \to (I,{\preceq})) .
\end{equation} Concerning the state $\mu = \nu_{\circ}$ constructed from the probability charge $\nu = \lim_{i\to\mathcal{F}} \nu_{H_{i}}$ in the proof of the implication \ref{theorem:amenable.trace.b1}$\Longrightarrow$\ref{theorem:amenable.trace.c1}, we note that \begin{displaymath}
	\mu(f) \, = \, \int f \, \mathrm{d}\nu \, = \, \lim\nolimits_{i \to \mathcal{F}} \int f \, \mathrm{d}\nu_{H_{i}}
\end{displaymath} for every $f \in \UCB(\Sph(H))$. We claim that $\UCB(\Sph(H),\R) \to \R, \, f \mapsto \mu(f)$ is a ring homomorphism. For this purpose, according to~\cite[Proposition~2.8]{SchneiderGAFA}, it suffices to check that, for every $f \in \UCB(\Sph(H),\R)$, \begin{equation}\label{eq:ring:homomorphism}
	\int {\left( f - \int f \, \mathrm{d}\nu_{H_{i}} \right)}^{2} \, \mathrm{d}\nu_{H_{i}} \, \longrightarrow \, 0 \quad (i \to (I,{\preceq})) .
\end{equation} To this end, let $f \in \UCB(\Sph(H),\R)$ and $\epsilon \in \R_{>0}$ and define $\epsilon_{0}\defeq \tfrac{\epsilon}{16(\Vert f \Vert_{\infty}+1)}$. Then we find $\delta \in \R_{>0}$ such that \begin{displaymath}
	\forall x,y \in \Sph(H) \colon \quad \Vert x-y \Vert \leq \delta \ \, \Longrightarrow \ \, \vert f(x)-f(y) \vert \leq \epsilon_{0} .
\end{displaymath} Now, if $x,y \in \Sph(H)$, then either $\Vert x-y \Vert \leq \delta$ and so $\lvert f(x) - f(y) \rvert \leq \epsilon_{0}$, or $\Vert x-y \Vert > \delta$ and therefore \begin{align*}
	\vert f(x) - f(y) \vert \, \leq \, 2\Vert f \Vert_{\infty} \, \leq \, \underbrace{2\Vert f \Vert_{\infty}\delta^{-1}}_{\eqdef \, \ell}\Vert x-y \Vert .
\end{align*} Hence, for all $x,y \in \Sph(H)$, \begin{displaymath}
	\lvert f(x) - f(y) \rvert \, \leq \, \max\{ \ell\Vert x-y \Vert, \,\epsilon_{0} \} \, \leq \, \ell\Vert x-y \Vert + \epsilon_{0} .
\end{displaymath} Consequently, \cite[Lemma~5.2]{SchneiderThomCMH} asserts the existence of some $f'' \in \Lip_{\ell}(\Sph(H),\R)$ such that $\lVert f - {f''} \rVert_{\infty} \leq \epsilon_{0}$. It follows that \begin{displaymath}
	f' \, \defeq \, (f'' \wedge \Vert f \Vert_{\infty}) \vee (-\Vert f \Vert_{\infty}) \, \in \, \Lip_{\ell}(\Sph(H),\R) ,
\end{displaymath} $\Vert f' \Vert_{\infty} \leq \Vert f \Vert_{\infty}$ and $\lVert f - {f'} \rVert_{\infty} \leq \epsilon_{0}$. For every $i \in I$, we conclude that \begin{align*}
	&\left\lVert \left( f-\int f\, \mathrm{d}\nu_{H_{i}} \right)^{2} - \left( f'-\int f'\, \mathrm{d}\nu_{H_{i}} \right)^{2} \right\rVert_{\infty} \\
	& \qquad \leq \, \left\lVert f+f'-\int f+f'\, \mathrm{d}\nu_{H_{i}} \right\rVert_{\infty} \left\lVert f-{f'}- \int f-f' \, \mathrm{d}\nu_{H_{i}} \right\rVert_{\infty} \\
	& \qquad \leq 8\Vert f \Vert_{\infty}\Vert f-{f'} \Vert_{\infty} \, \leq \, \tfrac{\epsilon}{2} .
\end{align*} Due to~\eqref{eq:asymptotic.dimension}, we find $i_{0} \in I$ such that \begin{displaymath}
	\forall i \in I \colon \quad i_{0} \preceq i \ \, \Longrightarrow \ \, 8\Vert f \Vert_{\infty}^{2}\exp\!\left( -\tfrac{\epsilon(2\dim(H_{i})-1)}{8\ell^{2}} \right) \, \leq \, \tfrac{\epsilon}{4} .
\end{displaymath} Then, for every $i \in I$ with $i_{0} \preceq i$, \begin{align*}
	&\int {\left( f - \int f \, \mathrm{d}\nu_{H_{i}} \right)}^{2}\, \mathrm{d}\nu_{H_{i}} \, \leq \, \int {\left( f' - \int f' \, \mathrm{d}\nu_{H_{i}} \right)}^{2}\, \mathrm{d}\nu_{H_{i}} + \tfrac{\epsilon}{2} \\
	&\qquad \leq \, 4\Vert f \Vert_{\infty}^{2}\nu_{H_{i}\!}\left(\left\{ x \in \Sph(H_{i}) \left\vert \, \left\lvert f'(x) - \int f' \, \mathrm{d}\nu_{H_{i}} \right\rvert > \tfrac{\sqrt{\epsilon}}{2} \right\}\right)\!\right. + \tfrac{\epsilon}{4} + \tfrac{\epsilon}{2} \\
	&\qquad \stackrel{\ref{theorem:spherical.concentration}}{\leq} \, 8\Vert f \Vert_{\infty}^{2}\exp\!\left( -\tfrac{\epsilon(2\dim(H_{i})-1)}{8\ell^{2}} \right) + \tfrac{3\epsilon}{4} \, \leq \, \tfrac{\epsilon}{4} + \tfrac{3\epsilon}{4} \, = \, \epsilon .
\end{align*} This proves~\eqref{eq:ring:homomorphism}. Therefore, $\UCB(\Sph(H),\R) \to \R, \, f \mapsto \mu(f)$ is a ring homomorphism according to~\cite[Proposition~2.8]{SchneiderGAFA}. It follows that \begin{align*}
	\mu(fg) \, &= \, \mu(((\Re f)(\Re g) - (\Im f)(\Im g)) + i((\Re f)(\Im g) + (\Im f)(\Re g))) \\
	& = \, ((\Re \mu(f))(\Re \mu(g)) - (\Im \mu(f))(\Im \mu(g))) \\
	& \qquad \qquad \qquad \qquad + i((\Re \mu(f))(\Im \mu(g)) + (\Im \mu(f))(\Re \mu(g))) \\
	&= \, \mu(f)\mu(g)
\end{align*} for all $f,g \in \UCB(\Sph(H))$, i.e., $\mu$ is a ring homomorphism from $\UCB(\Sph(H))$ to $\C$. \end{remark}
		
\section{Analogue of the Elek--Szabó theorem}\label{section:elek.szabo}

The goal of this section is to provide a direct counterpart to the Elek--Szabó theorem for hyperlinear groups. For this purpose, let us review some background from~\cite{ElekSzabo}. Let $X$ be a set and consider the natural action of the full symmetric group $\Sym(X)$ on the set of probability charges on $\Pow(X)$ given by \begin{displaymath}
	(g.\mu)(A) \, \defeq \, \mu\!\left( g^{-1}A \right) \qquad (A \in \Pow(X))
\end{displaymath} for every $g \in G$ and every probability charge $\mu$ on $\Pow(X)$. Given a probability charge $\mu$ on $\Pow(X)$, a \emph{$\mu$-near action}~\cite{ElekSzabo,PestovKwiatkowska}\footnote{In~\cite{ElekSzabo}, this was introduced as an \emph{almost action}. Our choice of terminology aligns with~\cite{PestovKwiatkowska}.} of a group $G$ on $X$ is a map $\pi \colon G \to \Sym(X)$ such that $\mu$ is $\pi(G)$-invariant and, for all $g,h \in G$, \begin{displaymath}
	\mu(\{ x \in X \mid \pi(gh)x = \pi(g)\pi(h)x \}) \, = \, 1 .
\end{displaymath}

\begin{thm}[Elek--Szabó~{\cite[Corollary~4.2]{ElekSzabo}}]\label{theorem:elek.szabo.original} A group $G$ is sofic if and only if there exist a set~$X$, a probability charge $\mu$ on $\Pow(X)$ and a $\mu$-near action $\pi \colon G \to \Sym(X)$ such that, for every $g \in G\setminus \{ e\}$, \begin{displaymath}
	\mu (\{ x \in X \mid x \ne \pi(g)x \}) \, = \, 1.
\end{displaymath} \end{thm}
	
Inspired by Theorem~\ref{theorem:amenable.basis} and Theorem~\ref{theorem:elek.szabo.original}, we make the following definition.

\begin{definition}\label{definition:near.representation} Let $G$ be a group, let $X$ be a set and let $\mu$ be a probability charge on $\Pow(X)$. A map $\pi \colon G \to \U(\ell^{2}(X))$ is said to be a \emph{$\mu$-near representation} if $\mu_{\bullet}$ is $\pi(G)$-invariant and, for all $g,h \in G$ and $\epsilon \in \R_{>0}$, \begin{displaymath}
	\mu(\{ x \in X \mid \lVert \pi(gh)x - \pi(g)\pi(h)x \rVert \leq \epsilon \}) \, = \, 1 .
\end{displaymath} \end{definition}

Let us explain more concretely how the definition above relates to the notion of an amenable near action due to Elek and Szabó. Given a set $X$, we will consider the group embedding $\iota \colon \Sym (X) \to \U(\ell^{2}(X))$ defined by \begin{displaymath}
	\iota(g)(f) \, \defeq \, f \circ {g^{-1}} \qquad \left(g \in \Sym(X), \, f \in \ell^{2}(X)\right) ,
\end{displaymath} and we will view $X$ as a subset of $\ell^{2}(X)$, where each $x \in X$ is identified with the corresponding $\delta_{x} \in \ell^{2}(X)$.

\begin{lem}\label{lemma:elek.szabo.comparison} Let $X$ be a set and let $\mu$ be a probability charge on $\Pow(X)$. \begin{enumerate}
	\item\label{lemma:elek.szabo.comparison.1} Let $g \in \Sym(X)$. Then $g.\mu = \mu$ if and only if $\iota(g).(\mu_{\bullet}) = \mu_{\bullet}$.
	\item\label{lemma:elek.szabo.comparison.2} Let $G$ be a group and let $\pi \colon G \to \Sym(X)$. Then $\pi$ is a $\mu$-near action if and only if $\iota \circ \pi$ is a $\mu$-near representation.
	\item\label{lemma:elek.szabo.comparison.3} Let $g \in \Sym(X)$. Then \begin{align*}
				\qquad &\mu (\{ x \in X \mid gx \ne x \}) = 1 \\
				&\qquad \Longleftrightarrow \quad \forall \epsilon \in \R_{>0}\colon \ \mu(\{ x \in X \mid \lvert \langle x,\iota(g)x \rangle \rvert \leq \epsilon \}) = 1 \\
				&\qquad \Longleftrightarrow \quad \exists \gamma \in \R_{>0} \colon \ \mu(\{ x \in X \mid \Vert x - \iota(g)x \Vert \geq \gamma \}) = 1.
			\end{align*}
\end{enumerate} \end{lem}

\begin{proof} \ref{lemma:elek.szabo.comparison.1} ($\Longrightarrow$) Suppose that $g.\mu = \mu$. Then \begin{displaymath}
	\int {\chi_{A}} \circ g \, \mathrm{d}\mu \, = \, \int {\chi_{g^{-1}A}} \, \mathrm{d}\mu \, = \, \mu\!\left( g^{-1}A \right) \, = \, (g.\mu)(A) \, = \, \mu(A) \, = \, \int {\chi_{A}} \, \mathrm{d}\mu
\end{displaymath} for every $A \subseteq X$. Since the linear space of $\{ \chi_{A} \mid A \subseteq X \}$ is $\Vert \cdot \Vert_{\infty}$-dense in $\ell^{\infty}(X)$ and $\ell^{\infty}(X) \to \C, \, f \mapsto \int f \, \mathrm{d}\mu$ is continuous, it follows that $\int f \circ g \, \mathrm{d}\mu = \int f \, \mathrm{d}\mu$ for every $f \in \ell^{\infty}(X)$. In particular, \begin{align*}
	(\iota(g).(\mu_{\bullet}))(a) \, &= \, \mu_{\bullet}(\iota(g)^{\ast}a\iota(g)) \, = \, \int \langle x,\iota(g)^{\ast}a\iota(g)x \rangle \, \mathrm{d}\mu(x) \\
		& = \, \int \langle gx,agx \rangle \, \mathrm{d}\mu(x) \, = \, \int \langle x,ax \rangle \, \mathrm{d}\mu(x) \, = \, \mu_{\bullet}(a)
\end{align*} for all $a \in \B(\ell^{2}(X))$, that is, $\iota(g).(\mu_{\bullet}) = \mu_{\bullet}$.
	
($\Longleftarrow$) Suppose that $\iota(g).(\mu_{\bullet}) = \mu_{\bullet}$. In order to see that $g.\mu = \mu$, let $A \subseteq X$. Consider the orthogonal projection $p \colon \ell^{2}(X) \to \ell^{2}(X)$ onto the closed linear subspace spanned by $A \subseteq \ell^{2}(X)$, i.e., the unique bounded linear operator such that \begin{displaymath}
	px \, = \, \begin{cases}
			\, x & \text{if } x \in A, \\
			\, 0 & \text{otherwise}
		\end{cases}
\end{displaymath} for every $x \in X$. Then \begin{align*}
	(g.\mu)(A) \, &= \, \mu\!\left(g^{-1}A\right) \, = \, \int {\chi_{g^{-1}A}} \, \mathrm{d}\mu  \, = \, \int \chi_{A}(gx) \, \mathrm{d}\mu(x) \, = \, \int \langle gx,pgx \rangle \, \mathrm{d}\mu(x) \\
		& = \, \int \langle x,\iota(g)^{\ast}p\iota(g)x \rangle \, \mathrm{d}\mu(x) \, = \, \mu_{\bullet}(\iota(g)^{\ast}p\iota(g)) \, = \, (\iota(g).(\mu_{\bullet}))(p) \\
		& = \, \mu_{\bullet}(p) \, = \, \int \langle x,px \rangle \, \mathrm{d}\mu(x) \, = \, \int \chi_{A}(x) \, \mathrm{d}\mu(x) \, = \, \mu(A) ,
\end{align*} as desired.
	
\ref{lemma:elek.szabo.comparison.2}+\ref{lemma:elek.szabo.comparison.3} This is a consequence of the fact that $X$ is a $\iota(g)$-invariant orthonormal subset of $\ell^{2}(X)$. \end{proof}

Based on the observations of Lemma~\ref{lemma:elek.szabo.comparison}, we formulate the hyperlinear analogue of the Elek--Szabó theorem.

\begin{thm}\label{theorem:elek.szabo} Let $G$ be a group. The following are equivalent. \begin{enumerate}
	\item\label{theorem:elek.szabo.1} $G$ is hyperlinear.
	\item\label{theorem:elek.szabo.2} There exist a set $X$, a probability charge $\mu$ on $\Pow(X)$ and a $\mu$-near representation $\pi \colon G \to \U(\ell^{2}(X))$ such that, for all $g \in G\setminus \{ e\}$ and $\epsilon \in \R_{>0}$, \begin{displaymath}
				\qquad \mu(\{ x \in X \mid \lvert \langle x,\pi(g)x \rangle \rvert \leq \epsilon \}) \, = \, 1.
			\end{displaymath}
	\item\label{theorem:elek.szabo.3} There exist a set $X$, a probability charge $\mu$ on $\Pow(X)$ and a $\mu$-near representation $\pi \colon G \to \U(\ell^{2}(X))$ such that \begin{displaymath}
				\qquad \forall g \in G\setminus \{ e\} \ \exists \gamma \in \R_{>0} \colon \ \ \mu(\{ x \in X \mid \Vert x - \pi(g)x \Vert \geq \gamma \}) = 1.
			\end{displaymath}
\end{enumerate} \end{thm}

In view of Lemma~\ref{lemma:elek.szabo.comparison}, we observe that the only essential difference in the criteria of Theorem~\ref{theorem:elek.szabo.original} and Theorem~\ref{theorem:elek.szabo} lies in the range of the map $\pi$. For the proof of Theorem~\ref{theorem:elek.szabo}, the following preparatory lemma will be useful.

\begin{lem}\label{lemma:near.representations.new} Let $G$ be a group, $X$ a set and $\mu$ a probability charge on $\Pow(X)$. A map $\pi \colon G \to \U(\ell^{2}(X))$ is a $\mu$-near representation if and only if $\pi$ is a $\mu_{\bullet}$-near representation. \end{lem}

\begin{proof} Let $\pi \colon G \to \U(\ell^{2}(X))$, and define 
\begin{displaymath}f_{g,h} \colon X \longrightarrow [0,2], \quad x \longmapsto \lVert \pi(gh)x - \pi(g)\pi(h)x \rVert\end{displaymath} for any pair of elements $g,h \in G$.
	
($\Longrightarrow$) Suppose that $\pi$ is a $\mu$-near representation. If $g,h \in G$, then \begin{align*}
	&\mu_{\bullet}((\pi(gh)-\pi(g)\pi(h))^{\ast}(\pi(gh)-\pi(g)\pi(h))) \\
	&\qquad = \, \int \langle x,(\pi(gh)-\pi(g)\pi(h))^{\ast}(\pi(gh)-\pi(g)\pi(h))x\rangle \, \mathrm{d}\mu(x) \\
	&\qquad = \, \int \langle (\pi(gh)-\pi(g)\pi(h))x,(\pi(gh)-\pi(g)\pi(h))x\rangle \, \mathrm{d}\mu(x) \, = \, \int f_{g,h}^{2} \, \mathrm{d}\mu \\
	&\qquad \leq \, \int 2f_{g,h} \, \mathrm{d}\mu \, = \, 2\int f_{g,h}\chi_{f_{g,h}^{-1}([0,\epsilon/2])} \, \mathrm{d}\mu + 2\int f_{g,h}\chi_{f_{g,h}^{-1}((\epsilon/2,2])} \, \mathrm{d}\mu \\
	&\qquad \leq \, \epsilon \mu\!\left( f_{g,h}^{-1}([0,\epsilon/2])\right)\! + 4 \mu\!\left( f_{g,h}^{-1}((\epsilon/2,2]) \right)\! \, \leq \, \epsilon
\end{align*} for every $\epsilon \in \R_{>0}$, thus $\mu_{\bullet}((\pi(gh)-\pi(g)\pi(h))^{\ast}(\pi(gh)-\pi(g)\pi(h))) = 0$. Hence, $\pi$ is a $\mu_{\bullet}$-near representation. 
	
($\Longleftarrow$) Assume now that $\pi$ is a $\mu_{\bullet}$-near representation. For all $g,h \in G$ and $\epsilon \in \R_{>0}$, from $\chi_{f_{g,h}^{-1}((\epsilon,2])} \leq \epsilon^{-1}f_{g,h}$ and the Cauchy--Schwartz inequality we infer that \begin{displaymath}
	\mu\!\left( f_{g,h}^{-1}((\epsilon,2]) \right)\! \, = \, \int \chi_{f_{g,h}^{-1}((\epsilon,2])} \, \mathrm{d}\mu\, \leq \, \epsilon^{-1}\int f_{g,h} \, \mathrm{d}\mu \, \leq \, \epsilon^{-1} \sqrt{\int f_{g,h}^{2} \, \mathrm{d}\mu} \, = \, 0 ,
\end{displaymath} thus $\mu\!\left( f_{g,h}^{-1}([0,\epsilon])\right) = 1$. This shows that $\pi$ is a $\mu$-near representation. \end{proof}

\begin{proof}[Proof of Theorem~\ref{theorem:elek.szabo}] \ref{theorem:elek.szabo.1}$\Longrightarrow$\ref{theorem:elek.szabo.2}. Suppose that $G$ is hyperlinear. Consider the directed set $(I,{\preceq})$ defined by $I \defeq \Pfin(G) \times \N_{>0}$ and \begin{displaymath}
	(E,k) \preceq (E',k') \quad :\Longleftrightarrow \quad E \subseteq E' \ \wedge \ k \leq k' \qquad ((E,k),(E',k') \in I).
\end{displaymath} Due to Proposition~\ref{proposition:hyperlinear.groups}, for each $i = (E,k) \in I$, we find a non-empty finite set $X_{i}$ and a map $\pi_{i} \colon G \to \U(\ell^{2}(X_{i}))$ such that \begin{align}
	&\forall g,h \in E \ \forall x \in X_{i}\colon \quad \lVert \pi_{i}(gh)x - \pi_{i}(g)\pi_{i}(h)x \rVert \leq \tfrac{1}{k} , \label{elek.szabo.almost.homomorphism} \\
	&\forall g \in E\setminus \{ e \} \ \forall x \in X_{i} \colon \quad \lvert \langle x,\pi_{i}(g)x \rangle \rvert \leq \tfrac{1}{k} . \label{elek.szabo.almost.freeness}
\end{align} Consider the disjoint union $X \defeq \bigsqcup_{i \in I} X_{i}$ along with the well-defined unique map $\pi \colon G \to \U(\ell^{2}(X))$ such that \begin{displaymath}
	\forall g \in G \ \forall f \in \ell^{2}(X) \ \forall i \in I \colon \qquad \pi(g)(f)\vert_{X_{i}} \, = \, \pi_{i}(g)(f\vert_{X_{i}}) .
\end{displaymath} Choose any (necessarily non-principal) ultrafilter $\mathcal{F}$ on $I$ containing \begin{displaymath}
	\{ \{ (E',k') \in I \mid (E,k) \preceq (E',k') \} \mid (E,k) \in I \} ,
\end{displaymath} and consider the probability charge \begin{displaymath}
	\mu \colon \, \Pow(X) \, \longrightarrow \, [0,1], \quad A \, \longmapsto \, \lim\nolimits_{i\to \mathcal{F}} \tfrac{\lvert A \cap X_{i} \rvert}{\lvert X_{i}\rvert} .
\end{displaymath} Concerning the projections $p_{i} \colon \ell^{2}(X) \to \ell^{2}(X_{i}), \, f \mapsto f\vert_{X_{i}}$ $(i \in I)$, we observe that \begin{displaymath}
	\mu_{\bullet}(a) \, = \, \int \langle x,ax \rangle \, \mathrm{d}\mu(x) \, = \, \lim\nolimits_{i \to \mathcal{F}} \tfrac{1}{\vert X_{i}\vert}\sum\nolimits_{x\in X_{i}}\langle x,ax\rangle \, = \, \lim\nolimits_{i \to \mathcal{F}} \tfrac{1}{\vert X_{i}\vert}\tr(p_{i}ap_{i}^{\ast})
\end{displaymath} for every $a \in \B(\ell^{2}(X))$. Thus, if $g \in G$, then \begin{align*}
	&(\pi(g).(\mu_{\bullet}))(a) \, = \, \mu_{\bullet}(\pi(g)^{\ast}a\pi(g)) \, = \, \lim\nolimits_{i \to \mathcal{F}} \tfrac{1}{\vert X_{i}\vert}\tr(p_{i}\pi(g)^{\ast}a\pi(g)p_{i}^{\ast}) \\
	&\qquad \quad = \, \lim\nolimits_{i \to \mathcal{F}} \tfrac{1}{\vert X_{i}\vert}\tr(\pi_{i}(g)^{\ast}p_{i}ap_{i}^{\ast}\pi_{i}(g)) \, = \, \lim\nolimits_{i \to \mathcal{F}} \tfrac{1}{\vert X_{i}\vert}\tr(p_{i}ap_{i}^{\ast}) \, = \, \mu_{\bullet}(a)
\end{align*} for every $a \in \B(\ell^{2}(X))$, that is, $\pi(g).(\mu_{\bullet}) = \mu_{\bullet}$. Moreover, if $g,h \in G$ and $\epsilon \in \R_{>0}$, then \begin{align*}
	&\mu(\{ x \in X \mid \lVert \pi(gh)x - \pi(g)\pi(h)x \rVert \leq \epsilon \}) \\
	&\qquad = \, \lim\nolimits_{i \to \mathcal{F}} \tfrac{1}{\vert X_{i} \vert} \lvert \{ x \in X_{i} \mid \lVert \pi(gh)x - \pi(g)\pi(h)x \rVert \leq \epsilon \} \rvert \\
	&\qquad = \, \lim\nolimits_{i \to \mathcal{F}} \tfrac{1}{\vert X_{i} \vert} \lvert \{ x \in X_{i} \mid \lVert \pi_{i}(gh)x - \pi_{i}(g)\pi_{i}(h)x \rVert \leq \epsilon \} \rvert \, \stackrel{\eqref{elek.szabo.almost.homomorphism}}{=} \, 1 ,
\end{align*} whence $\pi$ is a $\mu$-near representation. Finally, if $g \in G\setminus \{ e \}$ and $\epsilon \in \R_{>0}$, then \begin{align*}
	\mu(\{ x \in X \mid \lvert \langle x,\pi(g)x \rangle \rvert \leq \epsilon \}) \, & = \, \lim\nolimits_{i \to \mathcal{F}} \tfrac{1}{\vert X_{i} \vert} \lvert \{ x \in X_{i} \mid \langle x,\pi(g)x \rangle \rvert \leq \epsilon \} \rvert \\
	& = \, \lim\nolimits_{i \to \mathcal{F}} \tfrac{1}{\vert X_{i} \vert} \lvert \{ x \in X_{i} \mid \langle x,\pi_{i}(g)x \rangle \rvert \leq \epsilon \} \rvert \, \stackrel{\eqref{elek.szabo.almost.freeness}}{=} \, 1 .
\end{align*}

\ref{theorem:elek.szabo.2}$\Longrightarrow$\ref{theorem:elek.szabo.3}. This implication follows from the fact that, in any Hilbert space $H$, \begin{displaymath}
	\forall x,y \in \Sph(H) \colon \quad \Vert x-y \Vert \, = \, \sqrt{2-2\Re\langle x,y \rangle} \, \geq \, \sqrt{2-2\vert\langle x,y \rangle\vert} .
\end{displaymath}
	
\ref{theorem:elek.szabo.3}$\Longrightarrow$\ref{theorem:elek.szabo.1}. Let $X$ be a set, let $\mu$ be a probability charge on $\Pow(X)$, and let $\pi \colon G \to \U(\ell^{2}(X))$ be a $\mu$-near representation. Then $\phi \defeq \mu_{\bullet}$ is a state on $\B(\ell^{2}(X))$ according to Remark~\ref{remark:charges.induce.states}, and $\pi$ is a $\phi$-near representation due to Lemma~\ref{lemma:near.representations.new}. Now, if $g \in G$ and $\gamma \in \R_{\geq 0}$ satisfy \begin{displaymath}
	\mu(\underbrace{\{ x \in X \mid \lVert x - \pi(g)x \rVert \geq \gamma \}}_{\eqdef\,B}) \, = \, 1 ,
\end{displaymath} then \begin{align*}
	&\phi((1-\pi(g))^{\ast}(1-\pi(g))) \, = \, \int \langle x,(1-\pi(g))^{\ast}(1-\pi(g))x \rangle \, \mathrm{d}\mu(x) \\
	&\qquad = \, \int \lVert x-\pi(g)x\rVert^{2} \, \mathrm{d}\mu(x) \, \geq \, \int \gamma^{2}\chi_{B}(x) \, \mathrm{d}\mu(x) \, = \, \gamma^{2}\mu(B) \, = \, \gamma^{2} .
\end{align*} Consequently, \ref{theorem:elek.szabo.3} implies the condition~\ref{theorem:amenable.trace.d2} of Theorem~\ref{theorem:amenable.trace} and thus entails hyperlinearity of $G$. \end{proof}

The connection between hyperlinearity and amenability of unitary representations established by our results provides a new perspective on the relation between hyperlinearity and the Haagerup property. For comparison, we record below a characterization of the latter in the spirit of Theorem~\ref{theorem:elek.szabo}.

Let us recall that, if $X$ is any set, then \begin{displaymath}
     c_{0}(X) \, \defeq \, \left. \! \left\{ f \in \C^{X} \, \right\vert \forall \epsilon \in \R_{>0} \colon \, \{ x \in X \mid \vert f(x) \vert \geq \epsilon \} \in \Pfin(X) \right\}
\end{displaymath} constitutes a $C^{\ast}$-subalgebra of $\ell^{\infty}(X)$. A unitary representation $\pi$ of a group $G$ on a Hilbert space $H$ is said to \begin{itemize}
     \item[---\,] have \emph{almost invariant vectors} if, for every $E \in \Pfin(G)$ and every $\epsilon \in \R_{>0}$, there exists $x \in \Sph(H)$ such that $\sup_{g \in E} \Vert x-\pi(g)x \Vert \leq \epsilon$,
     \item[---\,] be a \emph{$c_{0}$-representation}~\cite[\S3, p.~337]{BekkaDeLaHarpe} or be \emph{strongly mixing}~\cite[Definition~1.1]{BergelsonRosenblatt} if, for all $x,y \in H$, the function $\kappa_{x,y} \colon G \to \C, \, g \mapsto \langle x,\pi(g)y \rangle$ belongs to $c_{0}(G)$.
\end{itemize} A countable group is said to have the \emph{Haagerup property} if it admits a $c_{0}$-representation with almost invariant vectors (cf.~\cite[Proposition~2.3]{Jol00}).

\begin{remark}\label{remark:c0} Let $\pi$ be a unitary representation of a group $G$ on a Hilbert space $H$. Suppose that $X \subseteq H$ spans a dense linear subspace of $H$. Since $c_{0}(G)$ is a closed ${}^{\ast}$-subalgebra of $\ell^{\infty}(G)$ and the map $H \times H \to \ell^{\infty}(G), \, (x,y) \mapsto \kappa_{x,y}$ is sesquilinear and continuous, $\pi$ is a $c_{0}$-representation if and only if $\{ \kappa_{x,y} \mid x,y \in X \} \subseteq c_{0}(G)$. \end{remark}

\begin{prop}[{\cite[\S3, p.~270]{Jol14}}]\label{proposition:jolissaint} A countable group has the Haagerup property if and only if it admits an amenable $c_{0}$-representation.
\end{prop}

\begin{proof} ($\Longrightarrow$) This implication follows from the observation that any unitary representation with almost invariant vectors is amenable.

($\Longleftarrow$) Let $\pi$ be an amenable $c_{0}$-representation of a countable group $G$ on a Hilbert space $H$. According to~\cite[Theorem~5.1]{bekka}, amenability of $\pi$ means that the representation $\tilde{\pi} \colon G \to \U(\B_{2}(H))$, on the Hilbert space $\B_{2}(H)$ of Hilbert--Schmidt operators on $H$, given by \begin{displaymath}
     \tilde{\pi}(g)a \, \defeq \, \pi(g)a\pi(g)^{\ast} \qquad (g \in G, \, a \in \B_{2}(H))
\end{displaymath} has almost invariant vectors. Moreover, if $p,q \in \Pro(H)$ satisfy $\rk(p) = \rk(q) = 1$ and $x \in \Sph(p(H))$ and $y \in \Sph(q(H))$, then \begin{align*}
     \kappa_{p,q}(g) \, &= \, \tr(p^{\ast}\pi(g)q\pi(g)^{\ast}) \, = \, \tr(\pi(g)^{\ast}p\pi(g)q) \, = \, \langle y,\pi(g)^{\ast}p\pi(g)qy \rangle \\
     & = \, \langle \pi(g)y,p\pi(g)y \rangle \, = \, \langle \pi(g)y,\langle x,\pi(g)y\rangle x \rangle \, = \, \langle x,\pi(g)y\rangle\langle \pi(g)y, x \rangle \\
     & = \, \langle x,\pi(g)y\rangle\overline{\langle x,\pi(g)y\rangle} \, = \, \kappa_{x,y}(g)\overline{\kappa_{x,y}(g)} \, = \, \! \left(\kappa_{x,y}\kappa_{x,y}^{\ast}\right)\!(g)
\end{align*} for every $g \in G$, whence $\kappa_{p,q} = \kappa_{x,y}\kappa_{x,y}^{\ast} \in c_{0}(G)$ due to $\pi$ being a $c_{0}$-representation. Since $\{ p \in \Pro(H) \mid \rk(p) = 1 \}$ spans a dense linear subspace of $\B_{2}(H)$ (see, e.g.,~\cite[Theorem~2.4.6, p.~56]{MurphyBook}
and~\cite[Theorem~2.4.17, p.~66]{MurphyBook}), we conclude that $\tilde{\pi}$ is a $c_{0}$-representation by Remark~\ref{remark:c0}. Hence, $G$ has the Haagerup property. \end{proof}

\begin{cor}\label{corollary:haagerup} A countable group $G$ has the Haagerup property if and only if there exist a set~$X$, a probability charge $\mu$ on $\mathcal{P}(X)$, and a representation $\pi \colon G \to \U(\ell^{2}(X))$ such that $\mu_{\bullet}$ is $\pi$-invariant and, for all $x,y \in X$ and $\epsilon \in \R_{>0}$, \begin{displaymath}
	\lvert \{ g \in G \mid \vert \langle x,\pi(g)y \rangle \vert \geq \epsilon \} \rvert \, < \, \infty .
\end{displaymath} \end{cor}

\begin{proof} This follows from Proposition~\ref{proposition:jolissaint}, Theorem~\ref{theorem:amenable.basis}, and Remark~\ref{remark:c0}. \end{proof}

\section{Kirchberg's factorization property}\label{section:kirchberg}

In this section, we establish characterizations of Kirchberg's factorization property in terms of the existence of unitary representations admitting a free invariant state (Theorem~\ref{theorem:kirchberg}). The proof relies on an amplification argument (Lemma~\ref{lemma:amplification}) inspired by~\cite[Proof of Corollary~1.2, (ii)$\Longrightarrow$(i), p.~561--562]{Kirchberg} and involving the following elementary fact.

\begin{remark}[cf.~{\cite[Exercise~3.2.7(1), p.~77]{CapraroLupini}}]\label{remark:complex} If $z \in \C\setminus\{ 1 \}$ and $\vert z \vert \leq 1$, then \begin{displaymath}
	\vert z+1 \vert^{2} = (\Re z + 1)^{2} + (\Im z)^{2} = (\Re z)^{2} + 2\Re z + 1 + (\Im z)^{2} \leq 2+2\Re z < 4
\end{displaymath} and therefore $\vert z+1 \vert < 2$. \end{remark}

\begin{lem}\label{lemma:amplification} Let $\rho_{0}$ be a unitary representation of a group $G$ on a Hilbert space $K_{0}$. Suppose that there exists a $\rho_{0}$-invariant state $\phi_{0} \colon \B(K_{0}) \to \C$ such that \begin{displaymath}
	\forall g \in G \setminus \{ e \} \colon \qquad \phi_{0}((1-\rho_{0}(g))^{\ast}(1-\rho_{0}(g))) \, > \, 0 .
\end{displaymath} For every $E \in \Pfin(G)$ and every $\epsilon \in \R_{>0}$, there exist a unitary representation $\pi$ of~$G$ on a Hilbert space $H$ and some $p \in \Pro(H)$ such that \begin{enumerate}
	\item[---\,] $\tfrac{1}{\epsilon} \leq \rk(p) < \infty$,
	\item[---\,] $\lVert \pi(g)p\pi(g)^{\ast} - p \rVert_{2} \leq \epsilon\Vert p \Vert_{2}$ for all $g \in E$,
	\item[---\,] $\vert\tr(\pi(g)p)\vert \leq \epsilon \tr(p)$ for every $g \in E\setminus \{ e \}$.
\end{enumerate} \end{lem}

\begin{proof} First of all, let us consider the Hilbert space $K_{1} \defeq K_{0} \oplus K_{0}$ along with the linear isometries $i \colon K_{0} \to K_{1}, \, x \mapsto (x,0)$ and $j \colon K_{0} \to K_{1}, \, x \mapsto (0,x)$ and the unitary representation \begin{displaymath}
	\rho_{1} \colon \, G \, \longrightarrow \, \U(K_{1}), \quad g \, \longmapsto \, \rho_{0}(g) \oplus \id_{K_{0}} .
\end{displaymath} It is easy to verify that \begin{displaymath}
	\phi_{1} \colon \, \B(K_{1}) \, \longrightarrow \, \C, \quad a \, \longmapsto \, \tfrac{1}{2}(\phi_{0}(i^{\ast}ai) + \phi_{0}(j^{\ast}aj))
\end{displaymath} is a $\rho_{1}$-invariant state. Moreover, for every $g \in G\setminus \{ e \}$, since $\vert \phi_{0}(\rho_{0}(g)) \vert \leq \Vert \rho_{0}(g) \Vert = 1$ and $\phi_{0}(\rho_{0}(g)) \ne 1$ as \begin{align*}
	2-2\Re \phi_{0}(\rho_{0}(g)) \, &\stackrel{\ref{remark:positive}\ref{remark:positive.1}}{=} \, 2-\phi_{0}(\rho_{0}(g))-\phi_{0}(\rho_{0}(g)^{\ast}) \\
	&= \, \phi_{0}((1-\rho_{0}(g))^{\ast}(1-\rho_{0}(g))) \, > \, 0 ,
\end{align*} we see that \begin{equation}\label{eq:complex}
	\vert \phi_{1}(\rho_{1}(g)) \vert \, = \, \tfrac{1}{2}\vert \phi_{0}(\rho_{0}(g)) + 1 \vert \, \stackrel{\ref{remark:complex}}{<} \, 1 .
\end{equation} Now, according to Remark~\ref{remark:positive}\ref{remark:positive.2}, \begin{displaymath}
	A \, \defeq \, \{ a \in \B(K_{1}) \mid \forall b \in \B(K_{1}) \colon \, \phi_{1}(ab) = \phi_{1}(ba) \}
\end{displaymath} constitutes a unital $C^{\ast}$-subalgebra of~$\B(K_{1})$. As $\phi_{1}$ is $\rho_{1}$-invariant, $\rho_{1}(G) \subseteq \U(A)$. Let $E \in \Pfin(G)$ and $\epsilon \in \R_{>0}$. We observe that \begin{displaymath}
	\gamma \, \defeq \, \sup\nolimits_{g \in E\setminus \{ e \}} \vert \phi_{1}(\rho_{1}(g)) \vert \, \stackrel{\eqref{eq:complex}}{<} \, 1. 
\end{displaymath} In turn, we find some $n \in \N_{>0}$ such that $\left( \tfrac{1+\gamma}{2} \right)^{n}\! \leq \epsilon$ and $2^{n} \geq \tfrac{1}{\epsilon}$. Choose $\delta \in \R_{>0}$ such that \begin{displaymath}
	2\sqrt{1 - (1-\delta)^{n}} \, \leq \, \epsilon .
\end{displaymath} By Lemma~\ref{lemma:brown.ozawa}, we find a non-zero finite-dimensional Hilbert space $L$ and a unital completely positive linear map $\sigma \colon A \to \B(L)$ such that, for every $g \in E \cup E^{-1}$, \begin{align}
	&\lvert 1 - \dim(L)^{-1}\tr(\sigma(\rho_{1}(g))^{\ast}\sigma(\rho_{1}(g))) \rvert \, \leq \, \delta , \label{eq:kirchberg.almost.invariance} \\
	&\left\lvert \phi_{1}(\rho_{1}(g)) - \dim(L)^{-1}\tr(\sigma(\rho_{1}(g))) \right\rvert \, \leq \, \tfrac{1-\gamma}{2} . \label{eq:kirchberg.approximate.freeness}
\end{align} Due to Theorem~\ref{theorem:stinespring}, there exist a Hilbert space $H_{0}$, a unital ${}^{\ast}$-algebra homomorphism $\pi_{0} \colon A \to \B(H_{0})$ and a linear isometry $v \colon L \to H_{0}$ such that \begin{equation}\label{eq:kirchberg.dilation}
	\forall a \in A \colon \qquad \sigma(a) = v^{\ast}\pi_{0}(a)v .
\end{equation} Of course, $v^{\ast}v = \id_{L}$ thanks to $v$ being a linear isometry. Thus, $p_{0} \defeq vv^{\ast} \in \Pro(H_{0})$ and $\rk(p_{0}) = \dim(L)$. Consider the Hilbert space $H_{1} \defeq H_{0} \oplus H_{0}$ along with the non-zero finite-rank projection $p_{1} \defeq p_{0} \oplus p_{0} \in \Pro(H_{1})$ and the unitary representation \begin{displaymath}
	\pi_{1} \colon \, G \, \longrightarrow \, \U(H_{1}), \quad g \, \longmapsto \, \pi_{0}(\rho_{1}(g)) \oplus \pi_{0}(\rho_{1}(g)) .
\end{displaymath} Furthermore, let us examine the Hilbert space $H \defeq H_{1}^{\otimes n}$, the non-zero finite-rank projection $p \defeq p_{1}^{\otimes n} \in \Pro(H)$, and the unitary representation \begin{displaymath}
	\pi \colon \, G \, \longrightarrow \, \U(H), \quad g \, \longmapsto \, \pi_{1}(g)^{\otimes n} .
\end{displaymath} Note that \begin{displaymath}
	\rk(p) \, = \, \rk(p_{1})^{n} \, = \, 2^{n}\rk(p_{0})^{n} \, = \, 2^{n}\dim(L)^{n} \, \geq \, 2^{n} \, \geq \, \tfrac{1}{\epsilon}. 
\end{displaymath} Now, for every $g \in E \cup E^{-1}$, \begin{allowdisplaybreaks} \begin{align*}
	\Vert p_{0}\pi_{0}(\rho_{1}(g))p_{0} \Vert_{2}^{2} \, &= \, \Vert vv^{\ast}\pi_{0}(\rho_{1}(g))vv^{\ast} \Vert_{2}^{2} \, \stackrel{\eqref{eq:kirchberg.dilation}}{=} \, \Vert v\sigma(\rho_{1}(g))v^{\ast} \Vert_{2}^{2} \, \stackrel{v\,\text{isom.}}{=} \, \Vert \sigma(\rho_{1}(g)) \Vert_{2}^{2} \\
	& = \, \tr(\sigma(\rho_{1}(g))^{\ast}\sigma(\rho_{1}(g))) \, \stackrel{\eqref{eq:kirchberg.almost.invariance}}{\geq} \, (1-\delta)\dim(L) \, = \, (1-\delta)\rk(p_{0}) , \\
	\Vert p_{1}\pi_{1}(g)p_{1} \Vert_{2}^{2} \, &= \, \Vert p_{0}\pi_{0}(\rho_{1}(g))p_{0} \oplus p_{0}\pi_{0}(\rho_{1}(g))p_{0} \Vert_{2}^{2} \, = \, 2\Vert p_{0}\pi_{0}(\rho_{1}(g))p_{0} \Vert_{2}^{2} \\
	&\geq \, 2(1-\delta)\rk(p_{0}) \, = \, (1-\delta)\rk(p_{1}) , \\
	\Vert p\pi(g)p \Vert_{2}^{2} \, &= \, \left\lVert p_{1}^{\otimes n}\pi_{1}(g)^{\otimes n}p_{1}^{\otimes n} \right\rVert_{2}^{2} \, = \, \left\lVert (p_{1}\pi_{1}(g)p_{1})^{\otimes n} \right\rVert_{2}^{2} \, = \, \Vert p_{1}\pi_{1}(g)p_{1} \Vert_{2}^{2n} \\
	&\geq \, (1-\delta)^{n}\rk(p_{1})^{n} \, = \, (1-\delta)^{n}\rk(p) \, = \, (1-\delta)^{n}\Vert p \Vert_{2}^{2}, \\
	\Vert (1-p)\pi(g)p \Vert_{2}^{2} \, &\stackrel{p \perp (1-p)}{=} \, \Vert \pi(g)p \Vert_{2}^{2} - \Vert p\pi(g)p \Vert_{2}^{2} \, = \, \Vert p \Vert_{2}^{2} - \Vert p\pi(g)p \Vert_{2}^{2} \\
	& \leq \, (1 - (1-\delta)^{n}) \Vert p \Vert_{2}^{2} .
\end{align*} \end{allowdisplaybreaks} For every $g \in E$, since $(p\pi(g)(1-p))^{\ast} = (1-p)\pi(g)^{\ast}p = (1-p)\pi(g^{-1})p$, we thus conclude that \begin{align*}
	\Vert \pi(g)p - p\pi(g) \Vert_{2} \, &\leq \, \Vert \pi(g)p - p\pi(g)p \Vert_{2} + \Vert p\pi(g)p - p\pi(g) \Vert_{2} \\
	& = \, \Vert (1-p)\pi(g)p \Vert_{2} + \Vert p\pi(g)(1-p) \Vert_{2} \\
	& = \, \Vert (1-p)\pi(g)p \Vert_{2} + \Vert(p\pi(g)(1-p))^{\ast} \Vert_{2} \\
	&\leq \, 2\sqrt{1 - (1-\delta)^{n}}\Vert p \Vert_{2} \, \leq \, \epsilon \Vert p \Vert_{2} ,
\end{align*} i.e., $\lVert \pi(g)p\pi(g)^{\ast} - p \rVert_{2} \leq \epsilon\Vert p \Vert_{2}$. Finally, if $g \in E\setminus \{ e \}$, then \begin{align*}
	\vert \tr (\pi_{0}(\rho_{1}(g))p_{0}) \vert \, &= \, \vert \tr (\pi_{0}(\rho_{1}(g))vv^{\ast}) \vert \, = \, \vert \tr (v^{\ast}\pi_{0}(\rho_{1}(g))v) \vert \, \stackrel{\eqref{eq:kirchberg.dilation}}{=} \, \vert \tr(\sigma(\rho_{1}(g))) \vert \\
	& \leq \, \vert \phi_{1}(\rho_{1}(g)) \dim(L) \vert + \lvert \phi_{1}(\rho_{1}(g))\dim(L) - \tr(\sigma(\rho_{1}(g))) \rvert \\
	& \stackrel{\eqref{eq:kirchberg.approximate.freeness}}{\leq} \, \gamma \dim(L) + \tfrac{1-\gamma}{2}\dim(L) \, = \, \tfrac{1+\gamma}{2}\rk(p_{0}) , \\
	\vert \tr (\pi_{1}(g)p_{1}) \vert \, &= \, \vert \tr(p_{0}\pi_{0}(\rho_{1}(g))p_{0} \oplus p_{0}\pi_{0}(\rho_{1}(g))p_{0}) \vert \, = \, 2 \vert \tr (\pi_{0}(\rho_{1}(g))p_{0}) \vert \\
	&\leq \, \tfrac{1+\gamma}{2}2\rk(p_{0}) \, = \, \tfrac{1+\gamma}{2}\rk(p_{1}) , \\
	\vert \tr(\pi(g)p) \vert \, &= \, \left\lvert \tr\!\left(\pi_{1}(g)^{\otimes n}p_{1}^{\otimes n}\right) \right\rvert \, = \, \left\lvert \tr\!\left((\pi_{1}(g)p_{1})^{\otimes n}\right) \right\rvert \, = \, \vert \tr(\pi_{1}(g)p_{1}) \vert^{n} \\
	& \leq \, \left(\tfrac{1+\gamma}{2}\right)^{n}\!\rk(p_{1})^{n} \, = \, \left(\tfrac{1+\gamma}{2}\right)^{n}\!\rk(p) = \, \left(\tfrac{1+\gamma}{2}\right)^{n}\!\tr(p) \, \leq \, \epsilon \tr(p) . \qedhere
\end{align*} \end{proof}

Now, let us recall that, if $\pi \colon A \to B$ is a unital ${}^{\ast}$-algebra homomorphism between unital $\Cstar$-algebras $A$ and $B$, then $\pi(A)$ is a unital $\Cstar$-subalgebra of $B$ (see, e.g.,~\cite[Corollary~5.7, p.~23]{ConwayBook}). Moreover, the following well-known fact will be needed. 

\begin{lem}[cf.~{\cite[Proposition~6.2.2, p.~215]{BrownOzawa}}]\label{lemma:arveson} Let $A$ be a unital $\Cstar$-algebra, let $H$ and $K$ be Hilbert spaces, let $\pi \colon A \to \B(H)$ be a unital ${}^{\ast}$-algebra homomorphism, and let $\iota \colon A \to \B(K)$ be a unital ${}^{\ast}$-algebra embedding. If $\phi \colon \B(H) \to \C$ is a hypertrace for $\pi(A)$, then there exists a hypertrace $\psi \colon \B(K) \to \C$ for $\iota(A)$ such that $\psi \circ \iota = \phi \circ \pi$. \end{lem}

\begin{proof} We reproduce the slightly amended proof from~\cite[Proposition~6.2.2, p.~215]{BrownOzawa} for the reader's convenience. Due to Arveson's extension theorem~\cite[Theorem~1.2.3]{Arveson} (see also~\cite[Theorem~1.6.1, p.~17]{BrownOzawa}), there exists a unital completely positive linear map $\Phi \colon \B(K) \to \B(H)$ such that $\Phi\vert_{\iota(A)} = \pi \circ {\iota^{-1}}$. Since $\Phi\vert_{\iota(A)}$ is a ${}^{\ast}$-ring homomorphism, \cite[Corollary~2.8]{Ozawa} moreover asserts that $\Phi(\iota(a)x\iota(b)) = \pi(a)\Phi(x)\pi(b)$ for all $a,b \in A$ and $x \in \B(K)$. Now, consider any hypertrace $\phi \colon \B(H) \to \C$ for $\pi(A)$. Then $\psi \defeq \phi \circ \Phi$ is a state on $\B(K)$ and $\psi \circ \iota = \phi \circ \Phi \circ \iota = \phi \circ \pi$. Finally, \begin{align*}
	\psi(\iota(a)x) \, &= \, \phi(\Phi(\iota(a)x)) \, = \, \phi(\pi(a)\Phi(x)) \\
		& = \, \phi(\Phi(x)\pi(a)) \, = \, \phi(\Phi(x\iota(a))) \, = \, \psi(x\iota(a))
\end{align*} for all $a \in A$ and $x \in \B(K)$, i.e., $\psi$ is a hypertrace for $\iota(A)$. \end{proof}

In order to clarify some relevant terminology and notation, let $G$ be a group. Let $\Cstar(G)$ denote the \emph{(full) group $\Cstar$-algebra} of $G$ (cf.~\cite[\S2.5, p.~43]{BrownOzawa}), i.e., the completion of the complex group algebra $\C[G]$ with respect to its largest $\Cstar$-norm, and consider the natural embedding $\iota_{G} \colon G \to \U(\Cstar(G)), \, g \mapsto \delta_{g}$. Note that $\Cstar(G)$ enjoys the following universal property (see, e.g.,~\cite[Proposition~2.5.2, p.~43]{BrownOzawa}): for any unitary representation $\pi$ of $G$ on a Hilbert space $H$, there exists a unique (necessarily unital) ${}^{\ast}$-algebra homomorphism $\theta_{\pi} \colon \Cstar(G) \to \B(H)$ such that ${\theta_{\pi}} \circ {\iota_{G}} = \pi$. In particular, there is a unique unital ${}^{\ast}$-algebra homomorphism $\rho_{G} \colon \Cstar(G) \to \U(\ell^{2}(G))$~such~that \begin{displaymath}
	\forall g \in G \ \forall f \in \ell^{2}(G) \ \forall x \in G \colon \qquad (\rho_{G}(\iota_{G}(g))f)(x) = f(g^{-1}x) .
\end{displaymath} Furthermore, let us define \begin{displaymath}
	\tau_{G} \colon \, \Cstar(G) \, \longrightarrow \, \C, \quad a \, \longmapsto \, \langle \delta_{e},\rho_{G}(a)\delta_{e} \rangle .
\end{displaymath} Then the group $G$ is said to have \emph{Kirchberg's factorization property}~\cite{KirchbergInventiones,Kirchberg} (see also~\cite[Theorem~6.4.3, p.~228]{BrownOzawa}) if for some (hence any, by Lemma~\ref{lemma:arveson}) pair $(H,\iota)$ consisting of a Hilbert space $H$ and a unital ${}^{\ast}$-algebra embedding $\iota \colon \Cstar(G) \to \B(H)$ there exists a hypertrace $\phi \colon \B(H) \to \C$ for $\iota(\Cstar(G))$ such that $\phi \circ \iota = \tau_{G}$.

\begin{thm}\label{theorem:kirchberg} Let $G$ be a group. The following are equivalent. \begin{enumerate}
	\item\label{theorem:kirchberg.a1} $G$ has Kirchberg's factorization property.
	\item\label{theorem:kirchberg.c1} There exist a unitary representation $\pi$ of $G$ on a Hilbert space $H$ and a $\pi$-invariant state $\mu$ on $\UCB(\Sph(H))$ such that \begin{displaymath}
		\qquad \forall g \in G\setminus \{ e\} \colon \quad  \mu(x \mapsto \vert\langle x,\pi(g)x \rangle\vert) = 0 .
	\end{displaymath}
	\item\label{theorem:kirchberg.c2} There exist a unitary representation $\pi$ of $G$ on a Hilbert space $H$ and a $\pi$-invariant state $\mu$ on $\UCB(\Sph(H))$ such that \begin{displaymath}
		\qquad \forall g \in G\setminus \{ e\} \colon \quad \mu(x \mapsto \lVert x - \pi(g)x \rVert) > 0 .
	\end{displaymath}	
	\item\label{theorem:kirchberg.d1} There exist a unitary representation $\pi$ of $G$ on a Hilbert space $H$ and a $\pi$-invariant state $\phi$ on $\B(H)$ such that \begin{displaymath}
			\qquad \forall g \in G \setminus \{ e \} \colon \quad \phi(\pi(g)) = 0 .
		\end{displaymath}
	\item\label{theorem:kirchberg.d2} There exist a unitary representation $\pi$ of $G$ on a Hilbert space $H$ and a $\pi$-invariant state $\phi$ on $\B(H)$ such that \begin{displaymath}
			\qquad \forall g \in G \setminus \{ e \} \colon \quad \phi((1-\pi(g))^{\ast}(1-\pi(g))) > 0 .
		\end{displaymath}
	\item\label{theorem:kirchberg.e1} There exist a set $X$, a unitary representation $\pi \colon G \to \U(\ell^{2}(X))$ and a probability charge $\mu$ on $\Pow(X)$ such that $\mu_{\bullet}$ is $\pi$-invariant and \begin{displaymath}
		\qquad \forall g \in G\setminus \{ e\} \ \forall \epsilon \in \R_{>0} \colon \ \ \mu(\{ x \in X \mid \lvert \langle x,\pi(g)x \rangle \rvert \leq \epsilon \}) = 1.
	\end{displaymath}
	\item\label{theorem:kirchberg.e2} There exist a set $X$, a unitary representation $\pi \colon G \to \U(\ell^{2}(X))$ and a probability charge $\mu$ on $\Pow(X)$ such that $\mu_{\bullet}$ is $\pi$-invariant and \begin{displaymath}
		\qquad \forall g \in G\setminus \{ e\} \ \exists \gamma \in \R_{>0} \colon \ \ \mu(\{ x \in X \mid \Vert x - \pi(g)x \Vert \geq \gamma \}) = 1.
	\end{displaymath}
\end{enumerate} \end{thm}

\begin{proof} The implications \ref{theorem:kirchberg.c1}$\Longrightarrow$\ref{theorem:kirchberg.d1}$\Longrightarrow$\ref{theorem:kirchberg.d2} and \ref{theorem:kirchberg.c2}$\Longrightarrow$\ref{theorem:kirchberg.d2} follow as in the proof of Theorem~\ref{theorem:amenable.trace}.
	
\ref{theorem:kirchberg.c1}$\Longrightarrow$\ref{theorem:kirchberg.c2}. Let $H$ be a Hilbert space and let $\mu$ be a state on $\UCB(\Sph(H))$. For every $u \in \U(H)$, we see that \begin{align*}
	2\mu(x \mapsto \Vert x-ux \Vert) \, &\geq \, \mu\!\left( x \mapsto \Vert x-ux \Vert^{2} \right)\! \, = \, 2-2 \mu(x\mapsto \Re \langle x,ux \rangle) \\
		& \geq \, 2-2\mu(x \mapsto \vert \langle x,ux\rangle\vert) ,
\end{align*} and hence \begin{displaymath}
	\mu(x \mapsto \Vert x-ux \Vert) \, \geq \, 1-\mu(x \mapsto \vert \langle x,ux\rangle\vert).
\end{displaymath} Thus, \ref{theorem:kirchberg.c1} entails~\ref{theorem:kirchberg.c2}.
	
\ref{theorem:kirchberg.a1}$\Longrightarrow$\ref{theorem:kirchberg.d1}. Suppose that $G$ has Kirchberg's factorization property. Then there exist a Hilbert space $H$, a unital ${}^{\ast}$-algebra embedding $\eta \colon \Cstar(G) \to \B(H)$ and a hypertrace $\phi \colon \B(H) \to \C$ for $\eta(\Cstar(G))$ such that $\phi \circ \eta = \tau_{G}$. Hence, $\pi \defeq \eta \circ {\iota_{G}} \colon G \to \U(H)$ is a unitary representation, $\phi$ is $\pi$-invariant, and $\phi(\pi(g)) = \phi(\eta(\iota_{G}(g))) = \tau_{G}(\iota_{G}(g)) = 0$ for every $g \in G\setminus \{e\}$, as desired.
	
\ref{theorem:kirchberg.d1}$\Longrightarrow$\ref{theorem:kirchberg.a1}. Let $\pi$ be a unitary representation of $G$ on a Hilbert space $H$ and let $\phi$ be a $\pi$-invariant state on $\B(H)$ such that $\phi(\pi(g)) = 0$ for every $g \in G\setminus \{e\}$. According to~\cite[Proposition~2.5.2, p.~43]{BrownOzawa}, there exists a (unique) unital ${}^{\ast}$-algebra homomorphism $\theta_{\pi} \colon \Cstar(G) \to \B(H)$ such that ${\theta_{\pi}} \circ {\iota_{G}} = \pi$. Then $\phi \circ {\theta_{\pi}} \circ {\iota_{G}} = \phi \circ \pi = {\tau_{G}} \circ {\iota_{G}}$ and thus \begin{equation}\label{eq:extension}
	\phi \circ {\theta_{\pi}} \, = \, {\tau_{G}} ,
\end{equation} as $\iota_{G}(G)$ spans a dense linear subspace of $\Cstar(G)$. By the Gelfand--Naimark representation theorem~\cite{GelfandNaimark}\footnote{For a modern exposition, see~\cite[Theorem~3.4.1, p.~94]{MurphyBook} or~\cite[Theorem~7.10, p.~33]{ConwayBook}.}, there exists a unital ${}^{\ast}$-algebra embedding $\eta \colon \Cstar(G) \to \B(K)$ for some Hilbert space $K$. By Lemma~\ref{lemma:arveson}, there exists a hypertrace $\psi \colon \B(K) \to \C$ for $\eta(\Cstar(G))$ such that $\psi \circ \eta = \phi \circ {\theta_{\pi}}$, whence $\psi \circ \eta = {\tau_{G}}$ due to~\eqref{eq:extension}. Therefore, $G$ has Kirchberg's factorization property.

\ref{theorem:kirchberg.d2}$\Longrightarrow$\ref{theorem:kirchberg.c1}$\wedge$\ref{theorem:kirchberg.e1}. Consider the directed set $(I,{\preceq})$ defined by $I \defeq \Pfin(G) \times \N_{>0}$ and \begin{displaymath}
	(E,k) \preceq (E',k') \quad :\Longleftrightarrow \quad E \subseteq E' \ \wedge \ k \leq k' \qquad ((E,k),(E',k') \in I).
\end{displaymath} According to~\ref{theorem:kirchberg.d2} and Lemma~\ref{lemma:amplification}, for each $i = (E,k) \in I$, we find a Hilbert space $H_{i}$, a homomorphism $\pi_{i} \colon G \to \U(H_{i})$, and a finite-rank projection $p_{i} \in \Pro(H_{i})\setminus \{ 0 \}$ such that \begin{align}
	&4\exp\!\left(-\tfrac{2\rk(p_{i})-1}{16k^{2}}\right) < \min\!\left\{ \tfrac{1}{k}, \tfrac{1}{(\vert E \vert+1) \rk(p_{i})} \right\} , \label{eq:kirchberg.concentration} \\
	&\forall g \in E \colon \quad \tfrac{1}{\Vert p_{i} \Vert_{2}}\lVert \pi_{i}(g)p_{i}\pi_{i}(g)^{\ast} - p_{i} \rVert_{2} \leq \tfrac{1}{k} , \label{eq:kirchberg.invariance} \\
	&\forall g \in E\setminus\{ e\} \colon \quad \tfrac{1}{\tr(p_{i})}\vert\tr(\pi_{i}(g)p_{i})\vert \leq \tfrac{1}{k} . \label{eq:kirchberg.freeness}
\end{align} Consider the Hilbert space sum $H \defeq \bigoplus_{i \in I} H_{i}$ and the homomorphism $\pi \colon G \to \U(H)$ defined by \begin{displaymath}
	\pi(g)(x) \, \defeq \, (\pi_{i}(g)x_{i})_{i \in I} \qquad (g \in G, \, x \in H) .
\end{displaymath} By means of the natural embeddings, we regard the Hilbert spaces $H_{i}$ $(i \in I)$ as subspaces of $H$. Choose any (necessarily non-principal) ultrafilter $\mathcal{F}$ on $I$ containing \begin{displaymath}
	\{ \{ (E',k') \in I \mid (E,k) \preceq (E',k') \} \mid (E,k) \in I \} .
\end{displaymath} To deduce~\ref{theorem:kirchberg.c1}, consider the state \begin{displaymath}
	\mu \colon \, \UCB(\Sph(H)) \, \longrightarrow \, \C, \quad f \, \longmapsto \, \lim\nolimits_{i\to \mathcal{F}} \int f \, \mathrm{d}\nu_{p_{i}(H_{i})} .
\end{displaymath} In order to establish $\pi$-invariance of $\mu$, let $g \in G$. Then $\nu_{\pi_{i}(g)p_{i}(H_{i})} = \pi(g)_{\ast}(\nu_{p_{i}(H_{i})})$ for every $i \in I$, and so \begin{align*}
	\mu(f \circ \pi(g)) \, &= \, \lim\nolimits_{i\to \mathcal{F}} \int f \circ \pi(g) \, \mathrm{d}\nu_{p_{i}(H_{i})} \\
	& = \, \lim\nolimits_{i\to \mathcal{F}} \int f \, \mathrm{d}\nu_{\pi_{i}(g)p_{i}(H_{i})} \, \stackrel{\eqref{eq:kirchberg.invariance}+\ref{lemma:mass.transportation}}{=} \, \lim\nolimits_{i\to \mathcal{F}} \int f \, \mathrm{d}\nu_{p_{i}(H_{i})} \, = \, \mu(f)
\end{align*} for every $f \in \Lip_{1}(\Sph(H),[-1,1])$. Since $\Lip_{1}(\Sph(H),[-1,1])$ spans a $\Vert \cdot \Vert_{\infty}$-dense linear subspace of $\UCB(\Sph(H))$ (see, e.g.,~\cite[Lemma~5.20, p.~68]{PachlBook}), thus $\mu(f \circ \pi(g)) = \mu(f)$ for all $f \in \UCB(\Sph(H))$, so that $\pi(g).\mu = \mu$ as intended. Moreover, if $i=(E,k) \in I$ and $g \in E\setminus \{ e \}$, then \begin{align*}
	&\int \vert \langle x,\pi(g)x \rangle \vert \, \mathrm{d}\nu_{p_{i}(H_{i})}(x) \, = \, \int \vert \langle x,\pi_{i}(g)x \rangle \vert \, \mathrm{d}\nu_{p_{i}(H_{i})}(x) \\
	& \qquad \leq \, \tfrac{\vert \tr(\pi_{i}(g)p_{i})\vert}{\tr(p_{i})} + \tfrac{1}{k} + \nu_{p_{i}(H_{i})}\!\left(\left\{ x \in \Sph(p_{i}(H_{i})) \left\vert \, \left\lvert \langle x,\pi_{i}(g)x \rangle - \tfrac{\tr(\pi_{i}(g)p_{i})}{\tr(p_{i})} \right\rvert > \tfrac{1}{k} \right\} \!\right. \right) \\
	& \qquad \stackrel{\eqref{eq:kirchberg.freeness}+\ref{lemma:burton.concentration}}{\leq} \, \tfrac{2}{k} + 4\exp\!\left(-\tfrac{2\rk(p_{i})-1}{16k^{2}\Vert \pi_{i}(g)\Vert^{2}}\right)\! \, = \, \tfrac{2}{k} + 4\exp\!\left(-\tfrac{2\rk(p_{i})-1}{16k^{2}}\right)\! \, \stackrel{\eqref{eq:kirchberg.concentration}}{\leq} \, \tfrac{3}{k} ,
\end{align*} which entails that \begin{displaymath}
	\mu(x \mapsto \vert \langle x,\pi(g)x \rangle \vert) \, = \, \lim\nolimits_{i \to \mathcal{F}} \int \vert \langle x,\pi(g)x \rangle \vert \, \mathrm{d}\nu_{p_{i}(H_{i})}(x) \, = \, 0
\end{displaymath} for every $g \in G \setminus \{ e\}$, as desired in~\ref{theorem:kirchberg.c1}. We now proceed to deducing~\ref{theorem:kirchberg.e1}. For every $i = (E,k) \in I$, considering \begin{displaymath}
	B_{i} \, \defeq \, \left\{ x \in \Sph(p_{i}(H_{i})) \left\vert \, \forall g \in E\setminus \{ e \} \colon \, \lvert \langle x,\pi_{i}(g)x \rangle \rvert \leq \tfrac{2}{k} \right\} \!\right. \, \in \, \Borel(\Sph(p_{i}(H_{i}))),
\end{displaymath} we see that \begin{align*}
	\Sph(p_{i}(H_{i}))\setminus B_{i} \, &= \, \bigcup\nolimits_{g \in E\setminus \{ e \}} \left\{ x \in \Sph(p_{i}(H_{i})) \left\vert \, \lvert \langle x,\pi_{i}(g)x \rangle \rvert > \tfrac{2}{k} \right\} \!\right. \\
	&\stackrel{\eqref{eq:kirchberg.freeness}}{\subseteq} \, \bigcup\nolimits_{g \in E\setminus \{ e \}} \left\{ x \in \Sph(p_{i}(H_{i})) \left\vert \, \left\lvert \langle x,\pi_{i}(g)x \rangle - \tfrac{\tr(\pi_{i}(g)p_{i})}{\tr(p_{i})} \right\rvert > \tfrac{1}{k} \right\} \!\right.
\end{align*} and therefore \begin{align*}
	&\nu_{p_{i}(H_{i})}(\Sph(p_{i}(H_{i}))\setminus B_{i}) \\
	&\qquad \leq \, \sum\nolimits_{g \in E\setminus \{ e\}} \nu_{p_{i}(H_{i})}\!\left( \left\{ x \in \Sph(p_{i}(H_{i})) \left\vert \, \left\lvert \langle x,\pi_{i}(g)x \rangle - \tfrac{\tr(\pi_{i}(g)p_{i})}{\tr(p_{i})} \right\rvert > \tfrac{1}{k} \right\} \!\right. \right) \\
	&\qquad \stackrel{\ref{lemma:burton.concentration}}{\leq} \, 4\vert E \vert\exp\!\left(-\tfrac{2\rk(p_{i})-1}{16k^{2}\Vert \pi_{i}(g)\Vert^{2}}\right)\! \, = \, 4\vert E \vert\exp\!\left(-\tfrac{2\rk(p_{i})-1}{16k^{2}}\right)\! \, \stackrel{\eqref{eq:kirchberg.concentration}}{<} \, \tfrac{1}{\rk(p_{i})} \, = \, \tfrac{1}{\dim (p_{i}(H_{i}))} .
\end{align*} Hence, thanks to Lemma~\ref{lemma:milman}, for each $i \in I$ the set $B_{i}$ contains an orthonormal basis $X_{i}$ for $p_{i}(H_{i})$. In turn, $\bigcup_{i \in I} X_{i}$ is an orthonormal system in $H$, thus we find an orthonormal basis $X$ for $H$ with $\bigcup_{i \in I} X_{i} \subseteq X$. Consider the probability charge \begin{displaymath}
	\nu \colon \, \Pow(X) \, \longrightarrow \, [0,1], \quad A \, \longmapsto \, \lim\nolimits_{i \to \mathcal{F}} \tfrac{\vert A\cap X_{i}\vert}{\vert X_{i}\vert}
\end{displaymath} and note that \begin{align*}
	\nu_{\bullet}(a) \, &= \, \int \langle x,ax \rangle \, \mathrm{d}\nu(x) \, = \, \lim\nolimits_{i \to \mathcal{F}} \tfrac{1}{\vert X_{i}\vert}\sum\nolimits_{x\in X_{i}}\langle x,ax\rangle \, = \, \lim\nolimits_{i \to \mathcal{F}} \tfrac{1}{\tr(p_{i})}\tr(ap_{i}) \\
	& \stackrel{\ref{corollary:trace}}{=} \, \lim\nolimits_{i \to \mathcal{F}} \int \langle x,ax \rangle \, \mathrm{d}\nu_{p_{i}(H_{i})}(x) \, = \, \mu(x \mapsto \langle x,ax \rangle) \, = \, \mu_{\bullet}(a)
\end{align*} for every $a \in \B(H)$. In particular, since $\mu$ is $\pi$-invariant, $\nu_{\bullet} = \mu_{\bullet}$ is $\pi$-invariant by Lemma~\ref{lemma:states.from.means}. Finally, since $X_{i} \subseteq B_{i}$ for each $i \in I$, \begin{align*}
	\nu(\{ x \in X \mid \lvert \langle x,\pi(g)x \rangle \rvert \leq \epsilon \}) \, & = \, \lim\nolimits_{i \to \mathcal{F}} \tfrac{1}{\vert X_{i} \vert} \lvert \{ x \in X_{i} \mid \langle x,\pi(g)x \rangle \rvert \leq \epsilon \} \rvert \\
	& = \, \lim\nolimits_{i \to \mathcal{F}} \tfrac{1}{\vert X_{i} \vert} \lvert \{ x \in X_{i} \mid \langle x,\pi_{i}(g)x \rangle \rvert \leq \epsilon \} \rvert \, = \, 1 
\end{align*} whenever $g \in G\setminus \{ e \}$ and $\epsilon \in \R_{>0}$, which establishes~\ref{theorem:kirchberg.e1}.

\ref{theorem:kirchberg.e1}$\Longrightarrow$\ref{theorem:kirchberg.e2}. This is analogous to the proof of Theorem~\ref{theorem:elek.szabo}\ref{theorem:elek.szabo.2}$\Longrightarrow$\ref{theorem:elek.szabo.3}.

\ref{theorem:kirchberg.e2}$\Longrightarrow$\ref{theorem:kirchberg.d2}. This is analogous to the proof of Theorem~\ref{theorem:elek.szabo}\ref{theorem:elek.szabo.3}$\Longrightarrow$\ref{theorem:elek.szabo.1}. \end{proof}

\begin{remark}\label{remark:ring.homomorphism.kirchberg} In the conditions~\ref{theorem:kirchberg.c1} and~\ref{theorem:kirchberg.c2} of Theorem~\ref{theorem:kirchberg}, one may additionally require that $\mu$ be a ring homomorphism from $\UCB(\Sph(H))$ to $\C$. Indeed, as the family $(H_{i},\pi_{i},p_{i})_{i \in I}$ selected in the proof of the implication \ref{theorem:kirchberg.d2}$\Longrightarrow$\ref{theorem:kirchberg.c1} satisfies~\eqref{eq:kirchberg.concentration}, it follows that \begin{displaymath}
	\dim(H_{i})\, \longrightarrow \, \infty \quad (i \to (I,{\preceq})) ,
\end{displaymath} whence the argument in Remark~\ref{remark:ring.homomorphism} shows that the resulting map $\mu$ constitutes a ring homomorphism. \end{remark}

Let us recall that a group $G$ is said to be \emph{Kazhdan} or have \emph{property $(T)$} if every unitary representation of $G$ with almost invariant vectors admits an invariant vector of norm one. In~\cite{ElekSzabo}, Elek and Szabó considered the class of groups $G$ such that \begin{equation}\label{eq:elek.szabo.factorization.property}
	\begin{split}
		&\exists X \text{ set} \ \, \exists \pi \colon G \to \Sym(X) \text{ homomorphism}  \\
		&\exists \mu \text{ $\pi$-invariant probability charge on } \Pow(X)  \\
		&\forall g \in G\setminus \{ e \} \colon \ \mu (\{ x \in X \mid x \ne \pi(g)x \}) = 1 .
	\end{split}
\end{equation} They observed that this class contains all amenable groups as well as all residually finite ones~\cite[Proposition~4.3]{ElekSzabo}, and they showed that every Kazhdan group satisfying~\eqref{eq:elek.szabo.factorization.property} must be residually finite~\cite[Proposition~4.4]{ElekSzabo}. From Lemma~\ref{lemma:elek.szabo.comparison}\ref{lemma:elek.szabo.comparison.1}$+$\ref{lemma:elek.szabo.comparison.3} and Theorem~\ref{theorem:kirchberg}, it follows that groups satisfying~\eqref{eq:elek.szabo.factorization.property} have Kirchberg's factorization property. In view of this, \cite[Proposition~4.4]{ElekSzabo} becomes a direct consequence of the following earlier theorem by Kirchberg~\cite{Kirchberg}, of which we provide a new proof.

\begin{thm}[{\cite[Theorem~1.1]{Kirchberg}}]\label{theorem:kazhdan} Every Kazhdan group with Kirchberg's factorization property is residually finite. \end{thm} 

\begin{proof} Let $G$ be a Kazhdan group with Kirchberg's factorization property. According to Theorem~\ref{theorem:kirchberg}, there exist a Hilbert space $H$, a representation $\pi \colon G \to \U(H)$ and a $\pi$-invariant state $\phi \colon \B(H) \to \C$ such that \begin{displaymath}
	\forall g \in G \setminus \{ e \} \colon \quad \phi((1-\pi(g))^{\ast}(1-\pi(g))) > 0 .
\end{displaymath} Consider the directed set $\mathcal{K}$ of $\pi$-invariant finite-dimensional linear subspaces of $H$. Then the closed linear subspace \begin{displaymath}
	H_{0} \, \defeq \, \overline{\bigcup \mathcal{K}} \, = \, \overline{\sum \mathcal{K}} \, \leq \, H
\end{displaymath} is $\pi$-invariant. Concerning the orthogonal projection $p \in \Pro(H)$ onto $H_{0}$, this means that $\pi(g)p = p\pi(g)$ for every $g \in G$. We argue that $\phi(p) = 1$. For contradiction, suppose that $\phi(p) < 1$, i.e., $\phi(1-p) = 1-\phi(p) > 0$. Consider the $\pi$-invariant closed linear subspace $H_{1} \defeq H_{0}^{\perp} = (1-p)(H)$ of $H$ along with the induced representation $\rho \colon G \to \U(H_{1}), \, g \mapsto \pi(g)\vert_{H_{1}}$. Since \begin{displaymath}
	\psi \colon \, \B(H_{1}) \, \longrightarrow \, \C, \quad a \, \longmapsto \, \tfrac{1}{\phi(1-p)}\phi(a(1-p))
\end{displaymath} is a well-defined state satisfying \begin{align*}
	\psi(\rho(g)^{\ast}a\rho(g)) \, &= \, \tfrac{1}{\phi(1-p)}\phi(\rho(g)^{\ast}a\rho(g)(1-p)) \, = \, \tfrac{1}{\phi(1-p)}\phi(\pi(g)^{\ast}a\pi(g)(1-p)) \\
		& = \, \tfrac{1}{\phi(1-p)}\phi(a\pi(g)(1-p)\pi(g)^{\ast}) \, = \, \tfrac{1}{\phi(1-p)}\phi(a(1-\pi(g)p\pi(g)^{\ast})) \\
		& = \, \tfrac{1}{\phi(1-p)}\phi(a(1-p)) \, = \, \psi(a)
\end{align*} for every $a \in \B(H_{1})$, we see that $\rho$ is amenable. Since $G$ is a Kazhdan group, $\rho$ must possess a non-trivial finite-dimensional subrepresentation by~\cite[Corollary~5.9]{bekka}. As $H_{1}$ does not contain a non-zero $\pi$-invariant finite-dimensional linear subspace, this gives the desired contradiction. Hence, $\phi(p) = 1$. Thanks to the Cauchy--Schwartz inequality, it follows that \begin{align*}
	&\vert \phi(a(1-p)) \vert \, \leq \, \sqrt{\phi(a^{\ast}a)} \sqrt{\phi((1-p)^{\ast}(1-p))} \, = \, \sqrt{\phi(a^{\ast}a)} \sqrt{\phi(1-p)} \, = \, 0 , \\
	&\vert \phi((1-p)a) \vert \, \leq \, \sqrt{\phi((1-p)^{\ast}(1-p))}\sqrt{\phi(a^{\ast}a)} \, = \, \sqrt{\phi(1-p)} \sqrt{\phi(a^{\ast}a)} \, = \, 0 
\end{align*} for all $a \in \B(H)$, wherefore \begin{displaymath}
	\phi(a) \, = \, \phi(pap) + \phi((1-p)ap) + \phi(a(1-p)) \, = \, \phi(pap)
\end{displaymath} for every $a \in \B(H)$. We conclude that the homomorphism \begin{displaymath}
	G \, \longrightarrow \, \prod\nolimits_{K \in \mathcal{K}} \U(K), \quad g \, \longmapsto \, (\pi(g)\vert_{K})_{K \in \mathcal{K}}
\end{displaymath} is injective: indeed, if $g \in G\setminus \{ e \}$, then \begin{align*}
	\phi(((1-\pi(g))p)^{\ast}(1-\pi(g))p) \, &= \, \phi(p(1-\pi(g))^{\ast}(1-\pi(g))p) \\
		& = \, \phi((1-\pi(g))^{\ast}(1-\pi(g))) \, > \, 0
\end{align*} and thus $(1-\pi(g))p \ne 0$, i.e., $\pi(g)\vert_{H_{0}} \ne \id_{H_{0}}$, whence $\pi(g)\vert_{K} \ne \id_{K}$ for some $K \in \mathcal{K}$. Being a Kazhdan group, $G$ is finitely generated by~\cite[Theorem~1.3.1, p.~36]{BekkaBook}. Since finitely generated subgroups of $\prod\nolimits_{K \in \mathcal{K}} \U(K)$ are residually finite due to Malcev's theorem (see, e.g.,~\cite[Theorem~6.4.13, p.~231]{BrownOzawa}), this completes the proof. \end{proof}

The reader is referred to~\cite{GromovHyperbolic,OzawaProc} for examples of infinite simple Kazhdan groups, and to~\cite{Thom} for an example of a sofic (hence hyperlinear) Kazhdan group without Kirchberg's factorization property.

\appendix 

\section{Proof of Lemma~\ref{lemma:pestov}}\label{section:proof.pestov.lemma}

This section provides a proof of Pestov's Lemma~\ref{lemma:pestov}, following the lines of~\cite{Pestov00}, only providing some additional details. We claim no originality. Moreover, for the lack of a suitable reference, we also record an elementary observation used in the proof of Lemma~\ref{lemma:mass.transportation}, namely Lemma~\ref{lemma:low.energy.conjugation}.

\begin{remark}\label{remark:norm.decomposition} Let $H$ be a Hilbert space, let $p_{1},\ldots,p_{n} \in \Pro(H)$ be pairwise orthogonal, and let $a_{1},\ldots,a_{n} \in \B(H)$. \begin{enumerate}
	\item\label{remark:norm.decomposition.1} If $p_{i}a_{i}p_{i} = a_{i}$ for each $i \in \{ 1,\ldots,n \}$, then \begin{displaymath}
			\qquad \Vert a_{1}+\ldots +a_{n} \Vert \, = \, \max\nolimits_{i \in \{ 1,\ldots,n \}} \Vert a_{i} \Vert .
		\end{displaymath}
	\item\label{remark:norm.decomposition.2} If $\rk(p_{i}) < \infty$ and $a_{i}p_{i} = a_{i}$ for each $i \in \{ 1,\ldots,n \}$, then \begin{displaymath}
			\qquad \Vert a_{1}+\ldots +a_{n} \Vert_{2}^{2} \, = \, \sum\nolimits_{i = 1}^{n} \Vert a_{i} \Vert_{2}^{2} .
		\end{displaymath}
\end{enumerate} \end{remark}

\begin{lem}[{\cite[Lemma~5.4]{Pestov00}}]\label{lemma:wedin} Let $H$ be a Hilbert space and let $p,q \in \Pro(H)$ with $n \defeq \rk(p) = \rk(q) < \infty$. Then there exist $p_{1},\ldots,p_{n},q_{1},\ldots,q_{n} \in \Pro(H)$ such that \begin{itemize}
	\item[---\,] $\rk(p_{i}) = 1 = \rk(q_{i})$ for each $i \in \{ 1,\ldots,n\}$,
	\item[---\,] $p_{i} \perp p_{j}$ and $p_{i} \perp q_{j}$ and $q_{i} \perp q_{j}$ for any two distinct $i,j \in \{ 1,\ldots,n \}$,
	\item[---\,] $p = p_{1} + \ldots + p_{n}$ and $q = q_{1} + \ldots + q_{n}$.
\end{itemize} In particular, \begin{displaymath}
	\Vert p-q \Vert \, = \, \max\nolimits_{i \in \{ 1,\ldots,n \}} \Vert p_{i}-q_{i} \Vert .
\end{displaymath} \end{lem}

\begin{proof} The first assertion is due to~\cite[Lemma~5.4]{Pestov00} (see also~\cite[Section~2]{Wedin}). The second then follows in combination with Remark~\ref{remark:norm.decomposition}\ref{remark:norm.decomposition.1}. \end{proof}

\begin{lem}[{\cite[Section~3]{Os}}]\label{lemma:os} Let $H$ be a Hilbert space, let $p,q \in \Pro(H)$, let $X \defeq p(H)$ and $Y \defeq q(H)$. Moreover, define\footnote{Note that $\Omega(X,Y)$ is precisely the Hausdorff distance between $\Sph(X)$ and $\Sph(Y)$ in $H$.} \begin{align*}
	\Theta(X,Y) \, &\defeq \, \max\!\left\{\sup\nolimits_{x\in\Sph(X)} \dist(x,Y), \,\sup\nolimits_{y\in\Sph(Y)} \dist(y,X) \right\} , \\
	\Omega(X,Y) \, &\defeq \, \max\!\left\{\sup\nolimits_{x\in\Sph(X)} \dist(x,\mathbb{S}(Y)), \, \sup\nolimits_{y\in\Sph(Y)} \dist(y,\Sph(X))\right\}.
\end{align*} Then \begin{displaymath}
	\Vert p-q\Vert \, = \, \Theta(X,Y) \, \leq \, \Omega(X,Y) \, \leq \, 2\Theta(X,Y) \, = \, 2\Vert p-q\Vert.
\end{displaymath} \end{lem}

\begin{proof} See~\cite[3.4(h)\,+\,3.12(b)]{Os}. \end{proof}

\begin{lem}\label{lemma:projection.distance.1} Let $H$ be a Hilbert space and let $x,y \in \Sph(H)$. Then \begin{displaymath} 
	\Vert x-\langle y,x\rangle y\Vert^{2} \, = \, 1-\vert \langle x,y\rangle \vert^{2} \, \in \, [0,1] .
\end{displaymath} \end{lem}

\begin{proof} We observe that \begin{align*}
	\Vert x-\langle y,x\rangle y\Vert^{2} \, &= \, \langle  x-\langle y,x\rangle y, x-\langle y,x\rangle y \rangle \\
		&= \, \langle x,x\rangle-\langle x,\langle y,x\rangle y\rangle + \langle \langle y,x\rangle y,\langle y,x\rangle y\rangle - \langle \langle y,x\rangle y,x\rangle \\
		&= \, 1-\langle y,x\rangle\langle x,y\rangle+\overline{\langle y,x\rangle}\langle y,x\rangle1-\overline{\langle y,x\rangle}\langle y,x\rangle \\
		&= \, 1-\langle y,x\rangle\langle x,y\rangle \, = \, 1-\overline{\langle x,y\rangle}\langle x,y\rangle \, = \, 1-\vert \langle x,y\rangle \vert^{2} \, \in \, [0,1] .\qedhere
\end{align*} \end{proof}

\begin{lem}\label{lemma:projection.distance.2} Let $H$ be a Hilbert space and let $p,q \in \Pro(H)$ with $\rk(p) = \rk(q) = 1$. \begin{enumerate}
	\item\label{lemma:projection.distance.2.1} For any $x \in \Sph(p(H))$ and $y \in \Sph(q(H))$, \begin{displaymath} 
			\qquad \Vert p-q \Vert \, = \, \Vert x-\langle y,x\rangle y\Vert \, \in \, [0,1] .
		\end{displaymath}
	\item\label{lemma:projection.distance.2.2} $\Vert p-q \Vert_{2} = \sqrt{2} \Vert p-q \Vert$.
\end{enumerate} \end{lem}

\begin{proof} \ref{lemma:projection.distance.2.1} Let $x \in \Sph(p(H))$ and $y \in \Sph(q(H))$. Since $qx = \langle y,x\rangle y$, we see that \begin{displaymath}
	\dist(x,q(H)) \, = \, \Vert x-qx \Vert \, = \, \Vert x-\langle y,x\rangle y\Vert \, \stackrel{\ref{lemma:projection.distance.1}}{=} \, \sqrt{ 1-\vert \langle x,y\rangle \vert^{2}} ,
\end{displaymath} wherefore \begin{align*}
	\sup\nolimits_{x' \in \Sph(p(H))} \dist(x',q(H)) \, &= \, \sup\nolimits_{c \in \Sph(\C)} \dist(cx,q(H)) \\
		&= \, \sup\nolimits_{c \in \Sph(\C)} \dist\!\left(x,c^{-1}q(H)\right) \\
		&= \, \dist(x,q(H)) \, = \, \sqrt{ 1-\vert \langle x,y\rangle \vert^{2}} .
\end{align*} By symmetry, this also implies that \begin{displaymath}
	\sup\nolimits_{y' \in \Sph(q(H))} \dist(y',p(H)) \, = \, \sqrt{ 1-\vert \langle y,x\rangle \vert^{2}} \, = \, \sqrt{ 1-\vert \langle x,y\rangle \vert^{2}} .
\end{displaymath} Consequently, \begin{displaymath}
	\Vert p-q \Vert \, \stackrel{\ref{lemma:os}}{=} \, \Theta(p(H),q(H)) \, = \, \sqrt{ 1-\vert \langle x,y\rangle \vert^{2}} \, \stackrel{\ref{lemma:projection.distance.1}}{=} \, \Vert x-\langle y,x\rangle y\Vert \, \in \, [0,1] .
\end{displaymath}
	
\ref{lemma:projection.distance.2.2} As the claim is trivial if $p=q$, we may and will assume that $p \ne q$. Consider any $x \in \Sph(p(H))$ and $y \in \Sph(q(H))$. Then \begin{align*}
	\Vert p-q \Vert_{2}^{2} \, &= \, \tr((p-q)^{\ast}(p-q)) \, = \, \Vert (p-q)x \Vert^{2} + \Vert (p-q)y \Vert^{2} \\
		& = \, \Vert x-\langle y,x \rangle y \Vert^{2} + \Vert y-\langle x,y \rangle x \Vert^{2} \, \stackrel{\ref{lemma:projection.distance.2.1}}{=} \, 2 \Vert p-q \Vert^{2} .
\end{align*} Hence, $\Vert p-q \Vert_{2} = \sqrt{2} \Vert p-q \Vert$ as desired. \end{proof}

\begin{lem}\label{lemma:projection.distance.3} Let $H$ be a Hilbert space, let $p,q \in \Pro(H)$ with $\rk(p) = \rk(q) = 1$ and $p\ne q$, and let $x \in p(H)$ and $y \in q(H)$ with $\Vert x+y \Vert \leq 1$. Then \begin{displaymath}
	\Vert x \Vert \, \leq \, \Vert p-q \Vert^{-1} .
\end{displaymath} \end{lem}

\begin{proof} Evidently, if $y=0$, then $\Vert x \Vert = \Vert x+y \Vert \leq 1 \stackrel{\ref{lemma:projection.distance.2}\ref{lemma:projection.distance.2.1}}{\leq} \Vert p-q \Vert^{-1}$. Also, the desired inequality is trivial for $x=0$. Consequently, we may and will assume that $x \ne 0 \ne y$. From \begin{displaymath}
	\left( x- \left\langle \tfrac{1}{\Vert y \Vert}y,x \right\rangle \tfrac{1}{\Vert y \Vert} y \right) \perp y ,
\end{displaymath} we infer that \begin{align*}
	1 \, &\geq \, \Vert x+y \Vert^{2} \, = \, \left\lVert x- \left\langle \tfrac{1}{\Vert y \Vert}y,x \right\rangle \tfrac{1}{\Vert y \Vert} y \right\rVert^{2} + \left\lVert \left( \left\langle \tfrac{1}{\Vert y \Vert}y,x \right\rangle \tfrac{1}{\Vert y \Vert} + 1 \right)y\right\rVert^{2} \\
		&\geq \, \left\lVert x- \left\langle \tfrac{1}{\Vert y \Vert}y,x \right\rangle \tfrac{1}{\Vert y \Vert} y \right\rVert^{2} \, = \, \Vert x \Vert^{2} \left\lVert \tfrac{1}{\Vert x \Vert} x- \left\langle \tfrac{1}{\Vert y \Vert}y,\tfrac{1}{\Vert x \Vert}x \right\rangle \tfrac{1}{\Vert y \Vert} y \right\rVert^{2} \\
		& \stackrel{\ref{lemma:projection.distance.2}\ref{lemma:projection.distance.2.1}}{=} \, \Vert x \Vert^{2}\Vert p-q \Vert^{2} , 
\end{align*} hence $\Vert x \Vert \leq \Vert p-q \Vert^{-1}$. \end{proof}

\begin{proof}[Proof of Lemma~\ref{lemma:pestov}] Consider the directed set $\mathcal{K}$ of $\pi$-invariant finite-dimensional linear subspaces of $H$. Then $H_{0} \defeq \bigcup \mathcal{K} = \sum \mathcal{K}$ is a $\pi$-invariant linear subspace of $H$ with $m \defeq \dim(H_{0}) = \sup_{K \in \mathcal{K}} \dim(K)$. If $m = \infty$, then the desired conclusion follows immediately. Suppose now that $m < \infty$. Then the restriction of the representation $\pi$ to $H_{0}^{\perp}$ has to be amenable by our hypothesis, while by construction it has no non-trivial finite-dimensional subrepresentation. Evidently, it suffices to show the existence of the desired projections for this subrepresentation of $\pi$. Therefore, without loss of generality, we may and will henceforth assume that $\pi$ is amenable and contains no non-trivial finite-dimensional subrepresentation.
	
We consider the directed set $(I,{\preceq})$ given by \begin{displaymath}
	I \, \defeq \, \{(E,\epsilon)\mid E\subseteq G \text{ compact}, \, \epsilon \in \R_{>0}\}
\end{displaymath} and \begin{displaymath}
	(E,\epsilon) \preceq (E',\epsilon') \quad :\Longleftrightarrow \quad E \subseteq E' \ \wedge \ \epsilon' \leq \epsilon \qquad ((E,\epsilon),(E',\epsilon') \in I) .
\end{displaymath} For each $\iota = (E, \epsilon) \in I$, we know from Lemma~\ref{lemma:folner} that the set \begin{displaymath}
	F_{\iota} \, \defeq \, \{ p \in \Pro(H) \mid 0 < \rk(p) < \infty, \ \forall g\in E\colon \, \Vert \pi(g)p\pi(g)^{\ast}-p\Vert_{1} \leq \epsilon \Vert p \Vert_{1} \}
\end{displaymath} is non-empty. Since the map \begin{displaymath}
	(I,{\preceq}) \, \longrightarrow \, ({\N_{>0}} \cup {\{ \infty\}},{\leq}), \quad \iota \, \longmapsto \, \sup\{ \rk(p) \mid p \in F_{\iota}\} 
\end{displaymath} is antitone, the net $(\sup\{ \rk(p) \mid p \in F_{\iota}\})_{\iota \in (I,{\preceq})}$ converges in the discrete topology to \begin{displaymath}
	n \, \defeq \, \inf\nolimits_{\iota\in I} \sup\{ \rk(p) \mid p \in F_{\iota}\} \, \in \, {\N_{>0}} \cup {\{ \infty\}}.
\end{displaymath} If $n = \infty$, then the claim of the lemma is established. For contradiction, we now assume that $n < \infty$. Let $\iota_{0} = (E_{0},\epsilon_{0}) \in I$ with $\sup\{ \rk(p) \mid p \in F_{\iota_{0}} \} = n$.
	
We equip $X \defeq \{ p \in \Pro(H) \mid \rk(p) = n \}$ with the complete metric induced by $\Vert \cdot \Vert$. As the latter is uniformly equivalent on $X$ to the one induced by $\Vert \cdot \Vert_{1}$, the filter $\mathcal{F}$ on $X$ generated by $\{ F_{\iota} \cap X \mid \iota \in I \}$ coincides with the one generated by the sets \begin{displaymath}
	F'_{\iota} \, \defeq \, \{ p \in X \mid \forall g\in E\colon \, \Vert \pi(g)p\pi(g)^{\ast}-p\Vert \leq \epsilon \} \qquad (\iota = (E,\epsilon) \in I) .
\end{displaymath} Note that the map $(I,{\preceq}) \to (\R_{\geq 0},{\leq}), \, \iota \mapsto \diam(F'_{\iota})$ is antitone. If \begin{displaymath}
	\inf\nolimits_{\iota \in I} \diam(F'_{\iota}) \, = \, 0 ,
\end{displaymath} i.e., $\mathcal{F}$ were a Cauchy filter on $X$, then the thus existing limit $p \defeq \lim \mathcal{F} \in X$ would satisfy $\pi(g)p\pi(g)^{\ast} = p$ for each $g \in G$, whence $\pi\vert_{p(H)}$ would constitute a non-trivial finite-dimensional subrepresentation of $\pi$, contradicting our hypothesis. Therefore, $\inf_{\iota \in I} \diam(F'_{\iota}) > 0$. Consequently, we find $(p_{\iota})_{\iota \in I}, \, (q_{\iota})_{\iota \in I} \in \prod_{\iota \in I} F_{\iota}'$ such that \begin{equation}\label{eq:positive.infimum}
	\inf\nolimits_{\iota \in I} \Vert p_{\iota} - q_{\iota} \Vert \, > \, 0 .
\end{equation} By Lemma~\ref{lemma:wedin}, for each $\iota \in I$ there exist $p_{\iota,1},\ldots,p_{\iota,n},q_{\iota,1},\ldots,q_{\iota,n} \in \Pro(H)$ such that \begin{itemize}
	\item[---\,] $\rk(p_{\iota,i}) = 1 = \rk(q_{\iota,i})$ for each $i \in \{ 1,\ldots,n\}$,
	\item[---\,] $p_{\iota,i} \perp p_{\iota,j}$ and $p_{\iota,i} \perp q_{\iota,j}$ and $q_{\iota,i} \perp q_{\iota,j}$ for any two distinct $i,j \in \{ 1,\ldots,n \}$,
	\item[---\,] $p_{\iota} = p_{\iota,1} + \ldots + p_{\iota,n}$ and $q_{\iota} = q_{\iota,1} + \ldots + q_{\iota,n}$,
\end{itemize} and necessarily, \begin{equation}\label{eq:distance.decomposition}
	\Vert p_{\iota}-q_{\iota} \Vert \, = \, \max\nolimits_{i \in \{ 1,\ldots,n \}} \Vert p_{\iota,i}-q_{\iota,i} \Vert .
\end{equation} Since $[0,1]^{n}$ is compact, the net $(\Vert p_{\iota,1}-q_{\iota,1} \Vert, \ldots ,\Vert p_{\iota,n}-q_{\iota,n} \Vert)_{\iota \in (I,{\preceq})}$ admits an accumulation point $d = (d_{1},\ldots,d_{n}) \in [0,1]^{n}$. From~\eqref{eq:positive.infimum} and~\eqref{eq:distance.decomposition}, we infer that~$d \ne 0$. In turn, $J \defeq \{ i \in \{ 1,\ldots,n \} \mid d_{i} \ne 0 \}$ is non-empty and $\delta \defeq \min_{i \in J} d_{i} >0$.
	
Due to $d$ being an accumulation point, there exists $\iota = (E,\epsilon) \in I$ such that \begin{itemize}
	\item[---\,] $E_{0} \cup E_{0}^{-1} \subseteq E$,
	\item[---\,] $\epsilon \leq \tfrac{\epsilon_{0}\delta}{32n}$,
	\item[---\,] $\sum\nolimits_{i \in \{ 1,\ldots,n\}\setminus J} \Vert p_{\iota,i} - q_{\iota,i} \Vert^{2} \leq \bigl(\tfrac{\epsilon_{0} \delta}{16n}\bigr)^{2}$, and
	\item[---\,] $\Vert p_{\iota,i} - q_{\iota,i} \Vert \geq \tfrac{\delta}{2}$ for each $i \in J$.
\end{itemize} Consider the linear subspace \begin{displaymath}
	K_{\iota} \, \defeq \, p_{\iota}(H) + \sum\nolimits_{i \in J} q_{\iota,i}(H) \, = \, p_{\iota}(H) + \left( \sum\nolimits_{i \in J} q_{\iota,i}\right)\!(H)
\end{displaymath} and note that \begin{displaymath}
	n \, \stackrel{J\ne \emptyset}{<} \, \dim(K_{\iota}) \, \leq \, 2n .
\end{displaymath} For every $x \in \Sph(q_{\iota}(H))$, considering $y \defeq \left( \sum_{i \in \{ 1,\ldots,n\}\setminus J} p_{\iota,i}\right)\!x + \left( \sum_{i \in J} q_{\iota,i} \right)\!x \in K_{\iota}$ we observe that \begin{align*}
	\left\lVert x-y \right\rVert^{2} \, &= \, \left\lVert \sum\nolimits_{i \in \{ 1,\ldots,n\}\setminus J} (q_{\iota,i}-p_{\iota,i})x \right\rVert^{2} \, = \, \sum\nolimits_{i \in \{ 1,\ldots,n\}\setminus J} \lVert (q_{\iota,i}-p_{\iota,i})x \rVert^{2} \\
		& \leq \, \sum\nolimits_{i \in \{ 1,\ldots,n\}\setminus J} \lVert q_{\iota,i}-p_{\iota,i} \rVert^{2} \, \leq \, \bigl(\tfrac{\epsilon_{0} \delta}{16n}\bigr)^{2} .
\end{align*} This shows that \begin{equation}\label{eq:approximation.spheres}
	\sup\nolimits_{x\in \Sph(q_{\iota}(H))}\dist(x,K_{\iota}) \, \leq \, \tfrac{\epsilon_{0} \delta}{16n} .
\end{equation}
	
Next we prove that \begin{equation}\label{eq:norm.control}
	\forall x \in p_{\iota}(H) \ \forall y \in \sum\nolimits_{i \in J} q_{\iota,i}(H)\colon \quad \Vert x+y \Vert \leq 1 \ \Longrightarrow \ \max\{ \Vert x \Vert, \Vert y \Vert \} \leq 2n\delta^{-1} .
\end{equation} To this end, let $x \in p_{\iota}(H)$ and $y \in \sum\nolimits_{i \in J} q_{\iota,i}(H)$ with $\Vert x+y \Vert \leq 1$. From \begin{displaymath}
	1 \, \geq \, \Vert x+y \Vert^{2} \, = \, \left\lVert \sum\nolimits_{i \in \{ 1,\ldots,n\}\setminus J} p_{\iota,i}x \right\rVert^{2} + \sum\nolimits_{i \in J} \Vert p_{\iota,i}x + q_{\iota,i}y \Vert^{2} ,
\end{displaymath} we infer that \begin{displaymath}
	\left\lVert \sum\nolimits_{i \in \{ 1,\ldots,n\}\setminus J} p_{\iota,i}x \right\rVert \, \leq \, 1 \, \stackrel{\ref{lemma:projection.distance.2}\ref{lemma:projection.distance.2.1}}{\leq} \, \delta^{-1} \, < \, 2\delta^{-1}
\end{displaymath} and, for every $i \in J$, \begin{displaymath}
	\max\{ \Vert p_{\iota,i}x \Vert, \Vert q_{\iota,i}y \Vert \} \, \stackrel{\ref{lemma:projection.distance.3}}{\leq} \, \Vert p_{\iota,i}-q_{\iota,i} \Vert^{-1} \, \leq \, 2\delta^{-1} .
\end{displaymath} Consequently, \begin{align*}
	\Vert x \Vert \, &\leq \, \left\lVert \sum\nolimits_{i \in \{ 1,\ldots,n\}\setminus J} p_{\iota,i}x \right\rVert + \sum\nolimits_{i \in J} \Vert p_{\iota,i}x \Vert \, \leq \, 2n\delta^{-1} , \\
	\Vert y \Vert \, &\leq \, \sum\nolimits_{i \in J} \Vert q_{\iota,i}y \Vert \, \leq \, 2n\delta^{-1} 
\end{align*} as desired.
	
We now claim that \begin{equation}\label{eq:almost.invariant.space}
	\forall g \in E \ \forall z \in \Sph(K_{\iota}) \colon \qquad \dist(\pi(g)z,K_{\iota}) \, \leq \, \tfrac{\epsilon_{0}}{2} .
\end{equation} To see this, let $z \in \Sph(K_{\iota})$. Then there exist $x \in p_{\iota}(H)$ and $y \in \sum\nolimits_{i \in J} q_{\iota,i}(H)$ such that $z = x+y$. In particular, $1 = \Vert z \Vert = \Vert x+y \Vert$ and thus $\max\{ \Vert x \Vert, \Vert y \Vert \} \leq 2n\delta^{-1}$ by~\eqref{eq:norm.control}. Now, let $g \in E$. Since $\pi(g)x \in \Vert x \Vert \Sph(\pi(g)p_{\iota}(H))$ and $\pi(g)y \in \Vert y \Vert \Sph(\pi(g) q_{\iota}(H))$, we see that \begin{align*}
	&\dist(\pi(g)x,K_{\iota}) \, \leq \, \dist(\pi(g)x,p_{\iota}(H)) \\
	& \quad \leq \, \sup\nolimits_{x' \in \Sph(\pi(g)p_{\iota}(H))} \dist(\Vert x \Vert x',p_{\iota}(H)) \, = \, \Vert x \Vert \sup\nolimits_{x' \in \Sph(\pi(g)p_{\iota}(H))} \dist(x',p_{\iota}(H)) \\
	& \quad \leq \, 2n\delta^{-1} \Theta (\pi(g)p_{\iota}(H),p_{\iota}(H)) \, \stackrel{\ref{lemma:os}}{=} \, 2n\delta^{-1} \Vert \pi(g)p_{\iota}\pi(g)^{\ast}-p_{\iota}\Vert \, \leq \, \tfrac{\epsilon_{0}}{4}
\end{align*} and \begin{align*}
	&\dist(\pi(g)y,K_{\iota}) \, \leq \, \sup\nolimits_{y' \in \Sph(\pi(g)q_{\iota}(H))} \dist(\Vert y \Vert y',K_{\iota}) \\
	& \qquad = \, \Vert y \Vert \sup\nolimits_{y' \in \Sph(\pi(g)q_{\iota}(H))} \dist(y',K_{\iota}) \\
	& \qquad \leq \, 2n\delta^{-1} \left( \sup\nolimits_{y' \in \Sph(\pi(g)q_{\iota}(H))} \dist(y',\Sph(q_{\iota}(H))) + \sup\nolimits_{y'' \in \Sph(q_{\iota}(H))} \dist(y'',K_{\iota}) \right) \\
	& \qquad \stackrel{\eqref{eq:approximation.spheres}}{\leq} \, 2n\delta^{-1} \Omega(\pi(g)q_{\iota}(H),q_{\iota}(H)) + 2n\delta^{-1} \tfrac{\epsilon_{0} \delta}{16n} \\
	& \qquad \stackrel{\ref{lemma:os}}{\leq} \, 4n\delta^{-1} \Vert \pi(g)q_{\iota}\pi(g)^{\ast}-q_{\iota}\Vert + \tfrac{\epsilon_{0}}{8} \, \leq \, \tfrac{\epsilon_{0}}{4} ,
\end{align*} whence \begin{displaymath}
	\dist(\pi(g)z,K_{\iota}) \, \leq \, \dist(\pi(g)x,K_{\iota}) + \dist(\pi(g)y,K_{\iota}) \, \leq \, \tfrac{\epsilon_{0}}{2} .
\end{displaymath} This finishes the proof of~\eqref{eq:almost.invariant.space}.
	
We conclude that \begin{equation}\label{eq:almost.invariant.space.2}
	\forall g \in E_{0} \colon \qquad \Theta (\pi(g)K_{\iota},K_{\iota}) \, \leq \, \tfrac{\epsilon_{0}}{2} .
\end{equation} Indeed, if $g \in E_{0}$ and thus $g, g^{-1} \in E$, then \begin{align*}
	&\Theta (\pi(g)K_{\iota},K_{\iota}) \, \stackrel{\ref{lemma:os}}{=} \, \max\!\left\{ \sup\nolimits_{x\in\Sph(\pi(g)K_{\iota})} \dist(x,K_{\iota}), \, \sup\nolimits_{x\in\Sph(K_{\iota})} \dist(x,\pi(g)K_{\iota}) \right\} \\
	& \qquad = \, \max\!\left\{ \sup\nolimits_{z\in\Sph(K_{\iota})} \dist(\pi(g)z,K_{\iota}), \, \sup\nolimits_{z\in\Sph(K_{\iota})} \dist\!\left(\pi\!\left(g^{-1}\right)\!z,K_{\iota}\right) \right\} \, \stackrel{\eqref{eq:almost.invariant.space}}{\leq} \, \tfrac{\epsilon_{0}}{2} .
\end{align*}
	
Finally, let $r_{\iota} \in \Pro(H)$ denote the orthogonal projection onto $K_{\iota}$. For every $g \in E_{0}$, \begin{align*}
	\Vert \pi(g)r_{\iota}\pi(g)^{\ast} - r_{\iota} \Vert \, &\stackrel{\ref{lemma:os}}{=} \, \Theta(\pi(g)K_{\iota},K_{\iota}) \, \stackrel{\eqref{eq:almost.invariant.space.2}}{\leq} \, \tfrac{\epsilon_{0}}{2}
\end{align*} and hence \begin{displaymath}
	\Vert \pi(g)r_{\iota}\pi(g)^{\ast} -r_{\iota} \Vert_{1} \, \leq \, 2 \rk(r_{\iota}) \Vert \pi(g)r_{\iota}\pi(g)^{\ast} - r_{\iota} \Vert \leq \, \epsilon_{0} \Vert r_{\iota}\Vert_{1} .
\end{displaymath} That is, $r_{\iota} \in F_{\iota_{0}}$. As $\rk(r_{\iota}) = \dim(K_{\iota}) > n$ and $\sup\{ \rk(p) \mid p \in F_{\iota_{0}} \} = n$, this yields the desired contradiction and thus completes the proof. \end{proof}

The following well-known fact is used in the proof of Lemma~\ref{lemma:mass.transportation}

\begin{lem}\label{lemma:low.energy.conjugation} Let $H$ be a Hilbert space and let $p,q \in \Pro(H)$ with $\rk(p) = \rk(q) < \infty$. Then there exists $u \in \U(H)$ such that $upu^{\ast} = q$ and $\Vert (1-u)p \Vert_{2} = \sqrt{2}\Vert p-q \Vert_{2}$. \end{lem}

\begin{proof} Thanks to Lemma~\ref{lemma:wedin}, we find $p_{1},\ldots,p_{n},q_{1},\ldots,q_{n} \in \Pro(H)$ such that \begin{itemize}
	\item[---\,] $\rk(p_{i}) = 1 = \rk(q_{i})$ for each $i \in \{ 1,\ldots,n\}$,
	\item[---\,] $p_{i} \perp p_{j}$ and $p_{i} \perp q_{j}$ and $q_{i} \perp q_{j}$ for any two distinct $i,j \in \{ 1,\ldots,n \}$,
	\item[---\,] $p = p_{1} + \ldots + p_{n}$ and $q = q_{1} + \ldots + q_{n}$.
\end{itemize}
	
Consider any $i \in \{ 1,\ldots,n\}$. Since \begin{displaymath}
	\Omega (p_{i}(H),q_{i}(H)) \, \stackrel{\ref{lemma:os}}{\leq} \, 2 \Vert p_{i} - q_{i} \Vert \, \stackrel{\ref{lemma:projection.distance.2}\ref{lemma:projection.distance.2.2}}{=} \, \sqrt{2} \Vert p_{i} - q_{i} \Vert_{2} ,
\end{displaymath} there exist $x_{i} \in \Sph(p_{i}(H))$ and $y_{i} \in \Sph(q_{i}(H))$ with $\Vert x_{i} - y_{i} \Vert \leq \sqrt{2}\Vert p_{i} - q_{i} \Vert_{2}$, in turn we find $u_{i} \in \U(p_{i}(H)+q_{i}(H))$ such that $u_{i}x_{i} = y_{i}$. Note that \begin{equation}\label{eq:two.norm.of.projections}
	\Vert p_{i}-u_{i}p_{i} \Vert_{2}^{2} \, = \, \Vert (p_{i}-u_{i}p_{i})x_{i} \Vert^{2} \, = \, \Vert x_{i}-y_{i} \Vert^{2} \, \leq \, 2\Vert p_{i} - q_{i} \Vert_{2}^{2} .
\end{equation} Moreover, we observe that $u_{i}p_{i}u_{i}^{\ast}y_{i} = u_{i}p_{i}x_{i} = u_{i}x_{i} = y_{i} = q_{i}y_{i}$. On the other hand, if $z \in p_{i}(H)+q_{i}(H)$ with $z \perp y_{i}$, then $u_{i}^{\ast}z \perp u_{i}^{\ast}y_{i} = x_{i}$ and thus $u_{i}p_{i}u_{i}^{\ast}z = 0 = q_{i}z$. Therefore, \begin{equation}\label{eq:conjugation}
	u_{i}p_{i}u_{i}^{\ast} \, = \, q_{i} .
\end{equation}
	
Now, as the linear subspaces $p_{1}(H)+q_{1}(H),\ldots, p_{n}(H)+q_{n}(H) \leq H$ are mutually orthogonal, there exists $u \in \U(H)$ such that $u\vert_{p_{i}(H)+q_{i}(H)} = u_{i}$ for each $i \in \{ 1,\ldots,n \}$. We conclude that \begin{align*}
	upu^{\ast} \, &= \, up_{1}u^{\ast} + \ldots + up_{n}u^{\ast} \, = \, up_{1}(up_{1})^{\ast} + \ldots + up_{n}(up_{n})^{\ast} \\
		& = \, u_{1}p_{1}(u_{1}p_{1})^{\ast} + \ldots + u_{n}p_{n}(u_{n}p_{n})^{\ast} \, = \, u_{1}p_{1}u_{1}^{\ast} + \ldots + u_{n}p_{n}u_{n}^{\ast} \\
		& \stackrel{\eqref{eq:conjugation}}{=} \, q_{1} + \ldots + q_{n} \, = \, q 
\end{align*} and, finally, \begin{align*}
	\Vert (1-u)p \Vert_{2}^{2} \, &= \, \left\lVert \sum\nolimits_{i=1}^{n} p_{i}-u_{i}p_{i} \right\rVert_{2}^{2} \, = \, \sum\nolimits_{i=1}^{n} \Vert p_{i}-u_{i}p_{i} \Vert_{2}^{2} \\
		& \stackrel{\eqref{eq:two.norm.of.projections}}{\leq} \, 2\sum\nolimits_{i=1}^{n} \Vert p_{i}-q_{i} \Vert_{2}^{2} \, \stackrel{\ref{remark:norm.decomposition}\ref{remark:norm.decomposition.2}}{=} \, 2 \Vert p-q \Vert_{2}^{2} .\qedhere
\end{align*} \end{proof}

\section*{Acknowledgments}

The authors would like to express their sincere gratitude towards Vladimir Pestov for bringing the work of Ara and Lledó~\cite{AraLledo} to their attention (see Remark~\ref{remark:dank.pestov}), and for numerous insightful comments on an earlier version of the present manuscript.


\end{document}